\documentclass[reqno,10pt]{amsart}

\usepackage{amsthm,amsmath,amssymb,amsfonts,bbm}
\usepackage{tikz}
\usepackage{float}
\usepackage{graphicx}
\usepackage{overpic}

\newtheorem{lemma}{Lemma}[section]
\newtheorem{proposition}{Proposition}[section]
\newtheorem{corollary}{Corollary}[section]
\newtheorem{remark}{Remark}[section]
\newtheorem{claim}{Claim}[section]
\newtheorem{theorem}{Theorem}[section]

\makeatletter
\@addtoreset{equation}{section}
\makeatother

\newcommand{\pr}{\mathbb{P}}

\newcommand{\esp}{\mathbb{E}}
\newcommand{\var}{ \mathbb{V}ar }

\newcommand{\N}{ \mathbb{N} }
\newcommand{\Z}{ \mathbb{Z}}
\newcommand{\R}{ \mathbb{R} }

\begin{document}

\title[Convergence of the random directed forest to the Brownian web]{A version of the random directed forest and its convergence to the Brownian web.}

\author{Glauco Valle and Leonel Zuazn\'abar.}

\address{
\newline
Glauco Valle
\newline
Universidade Federal do Rio de Janeiro, Instituto de Matem\'atica,
\newline  Caixa Postal 68530, 21945-970, Rio de Janeiro, Brasil.
\newline
e-mail: {\rm \texttt{glauco.valle@im.ufrj.br}}
\newline
\newline
Leonel Zuazn\'abar
\newline
Universidade de S\~ao Paulo,
\newline  R. do Mat\~ao, 1010 - Butant\~a, S\~ao Paulo - SP, CEP: 05508-090, Brasil.
\newline
e-mail: {\rm \texttt{lzuaznabar@ime.usp.br}}
}

\thanks{G. Valle was supported by CNPq grant 308006/2018-6 and FAPERJ grant E-26/202.636/2019. L. Zuazn\'abar was supported by CAPES}

\date{}

\maketitle

\begin{abstract}
Several authors have studied convergence in distribution to the Brownian web under diffusive scaling of systems of Markovian random walks. In a paper by R. Roy, K. Saha and A. Sarkar, convergence to the Brownian web is proved for a system of coalescing random paths - the Random Directed Forest- which are not Markovian. Paths in the Random Directed Forest do not cross each other before coalescence. Here we study a variation of the Random Directed Forest where paths can cross each other and prove convergence to the Brownian web. This provides an example of how the techniques to prove convergence to the Brownian web for systems allowing crossings can be applied to non-Markovian systems.
\end{abstract}

\section{Introduction.}

Several authors have studied convergence in distribution to the Brownian web, for instance  \cite{bmsv}, \cite{coletti2014convergence}, \cite{ffw}, \cite{fontes2004brownian} and \cite{fontes2015scaling}, \cite{roy2013random} to mention some works. The aim of most of these papers is the understanding of the universality class associated to the Brownian web. The Brownian web was formally introduced in \cite{fontes2004brownian}. You can find a review of the Brownian web and how it arises as the scaling limit of various one-dimensional models in \cite{schertzer2015brownian}.

In \cite{roy2013random} the authors study the Random Directed Forest, which is a system of coalescing space and time random paths on $\Z^2$ as we now describe. Suppose that the first coordinate of a point in $\Z^2$ represents space and the second one time. We start a space-time random path in each point of $\Z^2$. The path starting at $u\in \Z^2$ evolves as follows: every point in $\Z^2$ is open with some probability $p$ or closed with probability $1-p$ independently of the other points. We say that a point $v = (\tilde{x},\tilde{t})\in \Z^2$ is above $u = (x,t)$ if $\tilde{t}>t$. If the path is at space-time position $(x,t)\in\mathbb{Z}^2$, then it jumps to the nearest open point in the $L_1$ norm above $(x,t)$, if this nearest open point is unique. If it is not unique then the choice to decide where the path has to jump to is made uniformly over the nearest open points (see Figure \ref{fig:1}). Notice that two paths cannot cross each other and must coalesce when they meet each other. Futhermore, after a jump, it is possible that we have information about the "future" ahead of the position of the path; that is to say, maybe we know if some points above the current position of the path are open or closed. That is why we get a system of coalescing non-Markovian random paths. The random collection of linearly interpolated trajectories induced by the discrete random paths as above starting on every $u \in \Z^2$ is called the Random Directed Forest. 

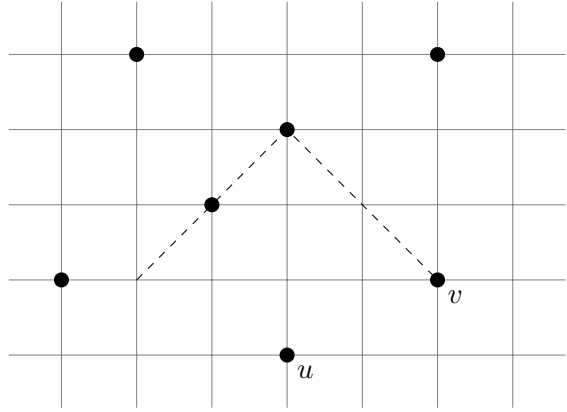
\begin{figure}
\label{fig:1}
\begin{tikzpicture}
\draw [step= 1cm,gray, very thin] (-0.7,-0.7) grid (6.7,4.7);
\fill[black] (3,0) circle (1mm) node[below right]{$u$}; \fill[black] (2,2) circle (1mm);\fill[black] (3,3) circle (1mm);\fill[black] (5,1) circle (1mm) node[below right]{$v$};\fill[black] (5,4) circle (1mm);\fill[black] (0,1) circle (1mm);\fill[black] (1,4) circle (1mm);
\draw[black,dashed] (1,1)--(3,3) -- (5,1);
\end{tikzpicture}
\caption{Open points in $\Z^2$ are marked by black dots. Notice that the closest open points above $u$ in the $L_1$ distance are those connected by the dashed line. Hence the path starting at $u$ moves to one of these points connected by the dashed line chosen uniformly among them; for instance it could be $v$.}
\end{figure}

R. Roy, K. Saha and A. Sarkar in \cite{roy2013random} proved that under diffusive scaling the Random Directed Forest converges in distribution to the Brownian web. Our initial aim was to consider a natural generalization of the Random Directed Forest that allows crossings before coalescence (as was studied in \cite{coletti2014convergence}) in the following way: a path does not necessarily jump to one of the closest open points in $L_1$ distance above the current position, but alternatively it jumps to a randomly chosen $L_1$ level set above the current position. 

In our version of the Random Directed Forest each path jumps to a randomly chosen $L_1$ level set above its current position but we use a different choice mechanism on the chosen level set. The jump should be made to the upmost open site on that chosen level set. Although we get a well defined system by imposing an uniform choice as in \cite{roy2013random}, we were not able to prove convergence to the Brownian web in this case. The problem here was to build a regeneration structure similar to that presented in \cite{roy2013random} which is the strategy to deal with the non-Markovianity.

Let us be more precise about the definition of our variation of the Random Directed Forest that allows crossings before coalescence. As before every point in $\Z^2$ is open with some probability $p$ or closed with probability $1-p$ independently of each other. Let $\{W_u:u\in\Z^2\}$ be an i.i.d. family of random variables supported on the set of positive integers. We will call the $k$-th level of $u=(u(1),u(2)) \in \Z^2$ the following set
\begin{align}\label{L}
    L(u,k):= \big\{ v=(v(1),v(2))\in\Z^2: v(2)>u(2) \text{ and } ||v-u||_1=k\big\}, 
\end{align}
where $||u||_1:=|u(1)|+|u(2)|$. The level set $L(u,k)$ is called open if it has at least one open point. Consider that the path moves to the highest open point in the $W_u$-th open level. If the path has two options to jump to then it makes an uniform choice. See Figure 2 as an example.

\begin{figure}\label{fig:2}
\begin{tikzpicture}
\draw [step= 1cm,gray, very thin] (-0.7,-0.7) grid (6.7,4.7);
\fill[black] (3,0) circle (1mm) node[below right]{$u$}; \fill[black] (2,2) circle (1mm);\fill[black] (3,3) circle (1mm);\fill[black] (5,1) circle (1mm);\fill[black] (5,4) circle (1mm);\fill[black] (0,1) circle (1mm) node [below left]{$v$};\fill[black] (1,4) circle (1mm);
\draw[black,dashed] (1,1)--(3,3) -- (5,1);
\draw[black,dashed] (0,1)--(3,4) -- (6,1);
\end{tikzpicture}
\caption{Notice that points connected by the dashed lines are the first and the second open levels of $u$. If $W_u=2$, for instance, the path starting in $u$ jumps to $v$.}
\end{figure}
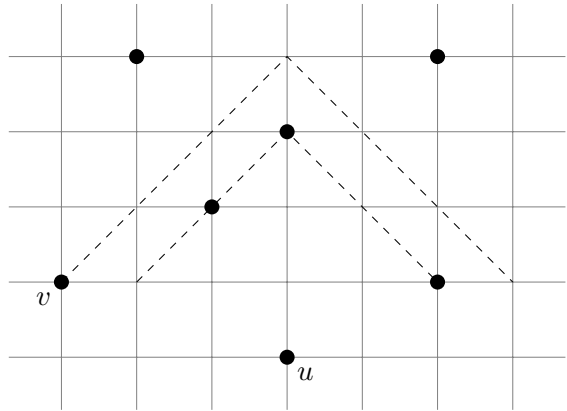

As the Directed Random Forest we now have a system of non-Markovian walks but in this case the paths can cross each other. Our goal is to prove convergence in distribution to the Brownian web under diffusive scaling for the closure of the system of linearly interpolated paths. See Theorem \ref{goal}. 

\medskip

In the next section, Section \ref{sec:main}, we are going to define formally the variation of the Random Directed Forest and state the convergence to the Brownian web. By the end of the same Section \ref{sec:main} we shall explain how the rest of the paper is divided in accordance with the steps that should be taken to prove the convergence result.

\bigskip

\section{The directed forest and the Brownian web.}
\label{sec:main}

Let us define formally the process described in the previous section. First let us fix some notation that will be used in the paper. We will denote by $\Z_+:=\{0,1,2,3,\dots\},\Z_{-}:=\{0,-1,-2,-3,\dots\}$ and $\N:=\{1,2,3,\dots\}$. Consider the following random variables:
\begin{enumerate}
\item [(i)] Let $K$ be a positive constant and  $(W_u)_{u\in\Z^2}$ be a family of i.i.d. random variables with support on $\N$ such that $\pr[W_u=1]>0$ and $\pr[W_u\leq K]=1$. Denote by $\pr_W$ the induced probability on $\N^{\Z^2}$.
\item [(ii)] Let $\{U_v:v\in\Z^2\}$ be a family of i.i.d. uniform random variables in $(0,1)$. Denote by $\pr_U$ the induced probability on $(0,1)^{\Z^2}$.
\end{enumerate}
We suppose that the two families above are independent of each other and thus have a joint distribution given by the product probability $\pr:=\pr_W\times\pr_U$ on the space $\N^{\Z^2}\times(0,1)^{\Z^2}$. In accordance with interacting particle systems terminology, points in $\mathbb{Z}^2$ will also be called sites and a configuration in $\N^{\Z^2}\times(0,1)^{\Z^2}$ will be called an environment for the system of random paths we define below. 

Fix some $p\in (0,1)$. We write $u=(u(1),u(2))$ for $u\in\Z^2$ and call open the sites in $V:=\{u\in\Z^2; U_u <p \}$ and closed those in $\Z^2\setminus V$. We will denote the index of the r-th open level of $u$ by $h(u,r)$, i.e.
\begin{align*}
  h(u,r):= \inf\Big\{k\geq 1: \sum_{j=1}^{k}\mathbbm{1}_{\{ L(u,j)\cap V\neq\emptyset\}}=r \Big\},
\end{align*}
where $L(u,j)$ is as defined in (\ref{L}).

\medskip

For $u\in\Z^2$ denote by $X[u]$ the unique (almost surely) point in $L(u,h(u,W_u))\cap V$ such that for every  $v \in L(u,h(u,W_u))\cap V$ either $X[u]$ is above $v$ or $U_{X[u]}>U_v$ and $X[u](2)=v(2)$.
Let us define the sequence $\{X_{m}^u: m\geq 0\}$  as $X_0^u = u$ and for $m\geq 1$, 
\begin{align*}
 X_{m}^u:=X[X_{m-1}^u]   \in \Z^2  \, .
\end{align*}  

 Define $\pi^u:[u(2),\infty]\rightarrow[-\infty,\infty]$ as $\pi^u(X_m^u(2)):=X_m^u(1)$ , linearly interpolated in the time interval $[X_m^u(2),X_{m+1}^u(2)]$ , and $\pi^u(\infty)=\infty$. We denote
\begin{align}
    \mathcal{X}:=\big\{ \pi^u : u\in \Z^2\big\}. 
\end{align}\label{$X$}
The system $\mathcal{X}$ is well defined for almost all realizations of the environment. This system is the variation of the Random Directed Forest which is the main object of study in this paper. From now on we call it the Generalized Random Directed Forest (GRDF).   

\medskip

We are interested in the diffusive rescaled GRDF. So let $\gamma > 0$ and $\sigma > 0$ be some fixed normalizing constants to be determined later, $u\in\Z^2$ and $n\in\N$. Let us define $\pi_{n}^u(t):=\frac{\pi^u(n^{2}\gamma t)}{n\sigma}$ for  $t\in[\frac{u(2)}{n^{2}\gamma},\infty), \pi_{n}^u(\infty)=\infty$ and 
\begin{align}\label{$X_n$}
\mathcal{X}_n:= \{ \pi_n^u : u\in \Z^2 \}. 
\end{align} 

\medskip

The system of coalescing paths $\mathcal{X}_n$ is the rescaled GRDF and our aim is to prove that its closure converges to the Brownian web as $n \rightarrow \infty$.

\medskip

Let us introduce the Brownian web. As in \cite{fontes2004brownian} take $(\bar{\R}^2,\rho)$ as the completion of $\R^2$ under the metric $\rho$ defined as
\begin{align*}
\rho\big((x_1,t_1),(x_2,t_2)\big):=\Big|\frac{\tanh (x_1)}{1+|t_1|}-\frac{\tanh(x_2)}{1+|t_2|}\Big|\vee\big|\tanh(t_1)-\tanh(t_2)\big|.
\end{align*}
We may consider $\bar{\R}^2$ as the image of $[-\infty,\infty]\times[-\infty,\infty]$ under the mapping
\begin{align*}
(x,t)\rightarrow \big(\Phi(x,t),\Psi(t)\big):=\Big(\frac{\tanh(x)}{1+|t|},\tanh(t)\Big).
\end{align*}
For $t_0\in[-\infty,\infty],$ let $C[t_0]$ be the set of functions from $[t_0,\infty]$ to $[-\infty,\infty]$ such that $\Phi\big(f(t),t\big)$ is continuous. Then define 
\begin{align*}
\Pi=\underset{t_0\in[-\infty,\infty]}{\bigcup}  C[t_0]\times\{t_0\}.
\end{align*} 
For $(f,t_0)$ in $\Pi$, let us denote by $\widehat{f}$ the function that extends $f$ to all $[-\infty,\infty]$ by setting it equal to $f(t_0)$ for $t\leq t_0$. Take
\begin{align*}
d\big((f_1,t_1),(f_2,t_2)\big)=\Big(\sup_{t\geq t_1\wedge t_2}|\Phi(\widehat{f_1}(t),t)-\Phi(\widehat{f_2}(t),t)|\Big)\vee|\Psi(t_1)-\Psi(t_2)|.
\end{align*}  
Let $\mathcal{H}$ be the set of compact subsets of $(\Pi,d)$ endowed with the Hausdorff metric $d_{\mathcal{H}}$:
\begin{align*}
d_{\mathcal{H}}(K_1,K_2):=\sup_{g_1\in K_1}\inf_{g_2\in K_2}d(g_1,g_2)\vee\sup_{g_2\in K_2}\inf_{g_1\in K_1}d(g_1,g_2), 
\end{align*}
for $K_1,K_2$ non-empty sets in $\mathcal{H}$. Let $\mathcal{F}_{\mathcal{H}}$ be the Borel $\sigma$-field induced by $(\mathcal{H},d_{\mathcal{H}})$.

\smallskip

The existence of the Brownian web as a random element of  $(\mathcal{H},\mathcal{F}_{\mathcal{H}})$ is the content of Theorem $2.1$ in \cite{fontes2004brownian} which we reproduce below:

\smallskip

\begin{theorem}\label{BW} 
There exists a $(\mathcal{H},\mathcal{F}_{\mathcal{H}})-$valued random variable $\mathcal{W}$, called the Brownian Web, whose distribution is uniquely determined by the following three properties:
\begin{enumerate}
\item [(i)] For any deterministic point $(x,t)$ in $\R^2$ there exists almost surely a unique path $\mathcal{W}_{x,t}$ starting from $(x,t).$
\item [(ii)] For any deterministic $n$, and $(x_1,t_1),\dots,(x_n,t_n)$ the joint distribution of $\mathcal{W}_{x_1,t_1},\dots,\mathcal{W}_{x_n,t_n}$ is that of coalescing Brownian motions.
\item[(iii)] For any deterministic, dense, countable subset $\mathcal{D}$ of $\R^2$, almost surely, $\mathcal{W}$ is the closure in $(\mathcal{H},\mathcal{F}_{\mathcal{H}})$
of $\{\mathcal{W}_{x,t}: (x,t)\in\mathcal{D}\}.$
\end{enumerate} 
\end{theorem} 

\medskip

 The next result is a criterion of convergence to the Brownian web which is a variation of Theorem $2.2$ proved in \cite{fontes2004brownian} that can be found as Theorem 1.4 in \cite{newman2005convergence}. These theorems (Theorem 2.2 in \cite{fontes2004brownian} and Theorem 1.4 in \cite{newman2005convergence}) have been the main tools to prove convergence to the Brownian web for several coalescing systems of random walks.

\smallskip

\begin{theorem}\label{convergence theo}
Let $\{\mathcal{Y}_n\}_{n\geq 1}$ be a sequence of $(\mathcal{H},\mathcal{F}_{\mathcal{H}})$-valued random variables We have that $\{\mathcal{Y}_n\}_{n\geq 1}$ converges to the Brownian web if the following conditions are satisfied:
\begin{enumerate}
\item [$(I)$] There exists some deterministic countable dense subset $D \subset \R^2$ and $\theta_n ^y\in\mathcal{Y}_n$ for any $y\in D$ satisfying: for any deterministic $y_1,\dots,y_m\in D, \theta_n^{y_1},\dots,\theta_n^{y_m}$ converge in distribution as $n\rightarrow\infty$ to coalescing Brownian motions starting in $y_1,\dots,y_m.$
\item[$(B)$] For every $\beta>0$, 
$$
\limsup_n\sup_{t>\beta}\sup_{t_0,a\in\R}\pr\big[|\eta_{\mathcal{Y}_n}(t_0,t,a-\epsilon,a+\epsilon)|> 1\big]\rightarrow 0, \text{ as }\epsilon\rightarrow 0^+ \, ,
$$ 
where $\eta_{\mathcal{Y}_n}(t_0,t,a,b)$ is the set of points in $\R\times\{t_0+t\}$ that are touched by paths which also touch some point in $[a,b]\times\{t_0\}$.
\item[$(E)$] For some $(\mathcal{H},\mathcal{F}_{\mathcal{H}})$-valued random variables $Y$ and $t>0$ take $Y^{t^-}$ as the subset of paths in $Y$ which start before or at time t. If $Z_{t_0}$ is the subsequential limit of $\{\mathcal{Y}_n^{t_0^-}\}_{n\geq 1}$ for any $t_0$ in $\R$, then for all $t,a,b$ in $\R$ with $t>0$ and $a<b$ we get
\begin{align*}
    \esp\big[|\hat \eta_{Z_{t_0}}(t_0,t,a,b)|\big]\leq \esp\big[|\hat \eta_{\mathcal{W}}(t_0,t,a,b)|\big] = \frac{b-a}{\sqrt{\pi t}} \, .
\end{align*}
where $\hat \eta_{\mathcal{Y}_n}(t_0,t,a,b)$ is the set of points in $(a,b) \times\{t_0+t\}$ that are touched by paths which also touch $\mathbb{R} \times\{t_0\}$,
\item[$(T)$]  Let $\Lambda_{L,T}:=[-L,L]\times[-T,T]\subset\R^2$ and for $(x_0,t_0)\in\R^2$ and $\rho,t>0,R(x_0,t_0;\rho,t):=[x_0-\rho,x_0+\rho]\times[t_0,t_0+t]\subset\R^2$. For $\mathbb{K} \in\mathcal{H}$ define $A_{\mathbb{K}}(x_0,t_0;\rho,t)$ to be the event that $\mathbb{K}$ contains a path touching  both $R(x_0,t_0;\rho,t)$ and the right or the left boundary  of $R(x_0,t_0;20\rho,4t)$. Then for every $\rho,L,T\in(0,\infty)$ 
\begin{align*}
    \frac{1}{t}\limsup_{n \rightarrow \infty}\sup_{(x_0,t_0)\in\Lambda_{L,T}}\pr\Big[A_{\mathcal{Y}_n}(x_0,t_0;\rho,t)\Big]\rightarrow0 \text{ as } t\rightarrow 0^+.
\end{align*}
\end{enumerate}
\end{theorem}

\bigskip

\begin{remark}
We should point out that the statement of condition (T) in \cite{newman2005convergence} uses the rectangle $R(x_0,t_0;17\rho,2t)$ instead of $R(x_0,t_0;20\rho,4t)$. This change of the constants helps with the proof of condition (T) in our case and it is simple to verify from \cite{newman2005convergence} that we are able to make this change without prejudice to the result.
\end{remark}

\bigskip

The main result in this paper is the convergence of the GRDF to the Brownian web under diffusive scaling. Before stating it, we need to check that $\overline{\mathcal{X}}$, the closure in $(\mathcal{H},d_{\mathcal{H}})$ of $\mathcal{X}$, is a well-defined random element of $(\mathcal{H},d_{\mathcal{H}})$ which is the content of the next result:
 
\begin{proposition}\label{well posedness}
We have that $\overline{\mathcal{X}}$ and $\overline{\mathcal{X}}_n$, $n\ge 1$, are almost surely compact subsets of $(\Pi,d)$. 
\end{proposition}

\smallskip

The proof of \ref{well posedness} is a direct consequence of Lemma \ref{paths in [a,b]} following the proof of Lemma 1.1 in \cite{newman2005convergence}.

\medskip

\begin{theorem}\label{goal} 
There exist positive constants $\gamma$ and $\sigma$ such that $\overline{\mathcal{X}}_n$, the closure of $\mathcal{X}_n$ in $(\Pi,d)$, converges in distribution to the Brownian web as $n$ goes to infinity.
\end{theorem}

\medskip

\begin{remark}
By choosing the scaling parameter as $n' = \lfloor n \sqrt{\gamma} \rfloor$, the reader can check that we can always choose $\gamma = 1$ in Theorem \ref{goal}. Our choice for the statement stems from the fact that the proof of Theorem \ref{goal} is based on a renewal structure introduced in Section \ref{sec:renewal} and $\gamma$ is a time correction. Indeed $\gamma$ will be chosen simply as the expectation of the increments of the renewal times and $\sigma$ as the standard deviation of the displacements of the process on these increments of time. The computation of these parameters (as functions of $p$ and the common distribution of $W_u$, $u\in\Z^2$) is not straightforward and depends on the renewal structure. Even if we consider the scaling correction $n' = \lfloor n \sqrt{\gamma} \rfloor$ and get a single parameter $\sigma/\sqrt{\gamma}$ which does not depend on renewal structure, the computation is not simple. Since the explicit values are unnecessary in our proofs, we do not discuss their computations any further.
\end{remark}

\medskip

The rest of the paper is devoted to the proof of Theorem \ref{goal} and we end this section explaining how it is divided. In Section \ref{sec:renewal} we introduce regeneration times where the random paths in the GRDF have no information about the future. This yields a Markovian structure we can rely on. In Section \ref{sec:skorohod} we describe what we call the Skorohod scheme for the GRDF which was introduced in \cite{coletti2009scaling} to estimate the tail probability of coalescing times. In Section \ref{coalescing time} we prove a central estimate related to the tail probability of the coalescing time of two random paths of the GRDF. The results from both Sections \ref{sec:renewal} and \ref{coalescing time} will be essential for the rest of the paper. In Sections \ref{sec:I}, \ref{sec:B}, \ref{sec:E} and \ref{sec:T} we prove respectively conditions $I$, $B$, $E$ and $T$.

\bigskip

\section{Renewal Times.}\label{sec:renewal}

In this section we prove the existence of stopping times where the random paths in the GRDF have no information about the future. The idea of using stopping times came from \cite{roy2013random} and is fundamental since we use it to build a renewal structure for the system. However in order to be able to define these renewal times analogously to them we had to consider that a GRDF path always jumps to the upmost open site in a chosen open level above its current position (For them, the choice is uniform on the first open level as we had already pointed out). The reason is that, if we choose the open level as in the GRDF, then paths can enter regions in the "future" which have already been observed (called explored region in \cite{roy2013random}); something that does not occur in the Random Directed Forest and is necessary in the proof of \cite{roy2013random}, see Remark \ref{explicacao}. Therefore we need a different approach to obtain a similar renewal structure to that of \cite{roy2013random}. To deal with paths that enter into the explored region, our approach is to search for open sites above the current position of the path that are also above the explored region. To guarantee that the path will jump to these open sites we need to impose jumps to the upmost open site on the chosen level. This is a way to force the path to leave the explored region which is not possible with an uniform choice among the open sites on the chosen level.
\medskip 

Recall the definition of the random paths in the GRDF, $(X_m^u)_{m\ge 1}$ , $u \in \mathbb{Z}^2$, from the previous section. As in \cite{roy2013random} let us denote by $\Delta_m(u)$, for $m\in\Z_+$ and $u\in\Z^2$, the set of sites above $X_m^u$ whose configuration is already known; i.e. $\Delta_0(u)=\emptyset$ and for $m\geq 1$,
\begin{align}\label{Delta}
\Delta_m(u):=&\Big[\Delta_{m-1}(u)\cup\big\{v\in\Z^2:||v-X_{m-1}^u||_1\leq||X_m^u-X_{m-1}^u||_1\big\}\Big]\nonumber\\
&\cap\big\{v\in\Z^2: v(2)>X_m^u(2)\big\}.
\end{align}
In \cite{roy2013random}, the authors call the set $\Delta_m(u)$ the "explored region" of the path starting at site $u$ after its m-th jump. These regions describe which part of the environment above the current position of the path is already known.

\begin{lemma}\label{tpr} 
	Let $u$ be a point in $\Z^2$. Then
	\begin{enumerate}
		\item [(i)] There exists a random variable $\tilde{\tau}(u)\ge 1$ such that $\Delta_{\tilde{\tau}(u)}(u)=\emptyset$.
		\item [(ii)] Taking $T(u):=X_{\tilde{\tau}(u)}(u)(2)$, where $\tilde{\tau}(u)$ is taken as in $(i)$, we have that $\esp\big[\big(T(u)-u(2)\big)^k\big]<\infty$, for all $k\geq 1$.
		\item [(iii)] There exists a random variable $Z(u)$ such that $\esp[Z(u)^k]<\infty$ for all $k\geq 1$ and 
		$$
		\sup_{u(2)\leq t\leq T(u)}|\pi^u(t)-\pi^u(0)|\leq Z(u)\, ,
		$$
		where $T(u)$ is taken as in $(ii)$.
		\item[(iv)] The distributions of the random variables $T(u)-u(2)$ and $Z(u)$, previously defined, do not depend on $u(1)$. Also we have that $T(u)-u(2)\leq Z(u)$.
		\item [(v)] For $k\geq u(2)$ let us define $\mathcal{F}^u_k:=\sigma\big(\{X_v,W_v:v\in\Z^2, u(2)\leq v(2)\leq k \}\big)$. Then $T(u)$ is a stopping time for the filtration $\{\mathcal{F}^u_k: k\geq u(2)\}$.
	\end{enumerate}
\end{lemma}

\begin{proof}
We start with some definitions. For $v$ in $\Z^2$ and $l$ in $\N$ put  
$$
    V(v,l):= \big\{ \big(v(1),v(2)+1\big),\dots,\big(v(1),v(2)+l\big)\big\},
$$
which are the first $l$ points immediately above $v$. Recall that the constant $K$  from the definition of the $W_v's$ is fixed at Section \ref{sec:main} and define the following event
$$
C(v) :=\{\text{all } w \text{ in } V(v,K) \text{ are open} \}.
$$
Notice that on $C(v)$ the path of the GRDF that starts in $v$, necessarily jumps (almost surely) to some $w \in V(v,K)$, i.e. $X^v_1 \in V(v,K)$. This is true because these sites are in the top of the first $K$ open levels of $v$ and $W_v\leq K$ (almost surely). Note that here we are using the hypothesis about the finite support of the law of the $W_v$'s and the assumption that if a level set contains more than one open site, then we choose the highest. 

Let us define the events 
$$ D(v):=\big\{W_w=1 \text{ for all } w \text{ in } V(v,K-1)\big\}
$$
and 
\begin{align}\label{E(v)}
    E(v):=C(v)\cap D(v).
\end{align}

On $D(v)$ each path of the GRDF that starts in $V(v,K-1)$ will jump to the nearest open point (in $L_1$ distance) above it. Hence on $E(v)$ one of the following cases should occur:
\begin{enumerate}
    \item [(a)] $X^v$ jumps to $\big(v(1),v(2)+K\big)$,
    \item [(b)] $X^v$ jumps to $\big(v(1),v(2)+l\big)$, for some $l=1,\dots,K-1$, and after that it moves to $\big(v(1),v(2)+K\big)$ passing consecutively by each one of the points $\big(v(1),v(2)+k\big)$, $k=l+1,...,K-1$.
\end{enumerate}
Therefore on $E(v)$ the path of the GRDF starting at $v$ hits $\big(v(1),v(2)+K\big)$ without any information about the environment above $v(2)+K$. At such time we have that a renewal occurs.

The proof is based on an argument to show that $E(X^u_m)$ occurs infinitely often. This argument also allow us to control the maximum displacement of $X^u$ between the renewal times. For $v\in\Z^2$ put
\begin{align}\label{H(u)}
H(v):=\inf\Big\{n\geq 1: \sum_{j=1}^{n}\mathbbm{1}_{\{(v(1),v(2)+j)\text{ is open}\}}=K\Big\} .
\end{align}
Take $\xi_1:= H(u)$ and let $\tilde{n}_1$ be the first index of a jump of $X^u$ from where we need information about the configuration above time $u(2) + \xi_1$ to decide where $X^u$ jumps next. This index is defined as 
\begin{align*}
    \tilde{n}_1 :=\inf\Big\{m\geq 1: \, &X^u_m(2)=\xi_1+u(2) \text{ or }
    \\&\sum_{i=1}^{u(2)+\xi_1-X^u_m(2)}\mathbbm{1}_{\{L(X^u_m,i) \text{ is open}\}}<W_{X^u_m}\Big\}.
\end{align*}
To help the understanding of the notation see Figure \ref{fig 1}.

\begin{figure}
\label{fig 1}
\begin{overpic}[scale = .9]
{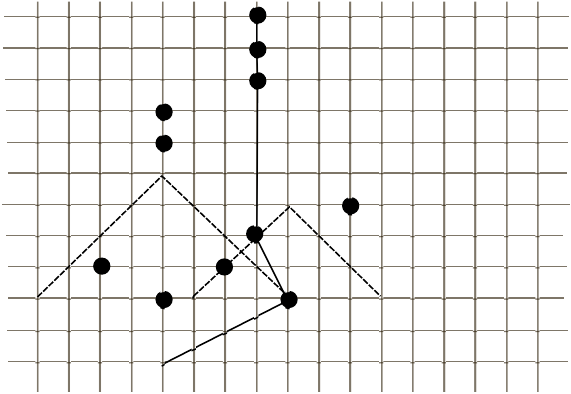}

\put(-6,5) {\tiny{$u(2)$}} 

\put(-13,49){\tiny{{$u(2)+\xi_1$}}}

\put(-32,65){\tiny{$T(u)=u(2)+\xi_1+\xi_2$}}

\put(29,3) {\tiny{$u$}} 

\put(46,25) {\tiny{$w$}}

\put(46,47) {\tiny{$w'$}}

\put(40,18.5) {\tiny{$w''$}}

\put(53,15) {\tiny{$v$}} 
\end{overpic}
\caption{ In this picture we assume that $K=3$, $W_u=3$, $W_v=1$ and $W_w= 3$. The blacks balls represent open sites.  Notice that $\xi_1=8$, $\tilde{n}_1 =2$, $X^u_{\tilde{n}_1}=w$, $w'= (X^u_{\tilde{n}_1}(1),u(2)+\xi_1)$, $\xi_2 = 3$ and $\tilde{n}_2 = 4$. Moreover the event $E(u)$ does not happen, but $E_1(u)$ does. Here while at site $v$, if $w$ was closed, then $X^u$ would jump to $w''$ and enter $\Delta_1(u)$.} 
\end{figure}
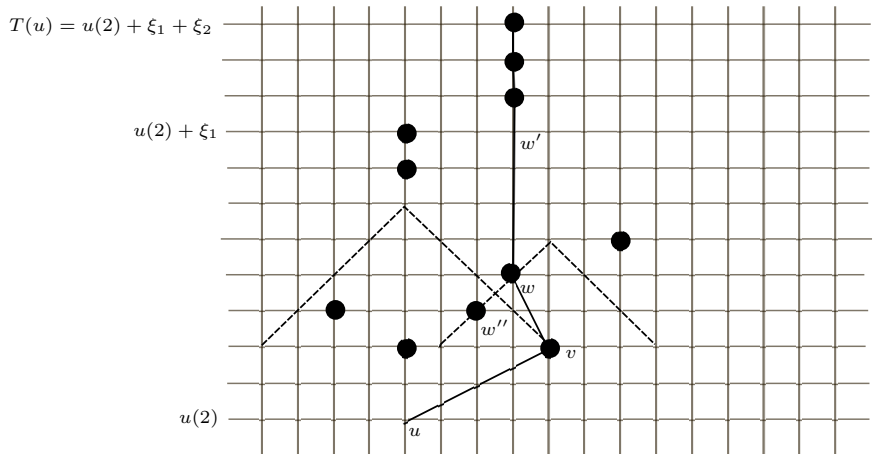

Now we use induction.

Having defined $\{\xi_1,\dots,\xi_j\}$ and $\{\tilde{n}_1,\dots,\tilde{n}_j\}$, let  us define $\xi_{j+1}$ and $t_{j+1}$ in the following way: denote by
	$$
	\eta_j:=u(2)+\sum_{i=1}^{j}\xi_i\, ,
	$$
	then take  $\xi_{j+1}$ and $\tilde{n}_{j+1}$ as follows,
	\begin{equation*}
	\xi_{j+1}:= H\big((X^u_{\tilde{n}_j}(1),\eta_j)\big)\, , \eta_{j+1}:=\xi_{j+1}+\eta_j\, ,
	\end{equation*}
	and $\tilde{n}_{j+1}$ as
	$$
		\tilde{n}_{j+1}:=\inf\Big\{m\geq \tilde{n}_j + 1: X_m(u)(2)=\eta_{j+1}\text{ or }
		W_{X_m(u)}>\sum_{i=1}^{\eta_{j+1}-X_m(u)(2)}\mathbbm{1}_{\{L(X_{m}(u),i) \text{ is open}\}}\Big\}.
		$$

From the sequences $\{\tilde{n}_j:j\geq 1\}$ and $\{\xi_j:j\geq 1\}$ define the sequence of events $\{E_j(u): j\geq 0\}$ as $E_0(u):=E(u)$ and for $j \geq 1$, 
$$
	E_j (u): = E\big( (X^u_{\tilde{n}_j}(1),\eta_j) \big)
$$
To help the understanding of the notation see again Figure \ref{fig 1}. As with $E(u)$, on $E_j(u)$ we claim that one of the following cases occurs:
\begin{enumerate}
    \item [(a')] $X^u$ jumps from $X^u_{\tilde{n}_j}$ to $\big(X^u_{\tilde{n}_j}(1),\eta_j + K\big)$, 
    \item [(b')] $X^u$ jumps from $X^u_{\tilde{n}_j}$ to $\big(X^u_{\tilde{n}_j}(1),\eta_j + l\big)$, for some $l=1,\dots,K$ and after that it moves to $\big(X^u_{\tilde{n}_j}(1),\eta_j +K\big)$ passing consecutively by each one of the points
$\big(X^j_{\tilde{n}_j}(1),\eta_j + k \big)$ for $k=l+1,...,K-1$.
\end{enumerate}
Let us see why the above claim is true. On $E_j(u)$ the sites
\begin{align}\label{X_{t_n+1}}
\big(X^u_{\tilde{n}_j}(1),\eta_j + 1\big),\dots,\big(X^u_{\tilde{n}_j}(1),\eta_j + K \big)    
\end{align}
are open and they are on the top of the first $K$ open levels of $\big(X^u_{\tilde{n}_j},\eta_j\big)$. By definition of $\tilde{n}_j$, $X^u$ jumps from $X^u_{\tilde{n}_j}$ to some site in these levels. Since the model gives preference to the highest open point, $X^u_{1+\tilde{n}_{j}}$ belongs to the set of points in (\ref{X_{t_n+1}}). If  $X^u_{1+\tilde{n}_{j}} \neq \big(X^u_{\tilde{n}_j}(1),\eta_j +K\big)$, by definition of the event $E_j(u)$, $X^u$ will move consecutively to the site immediately above it up to $\big(X^u_{\tilde{n}_j}(1),\eta_j +K\big)$.
\medskip

 Note that $\{E_j(u): j\geq 0\}$ are independent events with the same probability of success, so 
$$
    M:=\inf\{j \geq 0: E_j(u) \text{ occurs}\} + 1
$$
is a geometric random variable. Take 
$$
\tilde{\tau}(u) := \inf \Big\{ m : X^u_m (2) = u(2) + \sum_{i=1}^M \xi_i \Big\} \in [\tilde{n}_M , \tilde{n}_M + K] \, ,
$$ 
we have $\Delta_{\tilde{\tau}(u)}(u)=\emptyset$ and from this $(i)$ follows. To get $(ii)$ note that 
$$
 T(u):=X^u_{\tilde{\tau}(u)}(2)= u(2)+\sum_{i=1}^{M}\xi_i.
$$
Using Lemma \ref{sum moments} in the appendix we have $(ii)$. 

To get $(iii)$ note that $\tilde{n}_1 \leq \xi_1$ and each increment of $X^u_m$ up to $m=\tilde{n}_1$ is bounded above by $\xi_1$, so we have 
$$
    \sup_{0\leq s\leq X^u_{\tilde{n}_1}(2)}|\pi^u(s)-\pi^u(0)|\leq \xi_1^ 2.
$$
Similarly, between $\tilde{n}_j$ and $\tilde{n}_{j+1}$ the process $X^u$ does not make a jump bigger than $\sum_{i=1}^{j+1}\xi_i$ and the number of theses jumps is smaller than $\sum_{i=1}^{j+1}\xi_i$, hence
$$
    \sup_{X^u_{\tilde{n}_j}(u)\leq s\leq X^u_{\tilde{n}_{j+1}}(2)}|\pi^u(s)-\pi^u(0)|\leq \Big(\sum_{i=1}^{j+1}\xi_i\Big)^ 2.
$$
Then
\begin{align*}
\sup_{0\leq t\leq T(u)}|\pi^{u}(t)-\pi^u(0)|&\leq\sum_{k=1}^{M}[\xi_1+\dots+\xi_k]^2\leq M\big(\sum_{i=1}^{M}\xi_i\big)^2\leq\Big(\sum_{i=1}^M\xi_i \Big)^3:=Z(u).    
\end{align*}
Using Lemma \ref{sum moments} we have that $\esp\big[\big(Z(u)\big)^k\big]<\infty$ for all $k\geq 1$.

\medskip

It is simple to check that item $(iv)$ follows from the stationarity of the environment and item $(v)$ follows from the construction of the random variables $T(u)$ and $Z(u)$.
\end{proof}

\medskip

\begin{remark} \label{explicacao}
In Figure 3 we can see that it would be possible for a path to jump from its current position to a site inside an explored region. Consider the jump mechanism defined as follows: the path chooses a L1 open level above its current position and, in that level, it uniformly picks an open site. We can see, again from Figure 3, 
that it continues to be possible for a path to jump from its current position to a site inside an explored region (in Figure 3, once at site $v$ the path could choose to jump to $w$). Under both choice mechanisms the approach of Roy, Saha and Sarkar \cite{roy2013random} would not work, because their idea is based on the fact that the points above and with the same first coordinate of the current position are always in an unexplored region. The reader can check that the proof of Lemma \ref{tpr} does not work for the uniform choice jump mechanism, since we would not be able to guarantee the occurrence of jumps to sites immediately above the current position.
\end{remark}

\medskip

We can replicate recursively the construction made in proof of Lemma \ref{tpr} to get: 

\begin{corollary}\label{renewal time corollary one path}
Fix $u \in \Z^2$. Then there exist sequences of random variables $\{T_j(u):j\geq 1\}$, $\{Z_j(u): j\geq 1\}$ and $\{\tilde{\tau}_j(u):j\geq 1\}$ such that
\begin{enumerate}

\item[(i)] $\Delta_{\tilde{\tau}_j(u)}(u) = \emptyset$ for every $j\ge 1$.  

\item[ (ii)] $T_j(u) = X^u_{\tilde{\tau}_j(u)}(2)$, i.e. $\pi^{u}(T_j(u)) = X^u_{\tilde{\tau}_j(u)}(1)$, for every $j\geq 1$. 

\item[(iii)] For all $k$, $j\geq 1$ we have that $\esp[Z_j(u)^k]<\infty$ and taking $T_0(u)=u(2)$ we get that 
$$
\max \left\{ T_j(u)-T_{j-1}(u) \, , \, \sup_{T_{j-1}(u)\leq t\leq T_j(u)}|\pi^{u}(t)-\pi^{u}(T_{j-1}(u))| \right\} \leq Z_j(u) 
$$
for every $j\ge 1$. In particular for all $j$, $k\geq 1$ we have that $\esp[T_j(u)^k]<\infty$.

\item[(iv)] Put $T_0 = u(2)$ then $\{T_j(u) - T_{j-1}(u):j\geq 1\}$ and $\{Z_j(u): j\geq 1\}$ are sequences of i.i.d. random variables whose distributions do not depend on $u$. 

\item [(iv)] $\{T_j(u):j\geq 1\}$ are stopping times for the filtration $\{\mathcal{F}^u_n: n \geq u(2)\}$ defined in $(v)$ of Lemma \ref{tpr}.

\end{enumerate}
\end{corollary}

\medskip

The main result of this section is the following:

\begin{proposition}\label{renewal time proposition} 
 Fix $m\ge 1$ and $u_1,\dots,u_m \in \Z^2$ distinct sites at the same time level, i.e with equal second component. Then there exist random variables $T(u_1,\dots,u_m)$, $Z(u_1,\dots,u_m)$ and $\tau(u_i)\ge 1$ for $i=1,\dots,m$ such that 
 \begin{enumerate}
 \item[(i)] $\Delta_{\tau(u_i)}(u_i) = \emptyset$, for all $i=1,\dots,m$.
 
 \item[(ii)] $T(u_1,\dots,u_m) = X_{\tau(u_i)}^{u_i}(2)$, i.e. $\pi^{u_i}\big(T(u_1,\dots,u_m)\big)= X_{\tau(u_i)}^{u_i}(1)$, for all $i=1,\dots,m$. The distribution of $T(u_1,\dots,u_m) - u_1(2)$ does not depend on $u_1,\dots,u_m$. 
 
 \item[(iii)]For all $i=1,\dots,m$ we have 
 $$
 \max \left\{ T(u_1,\dots,u_m) - u_1(2) \, , \, \sup_{u_1(2)\le t \le T(u_1,\dots,u_m)} |\pi^{u_i}(t) - u_i(1)| \right\} \leq Z(u_1,\dots,u_m)
 $$
 and the distribution of $Z(u_1,\dots,u_m)$ does not depend on  $u_1,\dots,u_m$. Also for all $k\geq 1$ we have that $\esp\big[Z(u_1,\dots,u_m)^k\big]<\infty$ which also implies that $\esp\big[T(u_1,\dots,u_m)^k\big]<\infty$.
 \end{enumerate}
 \end{proposition}

\begin{proof}

Let us prove the proposition for $m=2$, the reader can check that the proof can be adapted in a straightforward way to the case $m>2$. The idea of the proof is very similar to the one used in Lemma \ref{tpr}. Fix $u_1, u_2 \in \Z^2$ with $u_1(2)=u_2(2)$ and let $\{T_j:j\geq 1\}$ be the renewal times for $X^{u_1}$ given by Corollary \ref{renewal time corollary one path}. Define $s_j$ as the first index $n$ such that $X^{u_2}_n$ needs information about the environment above $T_j-K$ before it jumps again:
\begin{align}\label{s_n}
s_j:=\inf\Big\{n\geq 0: &X_n^{u_2}(2)=T_j - K\text{ or }
\\& W_{X_n^{u_2}}>\sum_{i=1}^{T_j-K-X_n^{u_2}(2) }\mathbbm{1}_{\{L(X_{n}^{u_2},i) \text{ is open}\}}\Big\}\nonumber.
\end{align}
Also define the events 
$$
 \widehat{E}_j(u_2):=E\Big( (X_{s_j}^{u_2}(1),T_j-K) \Big) \, ,
$$
where the event $E(v)$, for $v \in \Z^2$, was defined in (\ref{E(v)}). 

\medskip

Since the only sites of the environment within time interval $[T_j-K ,T_j]$ that $X^{u_1}$ observes are those with first component equal to $\pi^{u_1}(T_j)$, the event $\widehat{E}_j(u_2)$ is independent of $\pi^{u_1}$ if $X_{s_j}^{u_2}(1) \neq \pi^{u_1}(T_j)$. If $X_{s_j}^{u_2}(1) = \pi^{u_1}(T_j)$ then $X^{u_1}$ and $X^{u_2}$ will have coalesced before or at time $T_j$, in this case we replace $\widehat{E}_j(u_2)$ by an independent event that has the same probability, which we will continue to denote by $\widehat{E}_n(u_2)$. This observation is to obtain a sequence of independent events, $(\widehat{E}_n(u_2))_{n\ge 1}$, that have the same probability and are also independent of the path $\pi^{u_1}$. Then the random variable 
$$
    G:=\inf\{j \geq 1: \widehat{E}_j(u_2) \text{ happens} \}
$$
is a geometric random variable. Take $\{\tilde{\tau}_j(u_1)\}_{n\geq 1}$ as in Corollary \ref{renewal time corollary one path} for $u_1$ and $\{s_j\}_{j\geq 1}$ as defined in (\ref{s_n}). Define $\tau(u_1):=\tilde{\tau}_{G}(u_1)$ and
$$
\tau(u_2):= \inf \Big\{ m : X^{u_2}_m(2) = T_G \Big\} \in [ s_G , s_{G} + K] \, .
$$ 
We have that
$$
    \Delta_{\tau(u_1)}(u_1)=\Delta_{\tau(u_2)}(u_2)=\emptyset.
$$
We also get that
$$
T(u_1,u_2):= T_{G}(u_1)=X_{\tau(u_1)}^{u_1}(2)=X_{\tau(u_2)}^{u_2}(2)
$$ 
is a common renewal time for the paths $\pi^{u_1}$ and $\pi^{u_2}$. 

Hence we have (i) and (ii) in the statement. The other claims in the statement follows from analogous arguments as those used in the proof of Lemma \ref{tpr}.
\end{proof}

\medskip

\begin{remark} (1) The times constructed in Proposition \ref{renewal time proposition}, and also referred to later at Corollary \ref{renewal time corollary}, are not unique and its specific definition are not relevant. However for the sake of simplifying the exposition, we will be always referring to those times constructed as in the proof of Proposition \ref{renewal time proposition}.
(2) Note that the parameter of the geometric random variable in the construction of the renewal time for the paths stating at $(0,0)$ and $(m,0)$, $m\in\Z,m\neq 0$, does not depend on $m$. 
This remark will be useful for Lemma \ref{crossing lower bound}.
\end{remark}

\medskip

Fix sites $u_1,\dots,u_m$ in $\Z^2$. To simplify suppose that $u_1$ is at the same time level or above $u_i$ for every $i=2,\dots m$. Wait until each process $X^{u_2},\dots,X^{u_m}$ needs information about the environment above $u_1(2)$, then we can proceed as in the proof of Proposition \ref{renewal time proposition} to obtain a similar renewal structure for the processes $X^{u_1},\dots,X^{u_m}$ even if they don't start necessarily at the same time level. So we have the following result:

\medskip

\begin{corollary}\label{renewal time corollary}
Fix $m\ge 1$ and $u_1,\dots,u_m \in \Z^2$ distinct points such that $u_1$ is at the same time level or above $u_i$ for every $i=2,\dots m$. Then there exist sequences of random variables $\{T_j(u_1,\dots,u_m)\}_{j\geq 1}$, $\{Z_j(u_1,\dots,u_m)\}_{j\geq 1}$ and $\{\tau_j(u_i)\}_{j\geq 1}$ for $i=1,\dots m$ such that,

\begin{enumerate}

\item[(i)] $\Delta_{\tau_j(u_i)}(u_i) = \emptyset$ for every $i=1,\dots,m$ and $j\geq 1$.  

\item[ (ii)] $X_{\tau_j(u_1)}^{u_1}(2) = X_{\tau_j(u_i)}^{u_i}(2)=T_j(u_1,\dots,u_m)$, i.e. $\pi^{u_i}(T_j(u_1,\dots,u_m)) = X_{\tau_j(u_i)}^{u_i}(1)$, for every $i=1,\dots,m$ and $j\geq 1$. 

\item[(iii)] Taking $T_0(u_1,\dots,u_m)=u_1(2)$ we have that
\begin{eqnarray*}
\lefteqn{\max \Big\{ \big| T_j(u_1,\dots,u_m) - T_{j-1}(u_1,\dots,u_m) \big| \, , } \\
& & \, \sup_{T_{j-1}(u_1,\dots,u_m) \leq t \leq T_j(u_1,\dots,u_m)}|\pi^{u_i}(t)-\pi^{u_i}(T_{j-1}(u_1,\dots,u_m))|
\Big\} \leq Z_j(u_1,\dots,u_m) \, ,
\end{eqnarray*}
for all $i=1,\dots,m$ and $j\geq 1$. For all $k,j\geq 1$ we have that 
$\esp\big[\big(Z_j(u_1,\dots,u_m)\big)^k\big]<\infty$ which also implies that
$\esp\big[\big(T_j(u_1,\dots,u_m)\big)^k\big]<\infty$.

\item[(iv)] $\{T_j(u_1,...,u_m) - T_{j-1}(u_1,...,u_m):j\geq 1\}$ and $\{Z_j(u_1,...,u_m): j\geq 1\}$ are sequences of i.i.d. random variables whose distributions do not depend on $u_1$, ... ,$u_m$. 

\item [(v)] $\{T_j(u_1,\dots,u_m):j\geq 1\}$ are stopping times for the filtration $\{\mathcal{F}^{u_1}_n: n \geq u_1(2)\}$ defined in $(v)$ of Lemma \ref{tpr}.
\end{enumerate}
\end{corollary}

\bigskip

\section{Skorohod's Scheme for the GRDF model.}\label{sec:skorohod}

In this section we describe the Skorohod scheme for the GRDF which was introduced in \cite{coletti2009scaling} to estimate the tail probability of coalescing times. We also rely on an adaptation of this method introduced in \cite{coletti2014convergence} to deal with systems allowing crossings. Let us first recall Skorohod's Representation Theorem for random variables taking values in  $\mathbb{Z}$. The reader could check Section 8.1 in \cite{durrett} for a detailed review.  

\smallskip

\begin{proposition} \label{SRT}
 If $X$ is a random variable taking values in $\mathbb{Z}$ with $\esp[X]=0$ and $\esp[X^2]<\infty$, then there exist a stopping time $S$ for a standard Brownian motion $\{B(s):s\geq 0\}$, so that 
$B(S)\overset{d}{=}X$ and $S=\inf\{s\geq 0: B(s)\notin (U,V)\}$ where the random vector $(U,V)$ is independent of $\{B(s):s\geq 0\}$ and takes values in $\{(0,0)\}\cup \ \{\dots,-2,-1\}\times\{1,2,\dots\}$.
\end{proposition}

\medskip

 Fix $m\in\Z, m\neq 0 $, and let us represent the distance between the paths $\pi^{(0,0)}$ and $\pi^{(m,0)}$ on their common renewal times by $\{Y^m_n:n\geq 0\}$; i.e. $Y^m_0:=m$ and
\begin{align*}
 Y^m_n:= X_{\tau_n(u_m)}^{u_m}(1)-X_{\tau_n(u_0)}^{u_0}(1) \, \text{ for } \, n\geq 1,
\end{align*}
where $u_0 = (0,0)$, $u_m = (m,0)$ and the sequences $\{\tau_n(u_0):n\geq 1\}$ and $\{\tau_n(u_m): n\geq 1\}$ are as in the statement of Corollary \ref{renewal time corollary}. 

\begin{remark}
The processes $\{Y^m_n:n\geq 0\}$, $m\ge 1$, are spatially inhomogeneous random walks with mean zero square integrable independent increments and thus they are also square integrable martingales.
\end{remark}

In Section \ref{coalescing time} we obtain estimates on the tail probability of the coalescing time for the paths $\pi^{(0,0)}$ and $\pi^{(m,0)}$. This coalescing time turns out to be $X^{u_0}_{\tau_\theta(u_0)}(2)$ with $\theta = \inf \{ n : Y^m_n = 0\}$. The Skorohod scheme furnishes a representation of $\{Y^m_n:n\geq 0\}$ that allows us to obtain proper estimates for the hitting time $\theta$. A difficulty that arises here is due to crossings of the paths $\pi^{(0,0)}$ and $\pi^{(m,0)}$, which means that $\{Y^m_n:n\geq 0\}$ can alternate between positive and negative values before $\theta$, this requires a proper control on the overshoot distribution of the process when it changes sign. 
For all $n\geq 1$ we can write
\begin{align*}
    Y_{n}^m = \sum_{j=1}^n\Big\{\big[X_{\tau_j(u_m)}^{u_m}(1)-X_{\tau_{j-1}(u_m)}^{u_m}(1)\big]- \big[X_{\tau_j(u_0)}^{u_0}(1)-X_{\tau_{j-1}(u_0)}^{u_0}(1)\big] \Big\} + m.
\end{align*}
Take $Z$ as in the statement of Corollary \ref{renewal time corollary} for the paths of the GRDF starting in $(0,0)$ and $(m,0)$. By symmetry and translation invariance of the GRDF model, we have that
\begin{align*}
    \esp\big[X_{\tau_j(u_0)}^{u_0}(1)-X_{\tau_{j-1}(u_0)}^{u_0}(1)\big]=\esp\big[X_{\tau_j(u_m)}^{u_m}(1)-X_{\tau_{j-1}(u_m)}^{u_m}(1)\big] = 0,
\end{align*}
and 
\begin{align*}
        \esp\Big[\big(X_{\tau_j(u_0)}^{u_0}(1)-X_{\tau_{j-1}(u_0)}^{u_0}(1)\big)^2\Big]=\esp\Big[\big(X_{\tau_j(u_m)}^{u_m}(1)-X_{\tau_{j-1}(u_m)}^{u_m}(1)\big)^2\Big]\leq \esp\big[Z^2\big]
\end{align*}
for all $j\geq 1$. 

Applying Proposition \ref{SRT} recursively, we can fix a standard Brownian motion $B = \{B(s):s\geq 0\}$ and stopping times $\{S_i:i\geq 1\}$ for the canonical filtration of $B$ such that
\begin{align*}
    \big[X_{\tau_j(u_m)}^{u_m}(1)-X_{\tau_{j-1}(u_m)}^{u_m}(1)\big]- \big[X_{\tau_j(u_0)}^{u_0}(1)-X_{\tau_{j-1}(u_0)}^{u_0}(1)\big] \overset{d}{=} B(S_n)- B(S_{n-1}),
\end{align*}
for all $n\geq 1$, where $S_0 = 0$ and
$$
 S_n:=\inf \Big\{ s\geq S_{n-1}: B(s)-B(S_{n-1}) \notin \big(U_n(B(S_{n-1})+m),V_n(B(S_{n-1})+m) \big) \Big\},
$$
where $\big\{(U_i(l),V_i(l)): l\in\Z, l\neq 0, i\geq 1\big\}$ is a family of independent random vectors taking values in $\{(0,0)\}\cup \ \{\dots,-2,-1\}\times\{1,2,\dots\}$ such that for every fixed $l \in\Z$, $(U_i(l),V_i(l))$, $i\ge 1$, are identically distributed. Since $\{B(S_n)-B(S_{n-1}): n\geq 1\}$ and 
$$
\Big\{\big[X_{\tau_j(u_m)}^{u_m}(1)-X_{\tau_{j-1}(u_m)}^{u_m}(1)\big]- \big[X_{\tau_j(u_0)}^{u_0}(1)-X_{\tau_{j-1}(u_0)}^{u_0}(1)\big]: n\geq 1 \Big\}
$$
are sequences of independent random variables, we have that
$$
Y^m_n \stackrel{d}{=}B(S_n) + m \quad \forall \, n\ge 0 \, .
$$
The above representation is called here "the Skorohod's scheme" for the difference process $\{Y^m_n:n\geq 0\}$.

Note that the case $(U_n(B(S_{n-1})),V_n(B(S_{n-1}))) = (0,0)$ implies that $Y^m_n = Y^m_{n-1}$ and $S_n = S_{n-1}$ and so
the sequence $\{S_i:i\geq 1\}$ does not need to be strictly increasing. 

\medskip

The next Lemma already shows how useful the above representation is for the difference process $\{Y^m_n:n\geq 0\}$. Note that if $Y^m_n = r > 0$ and $Y^m_\cdot$ changes sign at time $n+1$, then we can have $Y^m_{n+1} = U_n(r)+r < 0$. So the conditional distribution $U_n(r)+r|\{U_n(r)+r<0\}$ can be used to control the overshoot distribution for the difference process given that it is at state $r$ when it changes sign. However $U_n(r)+r|\{U_n(r)+r<0\}$ and the overshoot distribution are not exactly the same since we may have $Y^m_{n+1} = V_n(r) + r \neq U_n(r)+r$. Lemma \ref{crossing lower bound} below gives an uniform control with respect to $r$ on $U_n(r)+r|\{U_n(r)+r<0\}$. 

\medskip

\begin{lemma}\label{crossing lower bound}
 There exists a random variable $\widehat{R}$ such that $\esp\big[ \widehat{R}^k\big]<\infty$ for all $k \geq 1$, and
 \begin{align*}
    U_1(m)+m|\{U_1(m)+m<0\}\geq^{st} - \widehat{R},
 \end{align*}
 for all $m\geq 1$, where $\geq^{st}$ means  “is stochastically greater or equal than”.
\end{lemma}

\begin{proof} Take $T$ as the renewal time from Proposition \ref{renewal time proposition} for the processes $\pi^{(0,0)}$ and $\pi^{(m,0)}$. From the definition of the renewal time $T$, we always have that $|\pi^{(m,0)}(T) - \pi^{(0,0)}(T)| \le 2 T^2$. Indeed both paths will perform at most $T$ jumps during time interval $[0,T]$ and each jumps allows a maximum displacement of $T$, otherwise one of the dependence regions associated to these jumps would intersect the environment above time $T$, which contradicts the definition of $T$.

From the proofs of Lemma \ref{tpr} and Proposition \ref{renewal time proposition}, $T$ is bounded above by a random variable $Z$ which is a function of the random variables $\xi_j$, $M$ and $G$. By construction, these random variables do not depend on the positions of $X^{(m,0)}$ and $X^{(0,0)}$ relatively to each other, although $T$ does depend (for instance $T$ depends on the event $\pi^{(m,0)}$ and $\pi^{(0,0)}$ coalesce before time $T$ while $Z$ doesn't). This latter assertion obviously requires a close inspection of the proofs of Lemma \ref{tpr} and Proposition \ref{renewal time proposition}, but it is nonetheless straightforward.  

The event $\{U_1(m)+m<0\}$ is basically the event: $(\pi^{(m,0)}(T) - \pi^{(0,0)}(T)) \notin [0,m]$. So the event $\{U_1(m)+m<0\}$ depends only on the positions of $X^{(m,0)}$ and $X^{(0,0)}$ relatively to each other. Moreover $U_1(m)+m$ given $\{U_1(m)+m<0\}$ is bounded above by two times the square of the time $X^{(0,0)}$ spends at the right of $X^{(m,0)}$, which is bounded above, independently of the positions of $X^{(m,0)}$ and $X^{(0,0)}$ relatively to each other, by the sum of the random variables $\xi_j$ relative to renewal times
$\tau_j$, $j\ge 1$, for $X^{(0,0)}$ where $\pi^{(m,0)} < \pi^{(0,0)}$ and $X^{(0,0)}_{\tau_j}(2) \le T$. This sum is stochastically bounded above by the random variable $Z$ independently of $\{U_1(m)+m<0\}$ and $U_1(m)+m$ given $\{U_1(m)+m<0\}$ is necessarily stochastically bounded from below by $- 2 Z^2 = - \widehat{R}$. 
\end{proof}

\medskip

For any set $\mathbbm{A} \subset \mathbb{R}$, let us define $\nu^m_{\mathbbm{A}}$ as the first hitting time of $\mathbbm{A}$ of the random walk $\{Y^m_n:n\geq 1\}$; i.e. 
\begin{equation} \label{nuA}
    \nu^m_{\mathbbm{A}}:=\inf\{n\geq 1: Y_n^m\in\mathbbm{A}\}.
\end{equation}
The next lemma states, among other claims, that the probability of ocurrence of $k$ crossings before coalescence of two paths in the GRDF decays exponentially fast in $k$.  

\smallskip

\begin{lemma}\label{a_l} The following properties hold:
\begin{enumerate}
    \item [(i)] For all $m\in\N$ we have $\pr\big[\nu^m_{(-\infty,0]}<\infty\big] = 1$.
    \item [(ii)] $ \inf_{m\geq 1}\esp\big[Y^m_{\nu^m_{(-\infty,0]}}\big]>-\infty$.
    \item [(iii)] $\inf_{m\geq 1}\pr\big[Y^m_{\nu^m_{(-\infty,0]}}=0\big]>0$.
    \item [(iv)] Let us define the sequence $(a_l)_{l\geq 1}$ as
                \begin{equation*}
                a_1:=\inf\{n\geq 1: Y_n^1\leq 0\} 
                \end{equation*}
                and for $l \ge 2$
                \begin{equation*}
                a_l:=
                  \left\lbrace
                  \begin{array}{l}
                     \inf\{n\geq a_{l-1}: Y^1_n\geq 0\}; \text{ if } l \text{ is even} \\
                     \inf\{n\geq a_{l-1}: Y^1_n\leq 0\}; \text{ if } l \text{ is odd} \, . \\
                  \end{array}
                  \right. 
                \end{equation*}
                Then there exists a constant $c_1<1$ such that 
                $$
                \pr[Y^1_{a_j}\neq 0, \text{ for } j=1,\dots,k]\leq c_1^k,
                $$ 
                for all $k\geq 1$.
\end{enumerate}
\end{lemma}

\begin{proof}
We will use the Skorohod Scheme for the difference process $\{Y^m_n:n\geq 0\}$. Let $(B(s))_{s\ge 0}$ be a Brownian Motion and $(S_n)_{n\ge 1}$ be stopping times given by the Skorohod Scheme.

Let us start proving $(i)$. Fix $m \in \mathbb{N}$ and, to simplify notation, put $B'(s) = B(s) + m$, $s \ge 0$ so $(B'(s))_{s\ge 0}$ is a standard Brownian motion starting at $m$. Define
\begin{align*}
    D:=\big\{n \in [1,\nu^m_{(-\infty,0]}]\cap\N: (B'(s))_{s\geq 0} \text{ visits } (-\infty,0]  \text{ in the interval } (S_{n-1},S_n] \big\} .
\end{align*}
For each $n\in D$, we have two possibilities:
\begin{enumerate}

\item [(a)] $B'(S_n) = U_n(B'(S_{n-1}))+B'(S_{n-1})=0$ thus $n=\nu^m_{(-\infty,0]}$.

\item [(b)] $U_n(B'(S_{n-1}))+B'(S_{n-1})<0$. In this case we have that $(B'(s))_{s\geq 0}$ visits zero in the interval $(S_{n-1},S_n)$, before hitting 
$$\{U_n(B'(S_{n-1}))+B'(S_{n-1}), V_n(B'(S_{n-1}))+B'(S_{n-1})\}$$ 
at time $S_n$. By the Strong Markov property and Lemma \ref{crossing lower bound} we get that the probability that a standard Brownian motion starting at $0$ leaves the interval $[-\widehat{R},1]$ by the left side is a lower bound to the probability that $B'(S_n)=U_n(B'(S_{n-1}))+B'(S_{n-1})$. Note that from the statement and proof of Lemma \ref{crossing lower bound} that $\widehat{R}$ is independent of $(B(s))_{s\geq 0}$. Therefore
$$
 \#\{n\in D: B'(S_n)=V_n(B'(S_{n-1}))+B'(S_{n-1})\}\leq G-1
$$
where $G$ is a geometric random variable whose parameter is the probability that a standard Brownian motion starting at $0$ leaves the interval $[-\widehat{R},1]$ by the left side. Thus $\nu^{m}_{(-\infty,0]}<\infty$ almost surely and $(i)$ holds.
\end{enumerate}

\medskip

Now we prove $(ii)$. From the proof of $(i)$ above, we have a geometric random variable $G$ and i.i.d. random variables $\{\widehat{R}_i:i\geq 1\}$ such that 
$$
Y^m_{\nu^m_{(-\infty,0]}}\geq-\sum_{i=1}^G\widehat{R}_i.
$$
Hence 
$$
\esp\big[Y^m_{\nu^m_{(-\infty,0]}}\big]\geq-\esp\Big[\sum_{i=1}^{G}\widehat{R}_i\Big]
$$
for all $m\geq 1$, and from Lemma \ref{crossing lower bound} we have $(ii)$.
\smallskip

To prove  $(iii)$  consider $m,M\in\N$ with $m>M$.  
Since $\nu^m_{(-\infty,M]}\leq\nu^m_{(-\infty,0]}$, by $(i)$, we get that $\nu_{(-\infty,M]}^m$ is finite almost surely. Also note that
\begin{align*}
\pr\big[Y^m_{\nu^m_{(-\infty,0]}}=0\big]&\geq\pr\big[Y^m_{\nu^m_{(-\infty,0]}}=0,\nu^m_{(-\infty,0]}\neq\nu^m_{(-\infty,M]}\big]\\
&=\sum_{k=1}^{M}\pr\big[Y^m_{\nu^m_{(-\infty,0]}}=0,Y^m_{\nu^m_{(-\infty,M]}}=k\big]\\
&=\sum_{k=1}^{M}\pr\big[Y^m_{\nu^m_{(-\infty,0]}}=0\big| Y^m_{\nu^m_{(-\infty,M]}}=k\big]\pr\big[Y^m_{\nu^m_{(-\infty,M]}}=k\big].
\end{align*}
For all $1\leq k\leq M$ by the Strong Markov  property of $(Y_n^m)_{n\geq 0}$ and the translation invariance of the model, we have that
\begin{align*}
\pr\big[Y^m_{\nu^m_{(-\infty,0]}}=0 \big| Y^m_{\nu^m_{(-\infty,M]}}=k\big]=\pr\big[Y^k_{\nu^k_{(-\infty,0]}}=0\big].
\end{align*}
Hence
\begin{align*}
\pr\big[Y^m_{\nu^m_{(-\infty,0]}}=0 \big] & \geq\sum_{k=1}^{M}\pr\big[Y^k_{\nu^k_{(-\infty,0]}}=0\big]\pr\big[Y^m_{\nu^m_{(-\infty,M]}}=k\big]\\
&\geq\Big(\min_{1\leq k\leq M}\pr[Y^k_{\nu^k_{(-\infty,0]}}=0]\Big)\sum_{k=1}^{M}\pr\big[Y^m_{\nu^m_{(-\infty,M]}}=k\big]\\&=\Big(\min_{1\leq k\leq M}\pr\big[Y^k_{\nu^k_{(-\infty,0]}}=0\big]\Big)\pr\big[\nu^m_{(-\infty,0]}\neq\nu^m_{(-\infty,M]}\big]\\&\geq \Big(\min_{1\leq k\leq M}\pr\big[Y^k_{\nu^k_{(-\infty,0]}}=0\big]\Big)\Big(\inf_{\tilde{m}>M}\pr\big[\nu^{\tilde{m}}_{(-\infty,0]}\neq\nu^{\tilde{m}}_{(-\infty,M]}\big]\Big) .
\end{align*}
From the description of the GRDF it is straightforward to verify that $\pr[Y^k_{\nu^k_{(-\infty,0]}}=0]>0$ for all $k\geq 1$, indeed $\{Y^k_{\nu^k_{(-\infty,0]}}=0\}$ contains an event that can be specified by a configuration for the environment on a finite number of points. Then $\min_{1\leq k\leq M}\pr[Y^k_{\nu^k_{(-\infty,0]}}=0]>0$. Let us prove that for an adequate $M$ we have 
$$
\inf_{\tilde{m}>M}\pr[\nu^{\tilde{m}}_{(-\infty,0]}\neq\nu^{\tilde{m}}_{(-\infty,M]}]>0.
$$

Note that $\pr[\nu^{\tilde{m}}_{(-\infty,0]}=\nu^{\tilde{m}}_{(-\infty,M]}]=\pr[Y^{\tilde{m}}_{\nu^{\tilde{m}}_{(-\infty,M]}}\leq 0]$. By symmetry and translation invariance, we have 
\begin{align*}
    \pr\big[Y^{\tilde{m}}_{\nu^{\tilde{m}}_{(-\infty,M]}}\leq 0\big]&=\pr\big[Y^{(\tilde{m}-M)}_{\nu^{(\tilde{m}-M)}_{(-\infty,0]}}\leq- M\big]\leq\frac{1}{M}\big(-\esp\Big[Y^{(\tilde{m}-M)}_{\nu^{(\tilde{m}-M)}_{(-\infty,0]}}\Big]\big)
    \leq-\frac{1}{M}\inf_{m\geq 1}\esp\big[Y^m_{\nu^m_{(-\infty,0]}}\big]   
\end{align*}

By $(ii)$ we have that $\inf_{m\geq 1}\esp\big[Y^m_{\nu^m_{(-\infty,0]}}\big]>-\infty$, hence taking $M$ such that 
$$
c:=-\frac{1}{M}\inf_{m \ge 1}\esp\big[Y^m_{v^m_{(-\infty,0]}}\big]<1,
$$ 
we get
\begin{align*}
\inf_{\tilde{m}>M}\pr[Y^{\tilde{m}}_{\nu^{\tilde{m}}_{(-\infty,0]}}=0]\geq\Big(\min_{1\leq k\leq M}\pr[Y^k_{\nu^k_{(-\infty,0]}}=0]\Big)(1-c)>0,
\end{align*}
which completes the proof of (iii).

\medskip

Le us prove (iv). Define
\begin{align*}
c_1:= \sup_{m\geq 1}\pr[Y^m_{a_1}\neq 0].
\end{align*} 
By (iii) we get that $\pr[Y^1_{a_1}\neq 0]\leq c_1 < 1$. The proof will follow by induction on $k$. Suppose that $\pr[Y^1_{a_j}\neq 0, \text{ for } j=1,\dots,k]\leq c_1^{k}$. Here we are going to assume that $k$ is even, the case when $k$ is odd is similar. Write
\begin{align*}
&\pr[Y^1_{a_j}\neq 0 \text{ for } j=1,\dots,k+1] \\&=\sum_{m\geq 1}\pr[Y^1_{a_{k+1}}\neq 0,Y^1_{a_k}=m,Y^1_{a_j}\neq 0\text{ for } j=1,\dots,k-1]\\&=\sum_{m\geq 1}\pr[Y^1_{a_{k+1}}\neq 0| Y^1_{a_k}=m, \cap_{j=1}^{k-1}\{ Y^1_{a_j}\neq 0\}]\pr[Y^1_{a_k}=m,\cap_{j=1}^{k-1} \{Y^1_{a_j}\neq 0\}].
\end{align*}
By the Strong Markov property of $(Y_n^1)_{n\geq 0}$ and translation invariance, we have that
\begin{align*}
\pr[Y^1_{a_{k+1}}\neq 0|Y^1_{a_k}=m,Y^1_{a_j}\neq 0\text{ for } j=1,\dots,k-1]=\pr[Y^m_{a_1}\neq 0]\leq c_1.
\end{align*}
Hence
\begin{align*}
\pr[Y^1_{a_j}\neq 0 \text{ for } j=1,\dots,k+1]&\leq c_1\sum_{m\geq 1}\pr[Y^1_{a_k}=m,Y^1_{a_j}\neq 0\text{ for } j=1,\dots,k-1]\\&=c_1\pr[Y^1_{a_j}\neq 0 \text{ for } j=1,\dots,k]\\
&\leq c_1^{k+1}.
\end{align*}
\end{proof}

\medskip

\begin{lemma}\label{R_j}  
Fix $m=1$ and consider the sequence $(a_l)_{l\geq 1}$ as in the statement of Lemma \ref{a_l}.  Then there exist a standard Brownian motion $\{\mathbb{B}(s):s\geq 0\}$ independent of $Y^1$, an integrable random variable $R_0$ and  a sequence of independent random variables $\{R_i:i\geq 1\}$ with values on the non negative integers, so that $\{\mathbb{B}(s): s\geq 0\}$, $\{R_i:i\geq 1\}$ and $R_0$ are independent and satisfy: 
\begin{enumerate}
\item [(i)] $R_i|\{R_i\neq 0\}\overset{d}{=} R_0$ for all $i\geq 1$.
\item[(ii)] $S_{a_l}$ is stochastically dominated by $J_l$, which is defined as $J_0 = 0$,
\begin{align*}
J_1:=\inf\{s\geq 0: \mathbb{B}(s) - \mathbb{B}(0) = -(R_1 + R_0)\},
\end{align*}
and
\begin{align*}
J_{l}:=\inf\{s\geq J_{l-1}: \mathbb{B}(s)-\mathbb{B}(J_{l-1})=(-1)^l(R_l+R_{l-1})\}, \text{ for } l\geq 2.
\end{align*}
\item[(iii)] $Y^1_{a_l} \neq 0$ implies that $\mathbb{B}(J_l)\neq 0$, which is equivalent to $R_l \neq 0$, given that $\mathbb{B}(0) = R_0$.
\end{enumerate}
\end{lemma}

\medskip

\begin{proof} Recall the Skorohod Scheme for the difference process $\{Y^1_n:n\geq 1\}$ as well as the definitions of the geometric random variable $G$ and the i.i.d. random variables $\{\widehat{R}_i:i\geq 1\}$ in the proof of Lemma $\ref{a_l}$. Also as in the proof of Lemma $\ref{a_l}$ consider the following random set
\begin{align*}
    D:=\Big\{ n\in[1,\nu^1_{(-\infty,0]}]\cap\N: (B(s)+1)_{s\geq 0} \text{ visits } (-\infty,0] \text{ in the interval } (S_{n-1},S_n]\Big\}.
\end{align*}
By the proof of Lemma \ref{a_l} we have that 
\begin{align*}
    -\sum_{i=1}^{G}\widehat{R}_i \leq -\sum_{i=1}^{|D|}\widehat{R}_i \leq \inf_{0\leq s\leq \nu^1_{(-\infty,0]}} B(s) \, .
\end{align*}
Let us define the random variable $R_1$ as
\begin{equation*}
                R_1:=
                  \left\lbrace
                  \begin{array}{cl}
                      - \sum_{i=1}^{G}\widehat{R}_i & , \text{ if } |D|>1 \textrm{ or } |D| =1 \textrm{ and } U_n(B(S_{n-1})+1)+B(S_{n-1})+1 < 0 \textrm{ for } n \in D,  \\
                     0 & ,  \text{ if } |D| =1 \textrm{ and } U_n(B(S_{n-1})+1)+B(S_{n-1})+1 = 0 \textrm{ for } n \in D,
                  \end{array}
                  \right. 
\end{equation*}
and $R_0$ as an independent random variable such that $R_0\overset{d}{=}R_1|\{R_1\neq 0\}$. Now we need to consider standard Brownian Motions starting at $0$ $(\tilde{B}^j(s))_{s \ge 0}$, $j\ge 1$, independent of $(B(s))_{s \ge 0}$ and $\big\{(U_i(l),V_i(l)): l\in\Z, l\neq 0, i\geq 1\big\}$. Define 
\begin{align*}
    J_1:=\inf \{s\geq S_{a_1}: B(S_{a_1}) + \tilde{B}^1(s-S_{a_1}) = -(R_1+R_0)\}  
\end{align*}
which is above $S_{a_1}$ by definition. Define $(\mathbb{B}(s))_{0\le s \le J_1}$ as $\mathbb{B}(s) = B(s) + R_0$ for $0\le s \le S_{a_1}$ and $\mathbb{B}(s) = B(S_{a_1}) + \tilde{B}^1(s-S_{a_1}) + R_0$ for $S_{a_1}\le s \le J_1$.  Note that $Y^1_{a_1} = B(S_{a_1}) + 1 \neq 0$ implies that $R_1>0$, then $\mathbb{B}(J_1) = - R_1 < 0$. Moreover $J_1$ has the same distribution of $\nu^1_{(-\infty,-(R_1+R_0)]}$.

\medskip

From this point, it is straightforward to use an induction argument to build the sequence $\{R_j:j\ge 1\}$. At step $j$ in the induction argument, we consider initially an excursion of $(B(s))_{s\ge 0}$ in a time interval of size $(S_{a_{j}}-S_{a_{j-1}})$, and since $|B(S_{a_{j-1}})| \overset{st}{\leq} R_{j-1}$ we can obtain $R_j$ and define $J_j$ using $(\mathbb{B}(s))_{s\ge 0}$ as before. By the strong Markov property of $\{B(S_n): n\geq 1\}$, we obtain that the $R_j$'s are independent and $B(S_{a_j})\neq 0$ is equivalent to $\mathbb{B}(J_j)\neq 0$. Define $(\mathbb{B}(s))_{J_{l-1} \le s \le J_l}$ as $\mathbb{B}(s) = B(s) - B(S_{a_{l-1}}) + \mathbb{B}(J_{l-1})$ for $J_{l-1} \le s \le J_{l-1} + (S_{a_l} - S_{a_{l-1}})$ and $\mathbb{B}(s) = \tilde{B}^l \big(s - J_{l-1} - (S_{a_l} - S_{a_{l-1}})\big) + B(S_{a_l}) - B(S_{a_{l-1}}) + \mathbb{B}(J_{l-1})$ for $ J_{l-1} + (S_{a_l} - S_{a_{l-1}}) \le s \le J_l$. Note that $Y^1_{a_l} = B(S_{a_l}) + 1 \neq 0$ implies that $R_l > 0$, then $\mathbb{B}(J_l) = - R_l < 0$.
\end{proof}

\medskip

\section{Coalescing Time.}\label{coalescing time}

In this section we obtain an upper bound on the tail probability of the coalescence time of two paths in $\mathcal{X}$. This is a central estimate related to convergence to the Brownian web. The main ideas used here to get the bound come from \cite{coletti2009scaling},\cite{coletti2014convergence} and \cite{roy2013random}, although it is not a straightforward application of the techniques used before. Here we have another important difference with the Random Directed Forest studied in \cite{roy2013random}, because of the possibility of crossings before coalescence. This property does not allow us to adapt the proof given in \cite{roy2013random}. We will need the ideas used in \cite{coletti2014convergence}, where the authors work with a system allowing crossing to obtain the upper bound. 

\medskip

The aim of this section is to prove the following result.

\begin{proposition}\label{tau(u,v)}
Define $\vartheta := \inf\{t\geq 0:\pi^{(0,0)}(s)=\pi^{(1,0)}(s)\text{ for all }s\geq
 t\}$. Then there exists a positive constant $C$ such that
\begin{align*}
\pr[\vartheta >k]\leq\frac{C}{\sqrt{k}},
\end{align*}
for all $k\geq 1$.
\end{proposition}

\medskip

Put $\nu^m:=\nu^m_{\{0\}}$, for $m\in\Z$, see \eqref{nuA}, and also write $\nu := \nu^1$. So $\nu$ is the number of renewals associated to $\pi^{(0,0)}$ and $\pi^{(1,0)}$ required for coalescence between these two paths on a common renewal time, which in this case is $T_{\nu}$, where $(T_n)_{n\geq 1}$ are the renewal times defined in the statement of Corollary \ref{renewal time corollary} for the points $(0,0)$ and $(1,0)$. Therefore we have that $\vartheta \le T_{\nu}$ and Proposition \ref{tau(u,v)} follows directly from the next lemma.

\smallskip

\begin{lemma}\label{tau Y} 
There are positive constants $C_{1}$ and $C_{2}$ such that 
\begin{align}\label{pro tau > t}
\pr[\nu>k]\leq\frac{C_1}{\sqrt{k}}
\end{align}
and
\begin{align}\label{pro tau > t2}
\pr[T_{\nu}>k]\leq\frac{C_2}{\sqrt{k}} 
\end{align}
for every $k\ge 1$, where $(T_n)_{n\geq 1}$ are the renewal times defined in the statement of Corollary \ref{renewal time corollary} for the points $(0,0)$ and $(1,0)$.
\end{lemma}

\medskip

\begin{proof} Let us suppose that \eqref{pro tau > t} is true and use it to prove \eqref{pro tau > t2} with the same idea used in \cite{roy2013random}. Recall from  Corollary \ref{renewal time corollary} that $T_1$ has finite moments and define the constant $L:= 1/ 2 \esp[T_1]$. Then for $k\in\N$ 
\begin{align*}
\pr[T_{\nu}>k]&\leq\pr[T_{\nu}>k,\nu \leq Lk]+\pr[\nu >L k]\leq\pr[T_{\lfloor L k\rfloor}>k]+\pr[\nu >Lk].
\end{align*}
By $(\ref{pro tau > t})$, it is enough to prove that $\pr[T_{\lfloor L k\rfloor}>k]\leq\frac{C_3}{\sqrt{k}}$ for some constant $C_3$. Then
\begin{align*}
\pr\big[T_{\lfloor L k\rfloor}>k\big]&=\pr\Big[\sum_{i=1}^{\lfloor L k\rfloor}[T_i-T_{i-1}]>k\Big] \\
&=\pr\Big[\sum_{i=1}^{\lfloor L k\rfloor}[T_i-T_{i-1}]-{\lfloor L k\rfloor}\esp[
T_1]>k-\lfloor L k\rfloor \esp[T_1]\Big]\\
&\leq\frac{\var\Big[\sum_{i=1}^{\lfloor L k\rfloor}(T_i-T_{i-1})\Big]}{\Big(k-\lfloor L k\rfloor \esp[T_1]\Big)^2} =\frac{\lfloor L k\rfloor\var[T_1]}{\Big(k-\lfloor L k\rfloor \esp[T_1]\Big)^2}.
\end{align*}
Note that 
\begin{align*}
\sqrt{k}\frac{\lfloor L k\rfloor\var[T_1]}{\Big(k-\lfloor L k\rfloor \esp[T_1]\Big)^2}\rightarrow 0 \text{ as } k\rightarrow 0.
\end{align*}
Then there exists $M$ such that
\begin{align*}
\frac{\lfloor L k\rfloor\var[T_1]}{\Big(k-\lfloor L k\rfloor \esp[T_1]\Big)^2}\leq\frac{1}{\sqrt{k}}
\end{align*}
for all $k\geq M.$ Hence we can find a sufficiently large constant $C_3 > 0$ such that 
\begin{align*}
\frac{\lfloor L k\rfloor\var[T_1]}{\Big(k-\lfloor L k\rfloor \esp[T_1]\Big)^2}\leq\frac{C_3}{\sqrt{k}}
\end{align*}
for all $k\geq 1$. So we have  \eqref{pro tau > t2}. 

\medskip

To  prove (\ref{pro tau > t}) take $\{B(s):s\geq 0\}$ and $\{S_i:i\geq 1\}$ from the Skorohod's Scheme as in Section \ref{sec:skorohod}, then 
$$
 (Y^1_k)_{k \ge 1} \overset{d}{=} (B(S_k) + 1)_{k \ge 1}
$$
and define 
$$
\widehat{\nu}:= \inf\{k\geq 1: B(S_k) + 1 = 0\}.
$$
We will prove (\ref{pro tau > t}) for $\widehat{\nu}$ what implies (\ref{pro tau > t}) since $\widehat{\nu}\overset{d}{=}\nu$.

\smallskip

For a $\delta>0$ to be fixed later and every $k\in\N$ we have that,
\begin{align}\label{pr[tau>t]<pr[Sk >deltak]+ ...}
\pr[\widehat{\nu}>k]=\pr[S_k\leq\delta k,\widehat{\nu}>k] + \pr[S_k>\delta k,\widehat{\nu}>k].
\end{align}
First we get an upper bound on $\pr[S_k\leq\delta k,\widehat{\nu}>k]$. From the Skorohod's representation 
\begin{align*}
S_k = \sum_{i=1}^{k}\big(S_i-S_{i-1}\big)=\sum_{i=1}^{k}Q_i(B(S_{i-1})),
\end{align*} 
where $\big\{Q_i(m); i\geq1,m\in\Z\big\}$ are independent random variables and $Q_i(m)$ is independent of $(B(S_1),...,B(S_{i-1}))$ for all $i\in\N,m\in\Z$. By definition, on $\{\widehat{\nu}>k\}$ we have that $B(S_i)\neq 0$ for every $i\in\{1,\dots,k\}$.

\smallskip

Fix $\lambda>0$, then
\begin{align*}
\pr[S_k\leq\delta k,\widehat{\nu}>k]=\pr[e^{-\lambda S_k}\geq e^{-\lambda\delta k},\widehat{\nu}>k]\leq e^{\lambda\delta k}\esp[ e^{-\lambda S_k} \mathbbm{1}_{\{\widehat{\nu}>k\}}]
\end{align*}

\begin{claim}\label{sup E[eQ(m)]}
\begin{align*}
\esp[ e^{-\lambda S_k} \mathbbm{1}_{\{\widehat{\nu}>k\}}]\leq \Big(\sup_{m\in\Z\setminus\{0\}}\esp[e^{-\lambda Q(m)}]\Big)^k,
\end{align*}
where, for each $m$, $Q(m)$ is a random variable with the same distribution of $Q_1(m)$.
\end{claim}

\begin{proof}[Proof of Claim \ref{sup E[eQ(m)]}]
The proof is essentially the same given in Theorem $4$ in \cite{coletti2009scaling}. We include it here for the sake of completeness. Taking $\mathcal{F}_k:=\sigma(B(S_1),\dots,B(S_k))$ we have that
\begin{align*}
\esp[e^{-\lambda S_k}\mathbbm{1}_{\{\widehat{\nu}>k\}}]&=\esp\Big[\esp[e^{-\lambda S_k}\mathbbm{1}_{\{\widehat{\nu}>k\}}|\mathcal{F}_{k-1}]\Big]\\&\leq\esp\Big[e^{-\lambda S_{k-1}}\esp[e^{-\lambda Q_k(B(S_{k-1}))}\mathbbm{1}_{\{\widehat{\nu}>k-1\}}\mathbbm{1}_{\{B(S_{k-1}))\neq 0\}}|\mathcal{F}_{k-1}]\Big]\\&=\esp\Big[e^{-\lambda S_{k-1}}\mathbbm{1}_{\{\widehat{\nu}>k-1\}}\esp[e^{-\lambda Q_k(B(S_{k-1}))}\mathbbm{1}_{\{B(S_{k-1})\neq 0\}}|\mathcal{F}_{k-1}]\Big]
\end{align*}
and
\begin{align*}
\esp[e^{-\lambda Q_k(Y_{k-1})}\mathbbm{1}_{\{B(S_{k-1})\neq 0\}}|\mathcal{F}_{k-1}]&=\sum_{m\in\Z\setminus \{0\}}\esp[e^{-\lambda Q_k(m)}\mathbbm{1}_{\{B(S_{k-1})=m\}}|\mathcal{F}_{k-1}]\\&=\sum_{m\in\Z\setminus\{0\}}\mathbbm{1}_{\{Y_{k-1}=m\}}\esp[e^{-\lambda Q_k(m)}|\mathcal{F}_{k-1}]\\&=\sum_{m\in\Z\setminus\{0\}}\mathbbm{1}_{\{Y_{k-1}=m\}}\esp[e^{-\lambda Q(m)}]\\&\leq\sup_{m\in\Z\setminus\{0\}}\esp[e^{-\lambda Q(m)}].
\end{align*}
So, applying the above argument recursively we obtain
\begin{align*}
\esp[e^{-\lambda S_k}\mathbbm{1}_{\{\widehat{\nu}>k\}}]\leq\esp[e^{-\lambda S_{k-1}}\mathbbm{1}_{\{\widehat{\nu}>k-1\}}]\Big(\sup_{m\in\Z\setminus\{0\}}\esp[e^{-\lambda Q(m)}]\Big)\leq\Big(\sup_{m\in\Z\setminus\{0\}}\esp[e^{-\lambda Q(m)}]\Big)^k.
\end{align*}
\end{proof}

Using Claim \ref{sup E[eQ(m)]} we get that 
\begin{align*}
\pr[S_k\leq\delta k,\widehat{\nu}>k]\leq \Big(e^{\lambda\delta}\sup_{m\in\Z\setminus\{0\}}\esp[e^{-\lambda Q(m)}]\Big)^k \, .
\end{align*}

Let $Q_{-1,1}$ be the exit time of interval $(-1,1)$ by a Standard Brownian motion. If $(U(m),V(m))\neq (0,0)$, then $U(m) \le -1$ and $V(m) \ge 1$ so $Q_{-1,1} \le Q(m)$ almost surely. Therefore
\begin{align} \label{cotaUV}
\esp\big[e^{-\lambda Q(m)}\big]&=\esp\big[e^{-\lambda Q(m)}|(U(m),V(m))\neq (0,0))\big]\pr[(U(m),V(m))\neq (0,0)] \nonumber \\
&+\pr\big[(U(m),V(m))= (0,0)\big] \nonumber \\
&\leq\esp\big[e^{-\lambda Q_{-1,1}}\big]\Big(1-\pr\big[(U(m),V(m))= (0,0)\big]\Big)+\pr\big[(U(m),V(m))= (0,0)\big] 
\end{align} 

\begin{claim}\label{sup p[U(m),V(m)] = (0,0)} $0<c_1:=\sup_{m\in\Z\setminus\{0\}}\pr\big[(U(m),V(m))=(0,0)\big]<1.$
\end{claim}

Using Claim \ref{sup p[U(m),V(m)] = (0,0)} and \eqref{cotaUV}, we obtain that
$$
\esp\big[e^{-\lambda Q(m)}\big] \le c_1(1-c_2)+c_2 \, .
$$
where $c_2=\esp\big[e^{-\lambda Q_{-1,1}}\big]<1$. 

Now chose $\delta$ such that $c_3:=e^{\delta\lambda}\big[c_1(1-c_2)+c_2\big]<1$. Then 
\begin{align}
\pr[S_k\leq\delta k,\widehat{\nu}>k]\leq c_3^k\leq\frac{c_4}{\sqrt{k}},
\end{align}
for some suitable $c_4>0$. This gives the bound we need on the first term of \eqref{pr[tau>t]<pr[Sk >deltak]+ ...}. Let us prove Claim \ref{sup p[U(m),V(m)] = (0,0)} before dealing with the second term on the right hand side of \eqref{pr[tau>t]<pr[Sk >deltak]+ ...}.

\begin{proof}[Proof of Claim \ref{sup p[U(m),V(m)] = (0,0)}] 
To simplify notation write $W := W_{(0,0)}$. The proof uses the hypothesis that $P(W=1)>0$ given in the definition of the environment for the GRDF. However by a straightforward adaptation, one can see that this is not required for the claim to remain valid.  

Recall that the event $\{\big(U(m),V(m)\big)=(0,0)\}$ is equivalent to the event that two paths in the GRDF initially at distance $m$ remain at distance $m$ on their first common renewal time. So, one can check, see Figure 4, that for all $m\in\Z\setminus\{0\}$ we have
\begin{align*}
    \pr[\big(U(m),V(m)\big)=(0,0)]\geq(p\pr[W=1])^2 \, .
\end{align*}
\begin{figure}[H]\label{UV=00-1}
\begin{tikzpicture}
\draw [step= 1cm, gray, very thin] (-0.7,-0.7) grid (3.7,1.7);
\fill[black](0,1) circle (1mm) ;\fill[black](3,1) circle (1mm) ;
\fill[black](0,0) node[below right]{$u$} circle (1mm);\fill[black](3,0) circle (1mm) node[below right]{$v$} ;
\draw[->,black] (0,0) -- (0,1);\draw[->,black] (3,0) -- (3,1);
\end{tikzpicture}
\caption{If $W_u = W_v = 1$ and $u + e_2$ and $v + e_2$ are open, which occurs with probability $(p\pr[W=1])^2$, then $\big(U(m),V(m)\big)=(0,0)$.}
\end{figure}
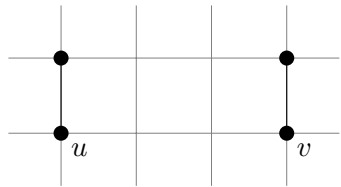
So $c_1:=\sup_{m\geq 1}\Big\{\pr[\big(U(m),V(m)\big)=(0,0)]\Big\}\geq(p\pr[W=1])^2>0.$ 

For the upper bound in the statement we have that
\begin{align*}
 \pr[\big(U(m),V(m)\big)\neq(0,0)]\geq (1-\frac{p}{2})p^3(1-p)^3(\pr[W=1])^3,   
\end{align*}
for all $m\in \mathbb{Z} \setminus \{0\} $, see Figure 5.
\begin{figure}[H]\label{UV=00-2}
\begin{tikzpicture}
\draw [step= 1cm, gray, very thin] (-0.7,-0.7) grid (4.7,2.7);
\fill[black](0,2) circle (1mm) ;\fill[black](4,1) circle (1mm);\fill[black](4,2) circle (1mm);
\draw[->,black] (0,0) -- (0,2);\draw[->,black] (3,0) -- (4,1);\draw[->,black] (4,1) -- (4,2);
\fill[white](0,1) circle (1mm);\draw[black](0,1) circle (1mm); circle (1mm)\fill[white](3,1) circle (1mm);\draw[black](3,1) circle (1mm);\fill[white](3,2) circle (1mm);\draw[black](3,2) circle (1mm);
\fill[black](0,0) circle (1mm) node[below right]{$u$};\fill[black](3,0) circle (1mm) node[below right]{$v$};
\end{tikzpicture}
\caption{If $W_u = W_v = W_{v+e_1+e_2} = 1$ and $u + 2 e_2$, $v+e_1+e_2$ and $v+e_1+2e_2$ are open and $u + e_2$, $v + e_2$ and $v+2e_2$ are closed, then with probability at least $(1-\frac{p}{2})p^3(1-p)^3(\pr[W=1])^3$, we have $(U(m),V(m)\big)\neq(0,0)$ for the pair of paths starting at $u$ and $v$.}
\end{figure}
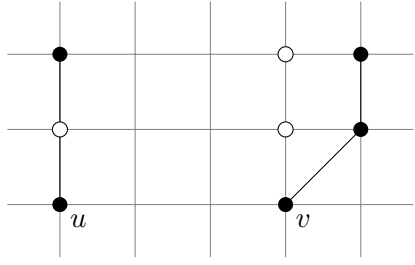
Hence $\inf_{m\in\Z\setminus\{0\}}\pr[\big(U(m),V(m)\big)\neq (0,0)]\geq (1-\frac{p}{2})p^3(1-p)^3(\pr[W=1])^3$. So,
\begin{align*}
c_1\leq 1 -  (1-\frac{p}{2})p^3(1-p)^3(\pr[W=1])^3<1.    
\end{align*}

\end{proof}

\smallskip

It remains to consider the second term on the right hand side of \eqref{pr[tau>t]<pr[Sk >deltak]+ ...}. To deal with it we consider an approach similar to \cite{coletti2014convergence}. Take the sequence $(a_l)_{l\geq 0}$ as in the statement of Lemma \ref{a_l}. Write 
\begin{align}
\pr[\widehat{\nu}> k, S_k>\delta k]\leq \sum_{l=1}^{k}\pr[\widehat{\nu}>k,S_k\geq \delta k, S_{a_{l-1}}<\delta k, S_{a_l}\geq \delta k] \, .
\end{align}
For now fix $l=1,\dots k$, from the definition of the Skorohod scheme
\begin{align*}
\{\widehat{\nu}>k,S_k\geq \delta k, S_{a_{l-1}}<\delta k, S_{a_l}\geq k\delta\}\subseteq\{B(S_{a_j}) + 1\neq 0, \text{ for } j = 1,\dots,l-1, S_{a_l}\geq k\delta\},
\end{align*}
thus
\begin{align*}
&\pr[\widehat{\nu}>k,S_k\geq \delta k, S_{a_{l-1}}<\delta k, S_{a_l}\geq k\delta]\\&\leq\pr[B(S_{a_j}) + 1 \neq 0, \text{ for } j = 1,\dots,l-1, S_{a_l}\geq k\delta]\\&=\pr\big[S_{a_l}\geq k\delta\big|B(S_{a_j}) + 1 \neq 0, \text{ for } j=1,\dots l-1\big]\pr\big[B(S_{a_j}) + 1 \neq 0, \text{ for } j=1,\dots l-1\big] \, .
\end{align*}

By item $(iv)$ in Lemma $\ref{a_l}$ we get 
\begin{align*}
\pr[\widehat{\nu}>k,S_k\geq \delta k, S_{a_{l-1}}<\delta k, S_{a_l}\geq k\delta]&\leq c_5^{l-1}\pr[S_{a_l}\geq k\delta|B(S_{a_j})+1\neq 0, \text{ for } j=1,\dots l-1].
\end{align*}
By Lemma \ref{R_j} there exist $({R}_j)_{j\geq 1}$ i.i.d. random variables with values on $\mathbb{N}$ which are independent of $(B(s))_{s\geq 0}$ such that if ${J}_0=0$ and
\begin{align*}
{J}_j := \inf\{s\geq {J}_{j-1}:B(s)-B({J}_{j-1})=(-1)^l({R}_j+{R}_{j-1})\}, j\geq 1,
\end{align*}
then 
\begin{align*}
\pr\big[S_{a_l} \geq k\delta \big|B(S_{a_j})+1\neq 0, \text{ for } j=1,\dots l-1\big] \le \pr[{J}_l \geq k\delta].
\end{align*}
Take $D_j:={J}_j-{J}_{j-1}$ for $j\geq 1$ and observe that $(D_j)_{j\geq 1}$ is an i.d. sequence, then we have that
\begin{align*}
\pr[\widehat{\nu}>k,S_k\geq \delta k, S_{a_{l-1}}<\delta k, S_{a_l}\geq k\delta]&\leq c_5^{l-1}\pr[{J}_l\geq k\delta]\leq c_5^{l-1} \, l \, \pr\Big[D_1\geq\frac{k\delta}{l}\Big].
\end{align*}

\smallskip

\begin{claim}\label{$D_j$}
 There exists a constant $c_7>0$ such that for every $x>0$ we have that
 \begin{align*}
 \pr[D_1\geq x]\leq\frac{c_7}{\sqrt{x}}.
 \end{align*}
\end{claim}


The previous claim is Lemma 3.6 in \cite{coletti2014convergence} where the reader can find a proof. Using Claim \ref{$D_j$} we have some constant $c_8$ such that
\begin{align*}
\pr[\widehat{\nu}>k,S_k\geq \delta k, S_{a_{l-1}}<\delta k, S_{a_l}\geq k\delta]\leq \frac{c_8 c_5^l l^{\frac{3}{2}}}{\sqrt{k}},
\end{align*}
then
\begin{align}
\pr[\widehat{\nu}> k, S_k>\delta k]\leq \sum_{l=1}^{k}\frac{c_8 c_5^l l^{\frac{3}{2}}}{\sqrt{k}}\leq\frac{c_8}{\sqrt{k}}\sum_{l=1}^{\infty} c_5^l l^{\frac{3}{2}}=\frac{c_9}{\sqrt{k}}.
\end{align}
Then we get that
\begin{align*}
\pr[\widehat{\nu}>k]\leq\frac{c_4+c_9}{\sqrt{k}}.
\end{align*}
\end{proof}   

As an immediate consequence of Proposition \ref{tau(u,v)} we have:

\smallskip

\begin{corollary}\label{tau(0,le)}
Let $u=(0,0),v=(l,0)$ and $\vartheta(u,v):= \inf\{t\geq 0: \pi^u(s)=\pi^v(s)\text{ for all }s\geq
 t\}$. Then there exists a positive constant $C$ such that
\begin{align*}
\pr[\vartheta(u,v)>k]\leq\frac{Cl}{\sqrt{k}}.
\end{align*}
\end{corollary}
\begin{proof}
Put $e_1:=(1,0)$. Since $\{ \vartheta(u,v)>k \} \subset \cup_{i=1}^l \{\vartheta\big((i-1)e_1,ie_1\big) > k\}$ we have that
\begin{align*}
\pr[\vartheta(u,v)>k]\leq\sum_{i=1}^l\pr\Big[\vartheta\big((i-1)e_1,ie_1\big)>k\Big]\leq\frac{Cl}{\sqrt{k}}.
\end{align*}
\end{proof}

\section{Condition $I$.}
\label{sec:I}

In this section we will prove condition $I$ of Theorem \ref{convergence theo}.  
The invariance principle for single paths can be proved analogously to the proof found in \cite{roy2013random} for the Directed Random Forest. All we need is a uniform bound on a moment of order higher than two for the increments of the path on the renewal times. Finally to get condition $I$, we will follow the technique introduced in \cite{coletti2014convergence}. It is based on building a coupling between a finite collection of GRDF paths and a collection of paths which are independent until coalescence, and have the same marginal distributions as those in the GRDF. The difficult again arises from the need to work with the renewal times to construct the coupling. 

\begin{proposition}\label{one path convergence}
There exist positive constants $\gamma$ and $\sigma$ such that for any $u\in\Z^2$ the rescaled path $\pi_{n}^u$, as defined in $(2.2)$, converges in distribution to a Brownian motion starting at u.
\end{proposition}

\begin{proof}
 Without loss of generality we can assume that $u=(0,0)$. To simplify notation we will omit $u$ from it, i.e. we will write $X_m$ instead of $X_m^u$, $\pi(t)$ instead of $\pi^u(t)$ and so on. Taking $T_0:=0,\tau_0:=0$ and $(T_j)_{j\geq 1}$, $(\tau_j)_{j\geq 1}$ as defined in Corollary \ref{renewal time corollary one path}. We introduce the auxiliary path $\widetilde{\pi}$ obtained from the linear interpolation of the values of $(\pi(t))_{t\geq 0}$ on the renewal times,  
\begin{align*}
\widetilde{\pi}(t)&:= \pi(T_j) + \frac{t-T_{j}}{T_{j+1}-T_{j}}\Big[\pi(T_{j+1})-\pi(T_j)\Big], \text{ for } T_j\leq t\leq T_{j+1}.
\end{align*}
Put
\begin{align*}
Y_j&:=\pi(T_j)-\pi(T_{j-1}) \text{ for } j\geq 1.\\
S_0&:=0, S_i:=\sum_{j=1}^iY_j \text{ for } i\geq 1.
\end{align*}
Let $\sigma^2=\var(Y_1)$, then by Donsker's invariance principle we have that $(\widehat{\pi}_{n}(t))_{t\ge 0}$, defined as
\begin{align*}
\widehat{\pi}_{n}(0):= 0\text{, }\widehat{\pi}_{n}(t):= \frac{1}{n\sigma}\Big[(n^{2}t-\lfloor n^{2}t\rfloor)Y_{\lfloor n^{2}t\rfloor+1}+S_{\lfloor n^{2}t \rfloor}\Big] \text{ for } t>0 \, ,
\end{align*}
converges in distribution as $n\rightarrow\infty$ to a standard Brownian motion $(B(t))_{t\ge 0}$. It turns out that $\widetilde{\pi}$ suitably rescaled is a time change of $\widehat{\pi}_{n}$ which will allow us to prove convergence for the former.

Define
\begin{align*}
A(t):= j + \frac{t-T_{j}}{T_{j+1}-T_{j}} \text{ for } T_{j}\leq t < T_{j+1}.
\end{align*} 
Put $N(t): = \sup\{j\geq 1: T_j \leq t\} \text{ for } t>0$ and 
note that  $N(t)\leq A(t)\leq N(t) + 1$ for all $t>0$. Since $T_{j}-T_{j-1}$, $j \ge 1$, are positive valued i.i.d random variables, $(N(t))_{t\ge 0}$ is a renewal process. By the Renewal Theorem $\frac{N(t)}{t}\rightarrow \frac{1}{\esp[T_{1}]}$  almost surely as $t\rightarrow \infty$. Hence for $\gamma=\esp[T_{1}]$ we have that $\frac{A(n^{2}\gamma t)}{n^{2}}\rightarrow t$ almost surely. For $n\geq 1$ let us rescale $\widetilde{\pi}$ as 
\begin{align*}
\widetilde{\pi}_{n}(t):= \frac{\widetilde{\pi}(\gamma n^{2}t)}{n\sigma} \text{ for }t\geq 0.
\end{align*}
Note that
\begin{align*}
 \widetilde{\pi}_{n}(t)= \widehat{\pi}_{n}\Big(\frac{A(n^{2}\gamma t)}{n^{2}}\Big) \text{ for } t\geq 0,
\end{align*}
and from this we get that $(\widetilde{\pi}_{n}(t))_{t \ge 0}$ converges in distribution to a $(B(t))_{t\ge 0}$, we leave the details to the reader. 

To prove the convergence of $(\pi_{n}(t))_{t \ge 0}$ to $(B(t))_{t\ge 0}$ it is enough to show that for any $\epsilon > 0$ and $s>0$, $\pr\big[\sup_{0\leq t \leq s} |\pi_{n}(t)-\widetilde{\pi}_{n}(t)|>\epsilon\big]\rightarrow 0$ as $n\rightarrow \infty$. Since 
\begin{align*}
\Big\{\sup_{0\leq t \leq s}|\pi_{n}(t)-\widetilde{\pi}_{n}(t)|>\epsilon\Big\} & = \Big\{\sup_{0\leq t \leq sn^{2}\gamma}|\pi(t)-\widetilde{\pi}(t)|>\epsilon n\sigma\Big\}\\
&\subset \bigcup_{j=0}^{N(sn^{2}\gamma)}\Big\{\sup_{T_{j}\leq t \leq T_{j+1}}|\pi(t)-\widetilde{\pi}(t)|>\epsilon n\sigma\Big\} 
\end{align*}
and $N(sn^{2}\gamma)\leq C \lfloor sn^{2} \rfloor$ for some $C>0$ for every $s > 0$ and $n\ge 1$, we have that  
 \begin{align*}
  \Big\{\sup_{0\leq t \leq sn^{2}\gamma}|\pi(t)-\widetilde{\pi}(t)|>\epsilon n\sigma\Big\}\subseteq \bigcup_{j=0}^{C \lfloor sn^{2}\rfloor}\Big\{\sup_{T_{j}\leq t \leq T_{j+1}}|\pi(t)-\widetilde{\pi}(t)|>\epsilon n\sigma\Big\}.
    \end{align*}
By definition and construction of the renewal structure, $\pi$ and $\widetilde{\pi}$ coincide at the renewal times and their increments are stationary, then
 \begin{align*}
\pr\Big[\sup_{0\leq t \leq s}|\pi_{n}(t)-\widetilde{\pi}_{n}(t)|>\epsilon\Big]\leq (C \lfloor sn^{2} \rfloor+1)\pr\Big[\sup_{0\leq t \leq T_{1}}\{|\pi(t)-\widetilde{\pi}(t)|\}>\epsilon n\sigma\Big].
\end{align*}  
Since $\pi(T_{1})=\widetilde{\pi}(T_{1})$ and $\sup_{0\le t \le T_1} |\pi(t)-\pi(T_{1})|\leq Z$, $\sup_{0\le t \le T_1}  |\widetilde{\pi}(t)-\pi(T_{1})|\leq Z$, where $Z$ is defined in Lemma \ref{tpr}, then
\begin{align*}
 \pr\Big[\sup_{0\leq t \leq s}|\pi_{n}(t)-\widetilde{\pi}_{n}(t)|>\epsilon\Big]\leq(C\lfloor sn^{2}\rfloor+1)\pr[2Z>\epsilon n\sigma]\leq \frac{2^{3}(C \lfloor sn^{2}\rfloor+1)\esp[Z^{3}]}{\epsilon^{3}\sigma^{3}n^{3}}
 \end{align*}
which converges to zero as $n\rightarrow\infty$.
\end{proof}

\medskip

The next proposition, which is condition I of Theorem \ref{convergence theo}, is the main result of this section. 

\medskip

\begin{proposition}\label{I}		
Let $\mathcal{X}_n$ be defined as in $(\ref{$X_n$})$ where the constants $\gamma$ and $\sigma$ are taken as in Proposition \ref{one path convergence}. Then for any $y_1,\dots,y_m\in\R^2$ there exist paths $\theta^{y_1}_n,\dots,\theta^{y_m}_n$ in $\mathcal{X}_n$, such that $(\theta_n^{y_1},\dots,\theta_n^{y_m})$ converges in distribution as $n\rightarrow\infty$ to coalescing Brownian motions starting at $y_1,\dots,y_m$. 
\end{proposition}

\medskip

To prove Proposition \ref{I} we will use a coupling argument. To build the coupling, we will need Proposition \ref{renewals times version} below, which is a version of Proposition \ref{renewal time proposition} that will be presented without proof because its proof follows the same lines as those of Proposition \ref{renewal time proposition}.

\medskip

\begin{proposition}\label{renewals times version}
 Let $\{U^1_v:v\in\Z^2\}$, $\{U^2_v:v\in\Z^2\}$, $\{W_v^1:v\in\Z^2\}$ and $\{W_v^2:v\in\Z^2\}$ be i.i.d. families  independent of each other such that the $U^j_v$, $j=1,2$, are Uniform random variables on $[0,1]$ and of $W_v^j$, $j=1,2$, are identically distributed positives random variables on $\mathbb{N}$ with finite support. Consider the GRDF systems
 \begin{align*}
  \mathcal{X}^1:=\{\pi^{1,v}:v\in\Z^2\} \text{ and }    \mathcal{X}^2:=\{\pi^{2,v}:v\in\Z^2\}
 \end{align*}
built respectively using the random variables $\big\{\{U^1_v:v\in\Z^2\},\{W_v^1:v\in\Z^2\}\big\}$ and  $\big\{\{U^2_v:v\in\Z^2\},\{W_v^2:v\in\Z^2\}\big\}$. Then for points $u_1^1 \neq \dots \neq u^1_{m_1}$ and  $u_1^2 \neq \dots \neq u_{m_2}^2$ in $\Z^2$ at the same time level, i.e. with equal second component, there exist random variables $T$, $Z$ and $\tau(u_i^j)$ for $j=1,2$, $1\leq i\leq m_j$, such that $T\leq Z$ and
 \begin{enumerate}
 \item[(i)] $\Delta^j_{\tau(u_i^j)}(u_i) = \emptyset$ and $X_{\tau(u_1^1)}^{1,u_1^1}(2)=X_{\tau(u_i^j)}^{j,u_i^j}(2)$ for $j=1,2$ and  $i=1,\dots,m_j$. Where for $j=1,2$ and $v\in\Z^2$ the sequence $\{X^{j,v}_k\}_{k\geq 0}$ is defined as in (\ref{$X$}) using the r.v. $\{U^j_v:v\in\Z^2\}$, $\{W_v^j:v\in\Z^2\}$ and $\{\Delta_k^j(v)\}_{k\geq 0}$ is defined as in (\ref{Delta}) for the sequence $\{X^{j,v}_k\}_{k\geq 0}$.
 
 \item[(ii)] Let $T:= X_{\tau(u_1^1)}^{1,u_1^1}(2)$, we have that its distribution depends on $m_1+m_2$ but not on $u_1^j,\dots,u^j_{m_j}$, $j=1,2$. For all $k\geq 1$ we get $\esp\big[T^k\big]<\infty$. Moreover $\pi^{j,u_i^j}(T) = X_{\tau(u^j_i)}^{j,u^j_i}(1)$ for $j=1,2$, $1\leq i\leq m_j$.
 
 \item[(iii)]For all $j=1,2$ and $1\leq i\leq m_j$ $i=1,\dots,m$ we have that $\sup_{0\le t \le T} |\pi^{j,u^j_i}(t) - u^j_i(1)|\leq Z$ and its distribution depends on $m_1+m_2$ but not on $u_1^j,\dots,u^j_{m_j}$ for $j=1,2$. Also for all $k\geq 1$ we get $\esp\big[Z^k\big]<\infty$.
  \end{enumerate}
 \end{proposition}

\medskip

\begin{proof}[Proof of the Proposition \ref{I}.]
Here we use a non-straightforward adaptation of the idea applied in \cite{coletti2014convergence} to prove condition $I$ for the Drainage Network model. We will consider the case $y_1 = (0,0)$, $y_2 = (1,0)$, ..., $y_m = (m,0)$, the other cases can be carried out in the same way (even when the paths do not start necessarily at the same time). So we are going to prove that for any $m\in\N$,
$$
(\pi^{(0,0)}_n,\pi^{(n\sigma,0)}_n, \dots,\pi^{(mn\sigma,0)}_n)
$$
converges in distribution to a vector of coalescing Brownian motions starting in $(0,0),\dots,(m,0)$ denoted here by $(B^{(0,0)},\dots,B^{(m,0)})$. To simplify the notation we will write $\pi^k := \pi^{(k,0)}$, $k \in \mathbb{Z}$, and $B^x := B^{(x,0)}$ for $x\in\R$. Here for the rescaled paths we use the notation: 
$$
\pi^{k}_n = \frac{\pi^{k \lfloor n \sigma \rfloor}(tn^2\gamma)}{n\sigma} \, .
$$
It is enough to fix an arbitrary $M>0$, suppose that $(B^{0},\dots,B^{m})$ and $(\pi^{0}_n,\pi^{1}_n, \dots,\pi^{m}_n)$ are restricted to time interval $[0,M]$ and prove the convergence, i.e.,
\begin{equation}
\label{eq:limm}
\lim_{n \rightarrow \infty} (\pi^{0}_n(t), \dots,\pi^{m}_n(t))_{0\le t \le M} \overset{d}{=} \big( B^0(t),\dots, B^m(t) \big)_{0\le t \le M} \, .
\end{equation}
By Proposition \ref{one path convergence} we have that 
\begin{align*}
\lim_{n \rightarrow \infty} ( \pi^0(t) )_{0\le t \le M} \overset{d}{=} \big( B^0(t) \big)_{0\le t \le M}.
\end{align*}
The proof will follow from induction in $m$. Let us suppose that 
\begin{align*}
\lim_{n \rightarrow \infty} (\pi^{0}_n(t), \dots,\pi^{(m-1)}_n(t))_{0\le t \le M}\overset{d}{=}\Big(B^0(t),\dots,B^{(m-1)}(t)\Big)_{0\le t \le M}. 
\end{align*} 
The proof of \eqref{eq:limm} from the induction hypothesis will be based on coupling techniques. We will build a path $\overline{\pi}^{m}_n$ which is independent of $(\pi^{0}_n,\pi^{1}_n, \dots,\pi^{(m-1)}_n)$ until coalescence with one of them, has the same distribution of $\pi^{m}$ and such that, in a proper way, $\pi_n^m$ and $\overline{\pi}^m_n$ are close to each other.  

We start constructing paths $\widetilde{\pi}^{0}$, ... , $\widetilde{\pi}^{(m-1)\lfloor n\sigma \rfloor}$ and $\widehat{\pi}^{m \lfloor n\sigma \rfloor}$ that are coupled to $(\pi^{0},\dots,\pi^{(m-1) \lfloor n\sigma \rfloor},\pi^{m \lfloor n\sigma \rfloor})$ in a way that they coincide until one of the latter paths moves a distance at least $n^{\frac{3}{4}}$ from its last position on the last renewal time. The idea is to replace sections of the environment by the same sections of an independent environment and then build the paths $\widetilde{\pi}^{0}$, ... , $\widetilde{\pi}^{(m-1)\lfloor n\sigma \rfloor}$ and $\widehat{\pi}^{m \lfloor n\sigma \rfloor}$ on the new concatenated environment. We suggest the reader to see Figure 6 although some definitions are still missing. The construction follows by induction:

\medskip

\noindent \textbf{Step 1:} Let $\{\widetilde{U}_v:v\in\Z^2\}$ and $\{\widehat{U}_v:v\in\Z^2\}$ be i.i.d. families of uniform random variables in $[0,1]$; $\{\widetilde{W}_v:v\in\Z^2\}$ and $\{\widehat{W}_v:v\in\Z^2\}$ be i.i.d families of random variables with the same distribution of $W_{(0,0)}$; independent of each other and of $\{U_v:v\in\Z^2\}$ and $\{W_v:v\in\Z^2\}$. Using them let us define the random variables $\{\widetilde{U}^1_v:v\in\Z^2\}$, $\{\widehat{U}^1_v:v\in\Z^2\}$, $\{\widetilde{W}^1_v:v\in\Z^2\}$ and $\{\widehat{W}^1_v:v\in\Z^2\}$ as follows:
\begin{equation*}
                \widehat{U}_v^1:=
                  \left\lbrace
                  \begin{array}{l}
                     U_v; \text{ if } |v(1)-mn\sigma|\leq n^{\frac{3}{4}} \text{ and } 0<v(2)\leq n^{\frac{3}{4}}; \\
                     \widehat{U}_v; \text{ otherwise},
                  \end{array}
                  \right. 
\end{equation*}
\begin{equation*}
                \widehat{W}_v^1:=
                  \left\lbrace
                  \begin{array}{l}
                     W_v; \text{ if } |v(1)-mn\sigma|\leq n^{\frac{3}{4}} \text{ and } 0<v(2)\leq n^{\frac{3}{4}}; \\
                     \widehat{W}_v; \text{ otherwise},
                  \end{array}
                  \right. 
\end{equation*}
\begin{equation*}
                \widetilde{W}_v^1:=
                  \left\lbrace
                  \begin{array}{l}
                     W_v; \text{ if } v(1)\leq (m-1)n\sigma + n^{\frac{3}{4}} \text{ and } 0<v(2)\leq n^{\frac{3}{4}} ;\\
                     \widetilde{W}_v; \text{ otherwise},
                  \end{array}
                  \right. 
\end{equation*}
and 
\begin{equation*}
                \widetilde{U}_v^1:=
                  \left\lbrace
                  \begin{array}{l}
                     U_v; \text{ if } v(1)\leq (m-1)n\sigma + n^{\frac{3}{4}} \text{ and } 0<v(2)\leq n^{\frac{3}{4}} ;\\
                     \widetilde{U}_v; \text{ otherwise}.
                  \end{array}
                  \right. 
\end{equation*}
Use $\{\widetilde{U}_v^1:v\in\Z^2\}$, $\{\widetilde{W}^1_v:v\in\Z^2\}$ to construct paths $\{\widetilde{\pi}^{0},\dots,\widetilde{\pi}^{(m-1)\lfloor n\sigma \rfloor}\}$ of the GRDF (not rescaled) starting respectively in $0,\lfloor n\sigma \rfloor,\dots, (m-1) \lfloor n\sigma \rfloor$ at time zero. Also use the families $\{\widehat{U}_v^1:v\in\Z^2\}$, $\{\widehat{W}^1_v:v\in\Z^2\}$ to construct a path $\widehat{\pi}^{m\lfloor n\sigma\rfloor}$ of the GRDF (not rescaled) starting at $m\lfloor n\sigma\rfloor$ at time zero, except for a coalescence rule imposing that when $\widehat{\pi}^{m\lfloor n\sigma\rfloor}$ meets one of $\widetilde{\pi}^{j\lfloor n\sigma \rfloor}$, $0\le j \le m-1$, then they coalesce following both the single path $\widetilde{\pi}^{j\lfloor n\sigma \rfloor}$. Let $T_1$ and $Z_1$ be the random variables associated to $\{\widetilde{\pi}^{0},\dots,\widetilde{\pi}^{(m-1)\lfloor n\sigma \rfloor},\widehat{\pi}^{m \lfloor n\sigma \rfloor} \}$ by Proposition \ref{renewals times version}. Note that on the event $\{Z_1\leq n^{\frac{3}{4}}\}$ the vector paths $(\widetilde{\pi}^{0},\dots,\widetilde{\pi}^{(m-1) \lfloor n\sigma \rfloor},\widehat{\pi}^{m \lfloor n\sigma \rfloor})$ coincide with $(\pi^{0},\dots,\pi^{m \lfloor n\sigma \rfloor})$  up to time $T_1 \le Z_1 \le n^{\frac{3}{4}}$. This ends Step 1.

\smallskip

\noindent \textbf{Step 2:} At time $T_1$ the environment on times $t > T_1$ is not known, so we can use other environmental random variables to extend GRDF paths after time $T_1$. So from time $T_1$, we define new iid families $\{\widetilde{U}^2_v:v\in\Z^2\}$, $\{\widehat{U}^2_v:v\in\Z^2\}$, $\{\widetilde{W}^2_v:v\in\Z^2\}$ and $\{\widehat{W}^2_v:v\in\Z^2\}$ independent of anything else as follows:
\begin{equation*}
                \widehat{U}_v^2:=
                  \left\lbrace
                  \begin{array}{l}
                     U_v; \text{ if } |v(1)-\widehat{\pi}^{1,m\lfloor n\sigma\rfloor}(T_1)|\leq n^{\frac{3}{4}} \text{ and } T_1<v(2)\leq T_1+n^{\frac{3}{4}}; \\
                     \widehat{U}_v; \text{ otherwise},
                  \end{array}
                  \right. 
\end{equation*}
\begin{equation*}
                \widehat{W}_v^2:=
                  \left\lbrace
                  \begin{array}{l}
                     W_v; \text{ if } |v(1)-\widehat{\pi}^{1,m\lfloor n\sigma\rfloor}(T_1)|\leq n^{\frac{3}{4}} \text{ and } T_1<v(2)\leq T_1+n^{\frac{3}{4}}; \\
                     \widehat{W}_v; \text{ otherwise},
                  \end{array}
                  \right. 
\end{equation*}
\begin{equation*}
                \widetilde{W}_v^2:=
                  \left\lbrace
                  \begin{array}{l}
                     W_v; \text{ if } v(1)\leq \max_{0\leq j\leq m-1}\widetilde{\pi}^{1,j\lfloor n\sigma\rfloor}(T_1)+ n^{\frac{3}{4}} \text{ and } T_1<v(2)\leq T_1 + n^{\frac{3}{4}} ;\\
                     \widetilde{W}_v; \text{ otherwise},
                  \end{array}
                  \right. 
\end{equation*}
and 
\begin{equation*}
                \widetilde{U}_v^2:=
                  \left\lbrace
                  \begin{array}{l}
                     U_v; \text{ if } v(1)\leq \max_{0\leq j\leq m-1}\widetilde{\pi}^{1,j\lfloor n\sigma\rfloor}(T_1)+ n^{\frac{3}{4}} \text{ and } T_1<v(2)\leq T_1 + n^{\frac{3}{4}} ;\\
                     \widetilde{U}_v; \text{ otherwise}.
                  \end{array}
                  \right. 
\end{equation*}
First set $\widetilde{\pi}^{2,0},\widetilde{\pi}^{2,\lfloor n\sigma \rfloor},\dots,\widetilde{\pi}^{2,(m-1)\lfloor n\sigma \rfloor}$ starting respectively at $\widetilde{\pi}^{0}(T_1)$, $\widetilde{\pi}^{\lfloor n\sigma \rfloor}(T_1)$, $\dots,$ $\widetilde{\pi}^{(m-1)\lfloor n\sigma \rfloor}(T_1)$ and using the environment $\{\widetilde{U}_v^2:v\in\Z^2\}$ ,$\{\widetilde{W}^2_v:v\in\Z^2\}$. If $\widehat{\pi}^{m\lfloor n\sigma\rfloor}$ has coalesced to $\widetilde{\pi}^{j\lfloor n\sigma \rfloor}$, $0\le j \le m-1$, set $\widehat{\pi}^{2,m\lfloor n\sigma\rfloor} = \widetilde{\pi}^{2,j\lfloor n\sigma \rfloor}$, otherwise $\widehat{\pi}^{m\lfloor n\sigma\rfloor}(T_1) \neq \widetilde{\pi}^{j\lfloor n\sigma \rfloor}(T_1)$ for $0\le j \le m-1$ and set $\widehat{\pi}^{2,m\lfloor n\sigma\rfloor}$ as the GRDF path starting at $\widehat{\pi}^{m\lfloor n\sigma\rfloor}(T_1)$ at time $T_1$ using the environment $\{\widehat{U}_v^2:v\in\Z^2\}$, $\{\widehat{W}^2_v:v\in\Z^2\}$, except for a coalescence rule imposing that when $\widehat{\pi}^{2,m\lfloor n\sigma\rfloor}$ meets one of $\widetilde{\pi}^{2,j \lfloor n\sigma \rfloor}$, $0\le j \le m-1$, then they coalesce following both the single path $\widetilde{\pi}^{2,j \lfloor n\sigma \rfloor}$. Again we have random variables $T_2$ and $Z_2$ for these paths as in Proposition \ref{renewals times version} and on the event $\{\max (Z_1,Z_2) \leq n^{\frac{3}{4}}\}$ the vector $(\widetilde{\pi}^{0},\dots,\widetilde{\pi}^{(m-1) \lfloor n\sigma \rfloor},\widehat{\pi}^{m \lfloor n\sigma \rfloor})$ coincide with $(\pi^{0},\dots,\pi^{m \lfloor n\sigma \rfloor})$  up to time $T_2 \le Z_1 + Z_2 \le 2 n^{\frac{3}{4}}$. Redefine, if necessary, $(\widetilde{\pi}^{0},\widetilde{\pi}^{\lfloor n\sigma \rfloor},\dots,\widetilde{\pi}^{(m-1)\lfloor n\sigma \rfloor})$ as $(\widetilde{\pi}^{2,0},\widetilde{\pi}^{2,\lfloor n\sigma \rfloor},\dots,\widetilde{\pi}^{2,(m-1)\lfloor n\sigma \rfloor})$ on time interval $T_1 < t \le T_2$. This ends Step 2.

\smallskip

We continue step by step replicating recursively Step k from Step k-1. We get $(T_k)_{k\geq 1}$, $(Z_k)_{k\geq 1}$ and $\{\widetilde{\pi}^{k,0},\widetilde{\pi}^{k,\lfloor n\sigma\rfloor},\dots,\widetilde{\pi}^{k,(m-1)\lfloor n\sigma\rfloor},\widehat{\pi}^{k,m\lfloor n\sigma\rfloor}\}$ for $k\geq 1$ such that on the event $\{\max (Z_1,...,Z_k) \leq n^{\frac{3}{4}}\}$ the vector $(\widetilde{\pi}^{0},\dots,\widetilde{\pi}^{(m-1) \lfloor n\sigma \rfloor},\widehat{\pi}^{m \lfloor n\sigma \rfloor})$ coincide with $(\pi^{0},\dots,\pi^{m \lfloor n\sigma \rfloor})$  up to time $T_k \le \sum_{j=1}^k Z_j \le k n^{\frac{3}{4}}$. 

\medskip

Let us define a version $\overline{\pi}^{m\lfloor n\sigma\rfloor}$ of $\widehat{\pi}^{m\lfloor n\sigma\rfloor}$ which is independent of $(\widetilde{\pi}^{0},\dots,\widetilde{\pi}^{(m-1)\lfloor n\sigma\rfloor})$ until coalescence and coincide with $\widehat{\pi}^{m\lfloor n\sigma\rfloor}$ until this latter path gets to a distance of at least $2 n^{3/4}$ of $(\widetilde{\pi}^{0},\dots,\widetilde{\pi}^{(m-1)\lfloor n\sigma\rfloor})$. Consider the following stopping time 
\begin{align*}
    \zeta:=\inf\Big\{k\geq 1: \max_{0\leq j\leq m-1}|\widehat{\pi}^{m\lfloor n\sigma\rfloor}(T_k)-\widetilde{\pi}^{j\lfloor n\sigma\rfloor}(T_k)|\leq 2n^{\frac{3}{4}} \Big\}.
\end{align*}
Define $\overline{\pi}^{m\lfloor n\sigma\rfloor}(t)=\widehat{\pi}^{m\lfloor n\sigma\rfloor}(t)$ for $0\leq t\leq T_\zeta$, see Figure 6. From time $T_\zeta$ and before coalescence with some $\widetilde{\pi}^{0},\dots,\widetilde{\pi}^{(m-1)\lfloor n\sigma\rfloor}$, we have that $\overline{\pi}^{m\lfloor n\sigma\rfloor}(t)$ evolves only through the environment $(\{\widehat{U}_v:v\in\Z^2\},\{\widehat{W}_v:v\in\Z^2\})$ as the path starting in $\widehat{\pi}^{m\lfloor n\sigma\rfloor}(T_{\zeta})$ at time $T_{\zeta}$. By coalescence we mean again that when $\overline{\pi}^{m\lfloor n\sigma\rfloor}(t)$ meets a path one of $\widetilde{\pi}^{j\lfloor n\sigma \rfloor}$, $0\le j \le m-1$, then they coalesce following both the single path $\widetilde{\pi}^{j\lfloor n\sigma \rfloor}$. Let 
\begin{align*}
    \widetilde{\pi}^j_n(t):=\frac{\widetilde{\pi}^{j\lfloor n\sigma\rfloor}(tn^2\gamma)}{n\sigma} \text{ for } j=0,\dots,m-1 \, ,
\end{align*}
\begin{align*}
    \widehat{\pi}^m_n(t):=\frac{\widehat{\pi}^{m\lfloor n\sigma\rfloor}(tn^2\gamma)}{n\sigma} \quad and \quad \overline{\pi}^{m}_n(t):= \frac{\overline{\pi}^{m\lfloor n\sigma\rfloor}(tn^2\gamma)}{n\sigma}
\end{align*}
be the rescaled versions of the constructed paths. 


\medskip
\begin{figure}
\label{fig:acoplamento}
\begin{overpic}[scale = 1]
{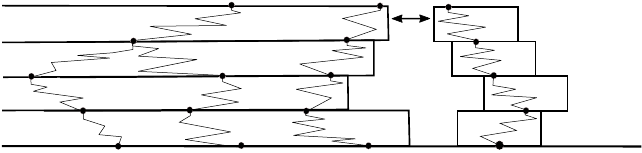}
\put(-5,6){ \tiny{$T_1$}} \put(-5,12){\tiny{ $T_2$}} \put(-5,17){\tiny{ $T_3$}} \put(-5,22){\tiny{ $T_{\zeta}$}}
\put(17,-2.5){\tiny{ $0$}} \put(33,-2.5){ \tiny{$\lfloor n\sigma\rfloor$}}
\put(51,-2.5){ \tiny{$2\lfloor n\sigma\rfloor$}} \put(72,-2.5){\tiny{ $3\lfloor n\sigma\rfloor$}}
\put(60.5,23){ \tiny{$\leq 2n^{3/4}$}}
\end{overpic}
\medskip
\caption{Here m=4 and we consider the GRDF paths $\pi^{0}$, $\pi^{\lfloor n\sigma \rfloor}$, $\pi^{2 \lfloor n\sigma \rfloor}$ and $\pi^{3 \lfloor n\sigma \rfloor}$. Recall that $T_0= 0$. In the picture $\pi^{3 \lfloor n\sigma \rfloor}$ is confined in the union of the drawn rectangles $[\pi^{3 \lfloor n\sigma \rfloor}(T_{j-1}) - n^{\frac{3}{4}},\pi^{3 \lfloor n\sigma \rfloor}(T_{j-1}) + n^{\frac{3}{4}}]\times [T_{j-1},T_j], j=1,\dots,4 $;  so $\pi^{3 \lfloor n\sigma \rfloor}$ remains at distance $n^{\frac{3}{4}}$ of its position on the previous renewal time. The paths $\pi^{0}$, $\pi^{\lfloor n\sigma \rfloor}$, $\pi^{2 \lfloor n\sigma \rfloor}$  are confined in the union of the drawn semi-infinite rectangles 
$[-\infty, \max_{l=0,1,2} \pi^{l \lfloor n\sigma \rfloor}(T_{j-1}) + n^{\frac{3}{4}}]\times [T_{j-1},T_j], j=1,\dots,4$; so
none of $\pi^{0}$, $\pi^{\lfloor n\sigma \rfloor}$, $\pi^{2 \lfloor n\sigma \rfloor}$ go beyond $n^\frac{3}{4}$ to the right of their rightmost position at the previous renewal time. We are also supposing that $T_4$ is the first renewal time such that $|\max_{l=0,1,2} \pi^{l \lfloor n\sigma \rfloor}(T_{j-1}) - \pi^{3 \lfloor n\sigma \rfloor}(T_{j-1})| \le 2 n^{\frac{3}{4}}$, thus $\zeta = 4$ and before time $T_4$ we have that $(\pi^{0},\pi^{\lfloor n\sigma \rfloor},\pi^{2 \lfloor n\sigma \rfloor},\pi^{3 \lfloor n\sigma \rfloor})$ coincide with  $(\widetilde{\pi}^{0},\widetilde{\pi}^{\lfloor n\sigma \rfloor},\widetilde{\pi}^{2 \lfloor n\sigma \rfloor},\overline{\pi}^{3 \lfloor n\sigma \rfloor})$.}
\end{figure}

\bigskip

\begin{remark}\label{coupling details} We point out that as a direct consequence of the definitions the following properties are satisfied:
\begin{enumerate}
\item[(i)] Before coalescence, the path $\overline{\pi}^{m}_n$ is independent of $\widetilde{\pi}^0_n,\dots,\widetilde{\pi}^{(m-1)}_n$.
\item[(ii)] For $s\le M$, on the event 
$$
\mathcal{A}_{n,s}:=\{T_{\zeta}>n^2\gamma s\},
$$ 
we have that $\widehat{\pi}_n^{m} (t) =\overline{\pi}_n^{m} (t)$ for every $0\le t \le s$.
\item[(iii)] From the induction hypothesis, item $(i)$, Proposition \ref{one path convergence} and the convergence result for coalescing random walks in \cite{newman2005convergence} we get
\begin{align*}
\lim_{n\rightarrow\infty}\big(\widetilde{\pi}^0_n,\dots,\widetilde{\pi}^{(m-1)}_n,\overline{\pi}^{m}_n\big)\overset{d}{=}\big(B^0,\dots,B^m\big).
\end{align*} 
We leave the details of this verification to the reader, but the point here is that, by the conditioning on the events that bounds the sizes of jumps by $n^{3/4}$ and whose probabilities converge to one as $n\rightarrow \infty$, we have that $\overline{\pi}^{m}_n$ will coalesce with $\widetilde{\pi}_n^{(m-1)}$ before $\widetilde{\pi}^{j}_n$, $1\le j \le m-2$, with probability that also goes to one. Moreover, the independence from $(i)$, Proposition \ref{one path convergence} and the convergence result for coalescing random walks in \cite{newman2005convergence} implies that $(\widetilde{\pi}_n^{(m-1)},\overline{\pi}^{m}_n)$ converges to a pair of coalescing Brownian motions. To finish, we only need to use the induction hypothesis.
\item[(iv)] On the event 
$$
\mathcal{B}_{n,M}:=\cap_{k=1}^{\lfloor Mn^2\gamma \rfloor +1}\{Z_k\leq n^{\frac{3}{4}}\}
$$ 
the vector of paths  $(\widetilde{\pi}^{0},\dots,\widetilde{\pi}^{(m-1)},\widehat{\pi}^{m})$ coincide with $(\pi^{0},\dots,\pi^{m})$  up to a time greater than $Mn^2\gamma$.
\item[(v)] Also on $\mathcal{B}_{n,M}$, if $|\widehat{\pi}^{m\lfloor n\sigma\rfloor}(t)-\widetilde{\pi}^{j\lfloor n\sigma\rfloor}(t)| \le 2n^{\frac{3}{4}}$ for some $0 \le j \le m-1$ and $t > 0$ then either there exist some $k$ such that $T_k < t$
and $\zeta \le k$ or $\widehat{\pi}^{m\lfloor n\sigma\rfloor}$ and $\widetilde{\pi}^{j\lfloor n\sigma\rfloor}$ cannot coalesce or cross each other before or at time $s = \max \{T_k: \, T_k < t\}$. This follows from the fact that on $\mathcal{B}_{n,M}$ the paths do not move by more than $n^{\frac{3}{4}}$ between renewal times, therefore if $\widehat{\pi}^{m\lfloor n\sigma\rfloor}$ and $\widetilde{\pi}^{j\lfloor n\sigma\rfloor}$ coalesce or cross each other, then they will be at distance at most $2n^{\frac{3}{4}}$ on the next renewal time.
\end{enumerate}
\end{remark}  

\smallskip

\begin{claim}\label{B_{n,M}}
For the event $\mathcal{B}_{n,M}$ as in Remark \ref{coupling details} we have that $\lim_{n \rightarrow \infty} \pr\big[\mathcal{B}_{n,M}^c\big]=0$.
\end{claim}
\begin{proof}
Note that
\begin{align*} 
\pr\big[\mathcal{B}_{n,M}^c\big]\leq (Mn^2\gamma +1)\pr\big[Z_1>n^{\frac{3}{4}}\big] \leq \frac{(Mn^2\gamma +1)\esp[Z_1^4]}{n^{3}}   
\end{align*}
which goes to zero as $n$ goes to infinity.
\end{proof}

\smallskip

Let $C([0,M],\mathbb{R}^{m+1})$ be the space of continuous $\mathbb{R}^{m+1}$-valued functions with domain $[0,M]$ endowed with the uniform topology, and fix an uniformly continuous function $H:C([0,M],\mathbb{R}^{m+1})\rightarrow \R$. We need to prove that 
\begin{align*}
\lim_{n \rightarrow \infty}  \esp\big[H\big(\pi^0_n,\dots,\pi^{m}_n\big)\big]= \esp\big[H\big(B^0,\dots,B^m\big)\big]. 
\end{align*}
By (iv) in Remark \ref{coupling details} and Claim \ref{B_{n,M}} we have that
\begin{align}\label{eq:Hconv1}
    \esp\Big[\big|H\big(\pi^0_n,\dots,\pi^{m}_n\big)-H\big(\widetilde{\pi}^0_n,\dots,\widetilde{\pi}^{(m-1)}_n,\widehat{\pi}^{m}_n\big)\big|\Big]\leq 2||H||_{\infty}\pr\big[\mathcal{B}^c_{n,M}\big]\rightarrow 0 
\end{align}
as n goes to infinity. 
By (iii) in Remark \ref{coupling details} and the induction hypothesis we have that 
\begin{align}\label{eq:Hconv2}
    \esp\Big[H\big(\widetilde{\pi}^0_n,\dots,\widetilde{\pi}^{(m-1)}_n,\overline{\pi}^{m}_n\big)\Big]\rightarrow \esp\Big[H\big(B^0,\dots,B^m\big)\Big]. 
\end{align}
Now write
\begin{align*}
    &\Big|\esp\big[H\big(\pi^0_n,\dots,\pi^{m}_n\big)\big]-\esp\big[H\big(B^0,\dots,B^m\big)\big]\Big|\\&\leq \esp\Big[\big|H\big(\pi^0_n,\dots,\pi^{m}_n\big)-H\big(\widetilde{\pi}^0_n,\dots,\widetilde{\pi}^{(m-1)}_n,\widehat{\pi}^{m}_n\big)\big|\Big]\\&+\esp\Big[\big|H\big(\widetilde{\pi}^0_n,\dots,\widetilde{\pi}^{(m-1)}_n,\widehat{\pi}^{m}_n\big)-H\big(\widetilde{\pi}^0_n,\dots,\widetilde{\pi}^{(m-1)}_n,\overline{\pi}^{m}_n\big)\big|\Big] \\&+ \Big|\esp\Big[H\big(\widetilde{\pi}^0_n,\dots,\widetilde{\pi}^{(m-1)}_n,\overline{\pi}^{m}_n\big)-H\big(B^0,\dots,B^m\big)\Big]\Big| \, ,
\end{align*}
then, from \eqref{eq:Hconv1} and \eqref{eq:Hconv2}, it is enough to prove that 
\begin{align*}
\lim_{n \rightarrow \infty}  \esp\Big[\big|H\big(\widetilde{\pi}^0_n,\dots,\widetilde{\pi}^{(m-1)}_n,\widehat{\pi}^{m}_n\big)-H\big(\widetilde{\pi}^0_n,\dots,\widetilde{\pi}^{(m-1)}_n,\overline{\pi}^{m}_n\big)\big|\Big] = 0 \, .
\end{align*}
Using (ii) and (iv) in Remark \ref{coupling details} we obtain
\begin{align*} &\esp\Big[\big|H\big(\widetilde{\pi}^0_n,\dots,\widetilde{\pi}^{(m-1)}_n,\widehat{\pi}^{m}_n\big)-H\big(\widetilde{\pi}^0_n,\dots,\widetilde{\pi}^{(m-1)}_n,\overline{\pi}^{m}_n\big)\big|\Big]\\&= \esp\Big[\big|H\big(\widetilde{\pi}^0_n,\dots,\widetilde{\pi}^{(m-1)}_n,\widehat{\pi}^{m}_n\big)-H\big(\widetilde{\pi}^0_n,\dots,\widetilde{\pi}^{(m-1)}_n,\overline{\pi}^{m}_n\big)\big|\mathbbm{1}_{\mathcal{A}_{n,M}^c}\Big]\\&\leq \esp\Big[\big|H\big(\pi^0_n,\dots,\pi^{(m-1)}_n,\pi^{m}_n\big)-H\big(\pi^0_n,\dots,\pi^{(m-1)}_n,\overline{\pi}^{m}_n\big)\big|\mathbbm{1}_{\mathcal{A}_{n,M}^c}\mathbbm{1}_{\mathcal{B}_{n,M}}\Big] +2||H||_{\infty}\pr\big[\mathcal{B}_{n,M}^c\big]. 
\end{align*}
Again, by Claim \ref{B_{n,M}} we just have to prove that
\begin{align*}
\lim_{n \rightarrow \infty} \esp\Big[\big|H\big(\pi^0_n,\dots,\pi^{(m-1)}_n,\pi^{m}_n\big)-H\big(\pi^0_n,\dots,\pi^{(m-1)}_n,\overline{\pi}^{m}_n\big)\big|\mathbbm{1}_{\mathcal{A}_{n,M}^c}\mathbbm{1}_{\mathcal{B}_{n,M}}\Big] = 0 \, .
\end{align*}
In order to stablish the above convergence, we need to define some stopping times. For $j=\{0,\dots,m-1\}$ consider \begin{align*}
    \zeta_j:=\inf\{k\geq 1: |\pi^{j\lfloor n\sigma\rfloor}(T_k)-\pi^{m\lfloor n\sigma\rfloor}(T_k)|\leq 2n^{\frac{3}{4}}\} \, ,
\end{align*}
where the definition is based on (v) in Remark \ref{coupling details} which implies that on $\mathcal{B}_{n,M}$ we only need to consider approximation between paths on the renewal times.
Then 
\begin{align*}
   & \esp\Big[\big|H\big(\pi^0_n,\dots,\pi^{(m-1)}_n,\pi^{m}_n\big)-H\big(\pi^0_n,\dots,\pi^{(m-1)}_n,\overline{\pi}^{m}_n\big)\big|\mathbbm{1}_{\mathcal{A}_{n,M}^c}\mathbbm{1}_{\mathcal{B}_{n,M}}\Big]\\&\leq\sum_{j=0}^{m-1} \esp\Big[\big|H\big(\pi^0_n,\dots,\pi^{(m-1)}_n,\pi^{m}_n\big)-H\big(\pi^0_n,\dots,\pi^{(m-1)}_n,\overline{\pi}^{m}_n\big)\big|\mathbbm{1}_{\mathcal{A}_{n,M}^c}\mathbbm{1}_{\mathcal{B}_{n,M}}\mathbbm{1}_{\{\zeta=\zeta_j\}}\Big].
\end{align*}
Given $\epsilon>0$, since $H$ is uniformly continuous, there exists $\delta_{\epsilon}>0$ such that: if
$\|f-g\|_{\infty} \leq \delta_{\epsilon}$, for $f, \, g \in C([0,M],\mathbb{R}^{m+1})$, then $\big|H(f)-H(g)\big|\leq \epsilon$. So, if 
$$
\sup_{0\leq t\leq M}|\pi^{m}_n(t)-\overline{\pi}^{m}_n(t)|\leq \delta_{\epsilon}
$$
we get
$$
\Big| H\big(\pi^0_n,\dots,\pi^{(m-1)}_n,\pi^{m}_n\big)-H\big(\pi^0_n,\dots,\pi^{(m-1)}_n,\overline{\pi}^{m}_n\big)\Big|\leq \epsilon. 
$$
 To simplify the notation let us denote $D_{n,j}:=\mathcal{A}_{n,M}^c\cap\mathcal{B}_{n,M}\cap\{\zeta=\zeta_j\}$. For $j=0,\dots,m-1$ we have that
\begin{align*}
    &\esp\Big[\big|H\big(\pi^0_n,\dots,\pi^{(m-1)}_n,\pi^{m}_n\big)-H\big(\pi^0_n,\dots,\pi^{(m-1)}_n,\overline{\pi}^{m}_n\big)\big|\mathbbm{1}_{D_{n,j}}\Big]\\& \leq\epsilon +2||H||_{\infty}\pr\Big[D_{n,j}\cap\{\sup_{0\leq t\leq M}|\pi^{m}_n(t)-\overline{\pi}^{m}_n(t)|>\delta_{\epsilon}\}\Big]\\&=\epsilon +2||H||_{\infty}\pr\Big[D_{n,j}\cap\{\sup_{0\leq t\leq Mn^2\gamma}|\pi^{m\lfloor n\sigma\rfloor}(t)-\overline{\pi}^{m\lfloor n\sigma\rfloor}(t)|>n\sigma\delta_{\epsilon}\}\Big].
\end{align*}
To control the uniform distance between $\pi^{m\lfloor n\sigma\rfloor}$ and $\overline{\pi}^{m\lfloor n\sigma\rfloor}$ on $[0,Mn^2\gamma]$ given the event $D_{n,j}$, we use the fact that both processes will coalesce with $\pi^{j \lfloor n\sigma\rfloor}$ on an interval of time after $T_\zeta$ that is negligible under diffusive scaling. So we need to introduce the random times
$$
\tau^j:=\inf\{t>0: \pi^{j\lfloor n\sigma\rfloor}(s)=\pi^{m\lfloor n\sigma\rfloor}(s), \ \forall s\geq t\}
$$
and 
$$
\overline{\tau}^j:=\inf\{t>0: \pi^{j\lfloor n\sigma\rfloor}(s)=\overline{\pi}^{m\lfloor n\sigma\rfloor }(s), \ \forall s\geq t\}.
$$
for $j=0,\dots,m-1$. Fix some $\beta\in(\frac{3}{2},2)$. Then for $j=0,\dots,m-1$ and $n$ large enough 
$$
\pr\Big[D_{n,j}\cap\{\sup_{0\leq t\leq Mn^2\gamma}|\pi^{m\lfloor n\sigma\rfloor}(t)-\overline{\pi}^{m\lfloor n\sigma\rfloor}(t)|>n\sigma\delta_{\epsilon}\}\Big]
$$    
is bounded above by        
\begin{align}   
\label{eq:nujs}
    &\pr\Big[D_{n,j}\cap\{\sup_{0\leq t\leq Mn^2\gamma}|\pi^{m\lfloor n\sigma\rfloor}(t)-\overline{\pi}^{m\lfloor n\sigma\rfloor}(t)|>n\sigma\delta_{\epsilon}\}\cap\{\tau^j,\overline{\tau}^j\in [T_{\zeta},T_{\zeta}+n^{\beta}\gamma]\}\Big] \nonumber \\
    &+\pr\Big[D_{n,j}\cap\{\tau^j> T_{\zeta}+n^{\beta}\gamma]\}\Big]+\pr\Big[D_{n,j}\cap\{\overline{\tau}^j> T_{\zeta}+n^{\beta}\gamma]\}\Big].
\end{align}
The first probability in \eqref{eq:nujs} is equal to 
$$
\pr\Big[D_{n,j}\cap\{\sup_{T_{\zeta_j}\leq t\leq Mn^2\gamma\wedge (T_{\zeta_j}+n^{\beta}\gamma)}|\pi^{m}_n(t)-\overline{\pi}^{m}_n(t)|>n\sigma\delta_{\epsilon}\}\cap\{\tau^j,\overline{\tau}^j\in [T_{\zeta_j},T_{\zeta_j}+n^{\beta}\gamma]\}\Big]
$$
which is bounded above by
\begin{align*}
&\pr\Big[D_{n,j}\cap\{\sup_{T_{\zeta_j}\leq t\leq Mn^2\gamma\wedge (T_{\zeta_j}+n^{\beta}\gamma)}|\pi^{m}_n(t)-\pi^{m}_n(T_{\zeta_j})|>\frac{n\sigma\delta_{\epsilon}}{2}\}\cap\{\tau^j,\overline{\tau}^j\in [T_{\zeta_j},T_{\zeta_j}+n^{\beta}\gamma]\}\Big]\\&+\pr\Big[D_{n,j}\cap\{\sup_{T_{\zeta_j}\leq t\leq Mn^2\gamma\wedge (T_{\zeta_j}+n^{\beta}\gamma)}|\overline{\pi}^{m}_n(t)-\overline{\pi}^{m}_n(T_{\zeta_j})|>\frac{n\sigma\delta_{\epsilon}}{2}\}\cap\{\tau^j,\overline{\tau}^j\in [T_{\zeta_j},T_{\zeta_j}+n^{\beta}\gamma]\}\Big] \\&
\le \pr\Big[\sup_{0\leq t\leq n^{\beta}\gamma}|\pi^{m\lfloor n\sigma\rfloor}(t)-\pi^{m\lfloor n\sigma\rfloor}(0)|>\frac{n\sigma\delta_{\epsilon}}{2}\Big]+\pr\Big[\sup_{0\leq t\leq n^{\beta}\gamma}|\overline{\pi}^{m\lfloor n\sigma\rfloor}(t)-\overline{\pi}^{m\lfloor n\sigma\rfloor}(0)|>\frac{n\sigma\delta_{\epsilon}}{2}\Big]
\, ,
\end{align*} 
where for the inequality we have used the Markov property on the renewal times. Both terms in the right hand side of the previous inequality are bounded above by
$$
\pr\Big[\sup_{0\leq t\leq n^{\beta}\gamma}\frac{|\pi^0(t)-\pi^0(0)|}{\sigma n^{\frac{\beta}{2}}}> \frac{n^{1-\frac{\beta}{2}}\delta_{\epsilon}}{2}\Big] \, ,
$$
which, by the choice of $\beta < 2$ and the invariance principle proved in Proposition \ref{one path convergence}, converges to zero as $n \rightarrow \infty$. Thus the first probability in \eqref{eq:nujs} converges to zero as $n \rightarrow \infty$. 

It remains to deal with the second and third terms in \eqref{eq:nujs}. Since
$$
|\pi^{j\lfloor n\sigma\rfloor}(T_{\zeta_j})-\pi^{m\lfloor n\sigma\rfloor}(T_{\zeta_j})|\leq 2n^{\frac{3}{4}}
$$
on $D_{n,j}$, then by Corollary \ref{tau(0,le)} there is some constant $C>0$ such that
\begin{align*}
    \pr\Big[D_{n,j}\cap\{\tau^j> T_{\zeta_j}+n^{\beta}\gamma\}\Big]\leq\frac{2C n^{\frac{3}{4}}}{n^{\frac{\beta}{2}}}
\end{align*}
which converges to zero as $n \rightarrow \infty$ because $\beta > 3/2$.
Even though $\pi^{j\lfloor n \sigma \rfloor}$ and $\overline{\pi}^{j\lfloor n \sigma \rfloor}$ are independent from time $T_{\zeta_j}$ until coalescence, we can prove the result stated in Corollary \ref{tau(0,le)} for these paths, following (in a simpler way) the same lines of the proof of that corollary. Thus we get a constant $\overline{C}$ such that 
$$
\pr\Big[D_{n,j}\cap\{\overline{\tau}^j> T_{\zeta_j}+n^{\beta}\gamma\}\Big]\leq\frac{2\overline{C} n^{\frac{3}{4}}}{n^{\frac{\beta}{2}}} \, ,
$$
which as before converges to zero as $n \rightarrow \infty$.

Hence
$$
\lim_{n \rightarrow \infty} \pr\Big[\mathcal{A}_{n,M}^c\cap\mathcal{B}_{n,M}\cap\{\zeta=\zeta_j\}\cap\{\sup_{0\leq t\leq Mn^2\gamma}|\pi^{m}_n(t)-\overline{\pi}^{m}_n(t)|>n\sigma\delta_{\epsilon}\}\Big] = 0 
$$
which finishes the proof.
\end{proof}

\bigskip

\section{Condition $B$} 
\label{sec:B}

We prove condition $B$ of Theorem \ref{convergence theo} at the end of this section. Before we prove it we need to introduce some definitions and establish some preliminary results. 

Recall that we are supposing $\pr[W_{(0,0)}\leq K]=1$, so $K$ is fixed in the definition of the model and the reader should keep it in mind since many objects in this section will depend on $K$. For $u=(u(1),u(2))\in\Z^2$ define the box on $\Z^2$
\begin{align*}
    \Gamma(u) :=\{u(1),\dots,u(1)+K-1\}\times\{u(2)-K+1,\dots,u(2)\} \, .
\end{align*}
We say that the box $\Gamma(u)$ is good if $W_v=1$ and $v$ is open for all $v\in \Gamma(u)$.
\begin{remark}\label{cross C_K}
We point out that when $\Gamma(u)$ is good then no path of the GRDF crosses it from left to right or right to left touching either $(-\infty,u(1)-1]\times\{u(2)\}$ or $[u(1)+K,\infty)\times\{u(2)\}$.  We will use this property to control the number of paths in the GRDF that cross some interval on a given time. This will be fundamental to obtain estimates on the distribution of the counting variables $\eta_{\mathcal{X}_n}(t_0,t,a,b)$.
\end{remark}
Let us define the following random variables
\begin{align*}
    g^+(u):=\inf\{n\geq 1: \Gamma(u+(n-1)Ke_1)\text{ is good}\},
\end{align*}
and 
\begin{align*}
    g^-(u):=\inf\{n\geq 1: \Gamma(u-(nK-1)e_1)\text{ is good}\}.
\end{align*}
Therefore $\Gamma\big(u+g^+(u)Ke_1\big)$ is the first translation of $\Gamma(u)$ to the right of $u$, by multiples of $K$, that is good; and $\Gamma\big(u-(g^-(u)K-1)e_1\big)$ is the first translation of $\Gamma(u)$ to the left of $u$, by multiples of $K$, that is good.

\smallskip

The first lemma below allow us to consider the counting variables $\eta_{\mathcal{X}_n}(t_0,t,a,b)$ only for $t_0 \in \Z$.

\smallskip

\begin{lemma}
\label{lemma:condB1}
Take $a<b\in\R$, $\mathcal{X}_n$ as defined in (\ref{$X_n$}) with $\gamma$ and $\sigma$ as in Proposition \ref{one path convergence} and $\eta_{\mathcal{X}_n}(t_0,t,a,b)$ as in Theorem \ref{convergence theo}. Then for all $\epsilon>0$ there exits a constant $M_{\epsilon}$, not depending on $a$, $b$, $\gamma$ and $\sigma$, such that
\begin{align*}
\pr\big[|\eta_{\mathcal{X}_n}(t_0,t,a,b)|> 1\big]\leq\pr\big[|\eta_{\mathcal{X}}(0,n^2\gamma t,n\sigma a-M_{\epsilon},n\sigma b+M_{\epsilon} )|> 1\big]+\epsilon
\end{align*}
 for all $t_{0}\in\R$, $t>0$ and $n\geq 1$. 
\end{lemma}

\begin{proof} 
Note that any path that crosses $[n\sigma a,n\sigma b]\times\{n^2\gamma t_0\}$ also crosses the interval
\begin{align*}
\Big[n\sigma a-Kg^-\big((\lfloor n\sigma a\rfloor,\lfloor n^2\gamma t_0\rfloor+1)\big),n\sigma b+Kg^+\big((\lfloor n\sigma b\rfloor+1,\lfloor n^2\gamma t_0\rfloor+1)\big)\Big]\times\{\lfloor n^2\gamma t_0\rfloor+1\}.
\end{align*}

Denote $a_n:=n \sigma a$, $b_n:=n \sigma b$. Then 
 \begin{align*}
 &\pr[ |\eta_{\mathcal{X}_n}(t_0,t,a,b)|>1]\\
 &=\pr\Big[|\eta_{\mathcal{X}}(n^2\gamma t_0,n^2\gamma t,a_n,b_n)|> 1]\\
 &\leq\pr\Big[\big|\eta_{\mathcal{X}}\big(\lfloor n^2\gamma t_0\rfloor,n^2\gamma t, a_n-Kg^-\big((\lfloor a_n\rfloor,\lfloor n^2\gamma t_0\rfloor+1)\big),b_n+Kg^+\big((\lfloor b_n\rfloor+1,\lfloor n^2\gamma t_0\rfloor+1)\big)\big)\big|>1\Big].
 \end{align*}
 Take $M_{\epsilon}$ big enough such that
 \begin{align*}
 \pr\Big[Kg^-\big((\lfloor a_n\rfloor,\lfloor n^2\gamma t_0\rfloor+1)\big)>M_{\epsilon}]=\pr\Big[Kg^+\big((\lfloor b_n\rfloor+1,\lfloor n^2\gamma t_0\rfloor+1)\big)>M_{\epsilon}\Big]\leq\frac{\epsilon}{2}.
 \end{align*}
 Then by translation invariance we get
 \begin{align*}
 \pr[|\eta_{\mathcal{X}_n}(t_0,t,a,b)|>1]&\leq \pr\big[|\eta_{\mathcal{X}}(\lfloor n^2\gamma t_0\rfloor,n^2\gamma t, a_n-M_{\epsilon},b_n+M_{\epsilon})|>1\big]+\epsilon\\&=\pr\big[|\eta_{\mathcal{X}}(0,n^2\gamma t, a_n-M_{\epsilon},b_n+M_{\epsilon})|>1\big]+\epsilon.
 \end{align*}
 \end{proof}
 
 \smallskip

Our next result says that the number of paths in $\mathcal{X}$, starting before time $t$, that cross a finite length interval at time $t$, have finite absolute moment of any order. 

\begin{lemma}\label{paths in [a,b]}  Let us define $\mathcal{X}^{t^-}$  as the set of paths in $\mathcal{X}$ that start before or at time $t$ and denote $\mathcal{X}^{t-}(t) = \{ \pi(t) : \pi \in \mathcal{X}^{t^-} \}$. Then we have that 
\begin{align*}
\esp\Big[\big|\mathcal{X}^{t^-}(t)\cap[a,b]\big|^k\Big]<\infty,
\end{align*}
for $a<b\in\R$ and $k\geq 1$.
\end{lemma}

\begin{proof}
We will assume that $t=a=0$ and $b=1$. The proof in the general case is analogous. For $j\in\Z$ let us define $\zeta_j:=\inf\big\{n\geq 0:\sum_{i=0}^{n}\mathbbm{1}_{\{(j,-i)\text{ is open}\}}=K\big\}$ and the random region $D$ as
 \begin{align*}
     D:=\big\{v\in\Z^2: -Kg^-((1,0))\leq v(1)\leq 1+Kg^+((1,0)) \text{ and } -\zeta_{v(1)}\leq v(2)\leq 0 \big\}.
 \end{align*}
In Figure 7 we show a possible realization of $D$.

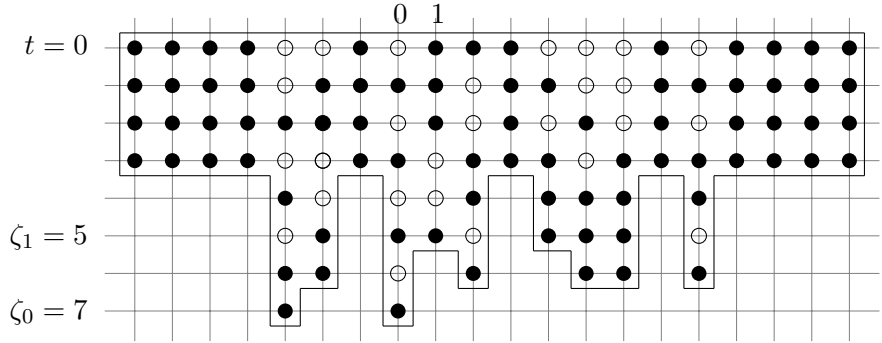
\begin{figure}[H]\label{D}
\begin{tikzpicture}
\draw [step= 0.5 cm, gray, very thin] (-0.4,-0.4) grid (9.9,3.9);

\fill[black](0,3.5) circle (1mm) ; \fill[black](0,3) circle (1mm) ;
\fill[black](0,2.5) circle (1mm) ; \fill[black](0,2) circle (1mm) ; 

\fill[black](0.5,3.5) circle (1mm) ; \fill[black](0.5,3) circle (1mm) ;
\fill[black](0.5,2.5) circle (1mm) ; \fill[black](0.5,2) circle (1mm) ;

\fill[black](1,3.5) circle (1mm) ; \fill[black](1,3) circle (1mm) ;
\fill[black](1,2.5) circle (1mm) ; \fill[black](1,2) circle (1mm) ;

\fill[black](1.5,3.5) circle (1mm) ; \fill[black](1.5,3) circle (1mm) ;
\fill[black](1.5,2.5) circle (1mm) ; \fill[black](1.5,2) circle (1mm) ;

\draw (2,3.5) circle (1mm); \draw (2,3) circle (1mm);
 \fill[black](2,2.5) circle (1mm) ; \draw (2,2) circle (1mm);
 \fill[black](2,1.5) circle (1mm) ;\draw (2,1) circle (1mm);
\fill[black](2,0.5) circle (1mm) ; \fill[black](2,0) circle (1mm) ;

\draw (2.5,3.5) circle (1mm); \fill[black](2.5,3) circle (1mm) ; 
\fill[black](2.5,2.5) circle (1mm) ; \draw (2.5,2) circle (1mm);
\draw (2.5,2.5) circle (1mm); \draw (2.5,2) circle (1mm);
\draw (2.5,1.5) circle (1mm); \fill[black](2.5,1) circle (1mm) ;
\fill[black](2.5,0.5) circle (1mm) ; 

\fill[black](3,3.5) circle (1mm) ; \fill[black](3,3) circle (1mm) ;
\fill[black](3,2.5) circle (1mm) ; \fill[black](3,2) circle (1mm) ;

\draw (3.5,3.5) circle (1mm);\fill[black](3.5,3) circle (1mm) ;
\draw (3.5,2.5) circle (1mm);\fill[black](3.5,2) circle (1mm) ;
\draw (3.5,1.5) circle (1mm); \fill[black](3.5,1) circle (1mm) ; 
\draw (3.5,0.5) circle (1mm);\fill[black](3.5,0) circle (1mm) ;

\fill[black](4,3.5) circle (1mm) ; \fill[black](4,3) circle (1mm) ;
\fill[black](4,2.5) circle (1mm) ; \draw (4,2) circle (1mm);
\draw (4,1.5) circle (1mm); \fill[black](4,1) circle (1mm) ;

\fill[black](4.5,3.5) circle (1mm) ; \draw (4.5,3) circle (1mm);
\draw (4.5,2.5) circle (1mm); \fill[black](4.5,2) circle (1mm) ;
\fill[black](4.5,1.5) circle (1mm) ; \draw (4.5,1) circle (1mm);
\fill[black](4.5,0.5) circle (1mm) ; 

\fill[black](5,3.5) circle (1mm) ; \fill[black](5,3) circle (1mm) ;
\fill[black](5,2.5) circle (1mm) ; \fill[black](5,2) circle (1mm) ;

\draw (5.5,3.5) circle (1mm); \fill[black](5.5,3) circle (1mm) ; 
\draw (5.5,2.5) circle (1mm); \fill[black](5.5,2) circle (1mm) ;
\fill[black](5.5,1.5) circle (1mm) ; \fill[black](5.5,1) circle (1mm) ;

\draw (6,3.5) circle (1mm); \draw (6,3) circle (1mm);
 \fill[black](6,2.5) circle (1mm) ; \draw (6,2) circle (1mm); 
 \fill[black](6,1.5) circle (1mm) ;\fill[black](6,1) circle (1mm) ;
 \fill[black](6,0.5) circle (1mm) ;

\draw (6.5,3.5) circle (1mm); \draw (6.5,3) circle (1mm);
\draw (6.5,2.5) circle (1mm);\fill[black](6.5,2) circle (1mm) ; 
\fill[black](6.5,1.5) circle (1mm) ; \fill[black](6.5,1) circle (1mm) ;
\fill[black](6.5,0.5) circle (1mm) ;

\fill[black](7,3.5) circle (1mm) ; \fill[black](7,3) circle (1mm) ;
\fill[black](7,2.5) circle (1mm) ; \fill[black](7,2) circle (1mm) ;

\draw (7.5,3.5) circle (1mm); \fill[black](7.5,3) circle (1mm) ; 
\draw (7.5,2.5) circle (1mm); \fill[black](7.5,2) circle (1mm) ;
\fill[black](7.5,1.5) circle (1mm) ; \draw (7.5,1) circle (1mm);
\fill[black](7.5,0.5) circle (1mm) ; 

\fill[black](8,3.5) circle (1mm) ; \fill[black](8,3) circle (1mm) ;
\fill[black](8,2.5) circle (1mm) ; \fill[black](8,2) circle (1mm) ; 

\fill[black](8.5,3.5) circle (1mm) ; \fill[black](8.5,3) circle (1mm) ;
\fill[black](8.5,2.5) circle (1mm) ; \fill[black](8.5,2) circle (1mm) ;

\fill[black](9,3.5) circle (1mm) ; \fill[black](9,3) circle (1mm) ;
\fill[black](9,2.5) circle (1mm) ; \fill[black](9,2) circle (1mm) ;

\fill[black](9.5,3.5) circle (1mm) ; \fill[black](9.5,3) circle (1mm) ;
\fill[black](9.5,2.5) circle (1mm) ; \fill[black](9.5,2) circle (1mm) ;

\draw (3.3,3.7) node[above right]{$0$};
\fill[black](3.8,3.7) node[above right]{$1$};
\draw (-1.8,0.3) node[below right]{$\zeta_{0}=7$};
\draw (-1.8,1.3) node[below right]{$\zeta_{1}=5$};
\draw (-1.6,3.8) node[below right]{$t=0$};

\draw[-,black] (-0.2,3.7) -- (-0.2,1.8); \draw[-,black] (-0.2,1.8) -- (1.8,1.8);
\draw[-,black] (1.8,1.8) -- (1.8,-0.2); \draw[-,black] (1.8,-0.2) -- (2.2,-0.2);
\draw[-,black] (2.2,-0.2) -- (2.2,0.3); \draw[-,black] (2.2,0.3) -- (2.7,0.3);
\draw[-,black] (2.7,0.3) -- (2.7,1.8); \draw[-,black] (2.7,1.8) -- (3.3,1.8);
\draw[-,black] (3.3,1.8) -- (3.3,-0.2); \draw[-,black] (3.3,-0.2) -- (3.7,-0.2); 
\draw[-,black] (3.7,-0.2) -- (3.7,0.8); \draw[-,black] (3.7,0.8) -- (4.3,0.8); 
\draw[-,black] (4.3,0.8) -- (4.3,0.3); \draw[-,black] (4.3,0.3) -- (4.7,0.3); 
\draw[-,black] (4.7,0.3) -- (4.7,1.8); \draw[-,black] (4.7,1.8) -- (5.3,1.8); 
\draw[-,black] (5.3,1.8) -- (5.3,0.8); \draw[-,black] (5.3,0.8) -- (5.8,0.8);
\draw[-,black] (5.8,0.8) -- (5.8,0.3); \draw[-,black] (5.8,0.3) -- (6.7,0.3);
\draw[-,black] (6.7,0.3) -- (6.7,1.8); \draw[-,black] (6.7,1.8) -- (7.3,1.8);
\draw[-,black] (7.3,1.8) -- (7.3,0.3); \draw[-,black] (7.3,0.3) -- (7.7,0.3);
\draw[-,black] (7.7,0.3) -- (7.7,1.8); \draw[-,black] (7.7,1.8) -- (9.7,1.8);
\draw[-,black] (9.7,1.8) -- (9.7,3.7); \draw[-,black] (9.7,3.7) -- (-0.2,3.7);

\end{tikzpicture}
\caption{In this picture we assume that $K=4$. The blacks balls represent open sites and the white ones represent closed sites. The region $D$ is given by the set of sites inside the contour in bold. Note that $g^+((1,0)) = 3$ and $g^-((1,0)) = 2$.} 
\end{figure}
 
It is simple to check that there is no path crossing $[0,1]\times\{0\}$ without landing in $D$, hence
\begin{align*}
\big|\mathcal{X}^{0^-}(0)\cap[0,1]\big|\leq \big|D\big|= \sum_{j=1}^{Kg^+((1,0))}\zeta_j + \sum_{j=0}^{Kg^-((1,0))}\zeta_{-j} \, .
\end{align*}
Since $g^+((1,0))$ and $g^-((1,0))$ are geometric random variables, using Lemma \ref{sum moments} we finish the proof.
\end{proof}

\smallskip

\begin{remark}
Using Lemma \ref{paths in [a,b]} we have that the number of paths in $\overline{\mathcal{X}}$ that cross $[a,b]\times\{t\}$ is finite. From this fact we get Proposition \ref{well posedness} following the proof given in \cite{newman2005convergence}.
\end{remark}

\medskip

We are going to need another result about renewal times. Here we need to define the renewal times for a finite collection of paths in $\mathcal{X}^{t^-}$, such that all we know about them is that they cross an interval $[a,b]$ at time $t$. Therefore we are interested in $\mathcal{X}^{t^-}(t) \cap [a,b] = \{\pi \in \mathcal{X}^{t^-}: \pi(t) \in [a,b]\}$. The proof is analogous to that of Proposition \ref{renewal time proposition} and it will be omitted.

\medskip

\begin{lemma}\label{renewals times 3.0}
Fix $a<b$ and consider the collection of paths $\Xi = \{\pi \in \mathcal{X}^{t^-}: \pi(t) \in [a,b]\}$. For any $\pi,\pi^0\in\Xi$  there exist random variables $T_{\pi,\pi^0}$ and $Z_{\pi,\pi^0}$ such that 
\begin{enumerate}
\item [(i)] $t<T_{\pi,\pi^0}$ and $T_{\pi,\pi^0}-t\leq Z_{\pi,\pi^0}$.
\item [(ii)] $T_{\pi,\pi^0}$  is a common renewal time for $\pi$ and $\pi^0$.
\item [(iii)] $\max_{\pi\in\{\pi,\pi^0\}}\sup_{t\leq s\leq T_{\pi,\pi^0}}|\pi(s)-\pi(t)|\leq Z_{\pi,\pi^0}$.
\item [(iv)] For all $k\in\N$ we have that $\esp\big[(Z_{\pi,\pi^0})^k\big]<\infty$.
\item [(v)] For any other pair $\widetilde{\pi},\widetilde{\pi}^0\in\Xi$ we have that $T_{\widetilde{\pi},\widetilde{\pi}^0}$ and $T_{\pi,\pi^0}$ are identically distributed given $\big|\mathcal{X}^{t^-}(t)\cap[a,b]\big|$. The same happens for  $Z_{\widetilde{\pi},\widetilde{\pi}^0}$ and  $Z_{\pi,\pi^0}$.
\end{enumerate}
\end{lemma}

\medskip

We need one more result before we prove condition $B$.

\smallskip

\begin{lemma}\label{lemma to the condition B} There exists a constant $C_1 > 0$ such that
\begin{align*}
    \pr\big[|\eta_{\mathcal{X}}(0,k,0,m)|>1\big]\leq\frac{C_1m}{\sqrt{k}} \, ,
\end{align*}
 for every $m\geq 1$.
\end{lemma}

\begin{proof}
Let us consider $[0,m]$ as the union of length one closed intervals $[i-1,i]$ which are not disjoint. Thus if $|\eta(0,k,0,m)| > 1$ and $|\eta(0,k,i - 1,i)| = 1$ for every $1 \le i \le m$, then there exists $1 \le j \le m-1$ such that $|\eta(0,k,j - 1,j+1)| > 1$. However $|\eta(0,k,j - 1,j)| = |\eta(0,k,j,j+1)| = 1$ also implies that the path starting from $j$ will have to coalesce with all the paths starting at $[j-1,j]$ and all the paths starting at $[j,j+1]$, thus $|\eta(0,k,j - 1,j+1)| = 1$ yielding a contradiction. Therefore $\{|\eta(0,k,0,m)| > 1\} \subset \cup_{i=1}^m \{|\eta(0,k,i-1,i)| > 1\} $ and 
	\begin{align*}
	\pr\big[|\eta_{\mathcal{X}}(0,k,0,m)|>1\big]\leq\sum_{i=1}^m\pr\big[|\eta_{\mathcal{X}}(0,k,i-1,i)|>1\big]=m\pr\big[|\eta_{\mathcal{X}}(0,k,0,1)|>1\big],
	\end{align*}
	and
	\begin{align}\label{enq:0}
	&\pr\big[|\eta_{\mathcal{X}}(0,k,0,1)|>1\big]\nonumber \\
	&=\sum_{j= 2}^{\infty}\pr\big[|\eta_{\mathcal{X}}(0,k,0,1)|>1\big||\mathcal{X}^{0^-}(0)\cap[0,1]|=j\big]\pr\big[\big||\mathcal{X}^{0^-}(0)\cap[0,1]|=j\big] .   
	\end{align}
	To simplify notation, denote $L = |\mathcal{X}^{0^-}(0)\cap[0,1]|$. Given $L=j$, let $\pi_1,\dots,\pi_j$ be the paths in $\mathcal{X}^{0^-}$ such that $\pi_i(0)\in[0,1]$ for $i=1,\dots,j$ and define
	\begin{align*}
	\vartheta_{i,i+1}:=\inf\{n\geq 1: \pi_i(t)=\pi_{i+1}(t), \text{ for all } t\geq n\}.
	\end{align*}
	Note that
	\begin{align}\label{enq:1}
	\pr\big[|\eta_{\mathcal{X}}(0,k,0,1)|>1\big|L=j\big]\leq\sum_{i=1}^{j-1}\pr\big[\vartheta_{i,i+1}>k\big|L=j\big] \, .
	\end{align}
	The problem to deal with the probabilities on the right hand side of \eqref{enq:1} is that $\pi_i$ and $\pi_{i+1}$ might be just crossing interval $[0,1]$ through linear interpolation between two connected sites in the environment. Thus we need to wait until their first renewal time after time $0$ to compare the distance between them and use the estimates on the distribution of their coalescing time. This also makes the proof rather lengthy.   
	
	For $i=1,\dots,j-1$ let $T_{i,i+1}$ and $Z_{i,i+1}$ be the random variables as in Lemma \ref{renewals times 3.0} for the paths $\pi_i$ and $\pi_{i+1}$. Then we have that
	\begin{equation}\label{cifra}
	\pr[\vartheta_{i,i+1}>k | L=j ] =\pr\Big[\vartheta_{i,i+1}>k,T_{i,i+1}>\frac{k}{2}\Big|L=j\Big]
	+\pr\Big[\vartheta_{i,i+1}>k,T_{i,i+1}\leq\frac{k}{2}\Big|L=j\Big]
	\end{equation}
	Consider the first term on the right hand side of \eqref{cifra} which, by (i) and (iv) in Lemma \ref{renewals times 3.0}, is bounded above by 
	\begin{equation}\label{cifra2}
	\pr\big[T_{i,i+1}>\frac{k}{2}\big|L=j\big] \le 
	\frac{2\esp\big[T_{i,i+1}\big|L=j\big]}{k} \le
	\frac{2\esp\big[Z_{1,2}\big|L=j\big]}{k} \, .
	\end{equation}
	Now we deal with the second term in \eqref{cifra} which is dominated from above by
	\begin{align}
	&\sum_{l=1}^{\infty}\pr\Big[\vartheta_{i,i+1} - T_{i,i+1} >\frac{k}{2} , \big|\pi_i(T_{i,i+1})-\pi_{i+1}(T_{i,i+1})\big|=l\Big|L=j\Big] = \nonumber\\
	&\quad =\sum_{l =1}^{\infty}\pr\Big[\vartheta_{i,i+1} - T_{i,i+1}>\frac{k}{2}\Big|\big|\pi_i(T_{i,i+1})-\pi_{i+1}(T_{i,i+1})\big|=l\Big] \times \nonumber\\
	& \qquad \qquad \qquad \times \pr\Big[\big|\pi_i(T_{i,i+1})-\pi_{i+1}(T_{i,i+1})\big|=l\Big|L=j\Big] \nonumber
	\end{align}
	By Collorary \ref{tau(0,le)} we have that 
	\begin{align}\label{enq:4}
	\pr\Big[\vartheta^{T}_{i,i+1}>\frac{k}{2}\Big|\big|\pi_i(T_{i,i+1})-\pi_{i+1}(T_{i,i+1})\big|=l\Big]\leq\frac{2lC}{\sqrt{k}}. \nonumber
	\end{align}
	Thus 
	\begin{eqnarray}\label{enq:5}
	\pr\big[\vartheta_{i,i+1}>k,T_{i,i+1}\leq\frac{k}{2}\big|L=j\big] & \leq &\sum_{l\geq 1}\frac{2lC}{\sqrt{k}} \, \pr\big[\big|\pi_i(T_{i,i+1})-\pi_{i+1}(T_{i,i+1})\big|=l\big|L=j\big] \nonumber \\ 
	& = & \frac{2C}{\sqrt{k}}\esp\big[\big|\pi_i(T_{i,i+1})-\pi_{i+1}(T_{i,i+1})\big|\big | L=j\Big].
	\end{eqnarray}
	Since  
	\begin{align*}
	&\big|\pi_i(T_{i,i+1})-\pi_{i+1}(T_{i,i+1})\big|\\
	&\le \big|\pi_i(T_{i,i+1})-\pi_{i}(0)\big| + \big|\pi_i(0)-\pi_{i+1}(0)\big| +\big|\pi_{i+1}(0)-\pi_{i+1}(T_{i,i+1})\big|\\
	&\le 2 Z_{i,i+1} +1 \, ,    
	\end{align*}
	we have that \eqref{enq:5} is bounded above by
	\begin{equation}
	\label{enq:5-1}
	\frac{2C}{\sqrt{k}}\esp\big[2Z_{i,i+1}+1\big|L=j\big]=\frac{2C}{\sqrt{k}}\esp\big[2Z_{1,2}+1\big|L=j\big].
	\end{equation}
	By \eqref{cifra}, \eqref{cifra2} and \eqref{enq:5-1}, we have that 
	\begin{eqnarray}\label{enq:6}
	\pr[ \vartheta_{i,i+1}>k | L=j ] & \leq & \frac{2\esp\big[Z_{1,2}\big|L=j\big]}{k}+\frac{2C}{\sqrt{k}}\esp\big[2Z_{1,2}+1\big|L=j\big] \nonumber \\
	& \leq & \frac{2(1+C)\esp\big[2Z_{1,2}+1\big|L=j\big]}{\sqrt{k}}.
	\end{eqnarray}
	Hence by $(\ref{enq:1})$ and $(\ref{enq:6})$,
	\begin{equation}\label{enq:7}
	\pr\big[\big|\eta_{\mathcal{X}}(0,k,0,1)\big|>1\big|L= j \big] \leq \frac{2(1+C)}{\sqrt{k}} \, j \,  \esp[2Z_{1,2}+1 | L=j ].
	\end{equation}
	Replacing $(\ref{enq:7})$ in $(\ref{enq:0})$ we get that $ \pr[|\eta_{\mathcal{X}}(0,k,0,1)|>1]$ is dominated by
	\begin{align}\label{enq:8}
	\frac{2(1+C)}{\sqrt{k}}\sum_{j= 2}^{\infty}j\esp\big[2Z_{1,2}+1\big|L=j\big]\pr[L=j] \, .
	\end{align}
	Note that
	\begin{eqnarray}\label{enq:9}
	\lefteqn{\sum_{j =2}^{\infty}j\esp\big[2Z_{1,2}+1\big|L=j\big]\pr\big[L=j\big]} \nonumber\\
	&\leq & \Big(\sum_{j=2}^{\infty}j^2\pr[L=j]\Big)^{\frac{1}{2}} \, \Big(\sum_{j= 2}^{\infty}\esp\big[2Z_{1,2}+1\big|L=j\big]^2\pr[L=j]\Big)^{\frac{1}{2}}\nonumber\\
	&\leq & \Big(\sum_{j=2}^{\infty}j^2\pr[L=j]\Big)^{\frac{1}{2}} \, \Big(\sum_{j=2}^{\infty}\esp\big[(2Z_{1,2}+1)^2\big|L=j\big]\pr[L=j]\Big)^{\frac{1}{2}} \nonumber \\
	& = & \esp\Big[\big|\mathcal{X}^{0^-}(0)\cap[0,1]\big|^2\Big]^{\frac{1}{2}}\esp\big[(2Z_{1,2}+1)^2\big]^{\frac{1}{2}}.
	\end{eqnarray}
	Take $C_1:=2(1+C)\esp\big[|\mathcal{X}^{0^-}(0)\cap[0,1]|^2]^{\frac{1}{2}}\esp\big[(2Z_{1,2}+1)^2\big]^{\frac{1}{2}}$ which is finite by Lemma \ref{paths in [a,b]} and Lemma \ref{renewals times 3.0}. Replacing $(\ref{enq:9})$ in $(\ref{enq:8})$ we have that
	\begin{align*}
	\pr[|\eta_{\mathcal{X}}(0,k,0,1)|>1]\leq\frac{C_1}{\sqrt{k}},
	\end{align*}
	which completes the proof.
\end{proof}

\medskip

We finish this section with the proof of condition B.

\smallskip

\begin{proof}[Proof of condition $B$ of Theorem \ref{convergence theo}] 
Fix $\epsilon > 0$ and take $M_\epsilon$ as in the statement of Lemma \ref{lemma:condB1}, from that lemma we get that
$$
\sup_{t_0,a\in\R}\pr\big[|\eta_{\mathcal{X}_n}(t_0,t,a-\epsilon,a+\epsilon)|> 1\big]
$$
is bounded above by
$$
\pr\Big[\big|\eta_{\mathcal{X}}(0,n^2\gamma t,n\sigma (a-\epsilon) -M_{\epsilon},n\sigma (a+\epsilon) + M_{\epsilon} )\big|> 1\Big]+\epsilon \, .
$$
Then by Lemma \ref{lemma to the condition B}
\begin{align*}
\sup_{t>\beta}\sup_{t_0,a\in\R}\pr\big[|\eta_{\mathcal{X}_n}(t_0,t,a-\epsilon,a+\epsilon)|> 1\big] \leq\frac{C_1}{n\sqrt{\gamma \beta}}2(n\sigma\epsilon+M_{\epsilon})+\epsilon.
\end{align*}
Hence
\begin{align*}
\limsup_{n \rightarrow \infty} \sup_{t>\beta}\sup_{t_0,a\in\R}\pr\big[|\eta_{\mathcal{X}_n}(t_0,t,a-\epsilon,a+\epsilon)|> 1\big]\leq \Big( \frac{2C_1\sigma}{\sqrt{\beta\gamma}}+ 1 \Big) \epsilon \, \rightarrow 0 \text{ as } \epsilon\rightarrow 0^+.
\end{align*}
So we have condition $B$.
\end{proof}

\bigskip

\section{Condition E}
\label{sec:E}

Following \cite{newman2005convergence}, the verification of condition E is a consequence of Lemmas \ref{seq limit locally finite} and \ref{seq limit coalescing brow motion} just below, which are versions of respectively Lemmas 6.2 and 6.3 in that paper.  

Recall the notation from Section \ref{sec:main}. For a set of paths $Y \subset \Pi$ define
\begin{enumerate}
\item[(i)] $Y^{s^-}$ := the subset of paths in $Y$ that start before or at time $s$;
\item [(ii)] For $A\subset\R$ define  $Y^{s^-,A}:=\{\pi\in Y^{s-}; \pi(s)\in A\}$;
\item[(iii)] For $s\leq t$ and $A\subset\R$ define  $Y^{s-}(t):=\{\pi(t); \pi\in Y^{s-} \}$ and $Y^{s-,A}(t):=\{\pi(t); \pi\in Y^{s-,A} \}$.
\end{enumerate}

\smallskip

\begin{lemma}\label{seq limit locally finite}
Let $Z_{t_0}$ be any subsequential limit of $\{\mathcal{X}_n^{t_0^-}\}$. For any $\epsilon>0$, $Z_{t_0}(t_0+\epsilon)$ is almost surely locally finite and 
 \begin{align*}
 \esp\Big[\big|Z_{t_0}(t_0+\epsilon)\cap(a,b)\big|\Big]\leq\frac{(b-a)C_4}{\sqrt{\epsilon}}.
 \end{align*}
\end{lemma}

\smallskip

The proof of Lemma \ref{seq limit locally finite} follows from Lemma \ref{corollary to E} below as Lemma 6.2 follows from Lemma 6.4 in \cite{newman2005convergence}.

\smallskip

\begin{lemma}\label{seq limit coalescing brow motion}
Let $Z_{t_0}$ be any subsequential limit of $\{\mathcal{X}_n^{t_0^-}\}$ and $\epsilon>0$. Denote by $Z_{t_0}^{(t_0+\epsilon)_T}$ the set of paths in $Z_{t_0}$ truncated at time $t_0+\epsilon$. Then $\mathcal{Z}_{t_0}^{(t_0+\epsilon)_T}$ is distributed as coalescing Brownian motions starting from the random set $Z_{t_0}(t_0+\epsilon)\subset\R^2$.
\end{lemma}

\smallskip

The proof of Lemma \ref{seq limit coalescing brow motion} is analogous to the proof of Lemma 6.3 in \cite{newman2005convergence} for nonsimple random walks.

\smallskip

The remain of the section is devoted to state and prove Lemma \ref{corollary to E}, but we first need the following result:

\smallskip

\begin{lemma}\label{lemma 1} There exists a constant $C_2$ such that
\begin{align*}
\esp\Big[\Big|\mathcal{X}^{0^-}(t)\cap [0,1) \Big|\Big]\leq\frac{C_2}{\sqrt{t}}
\end{align*}
for all $t>0.$
\end{lemma}

\begin{proof}
Fix $M \in \mathbb{N}$. By translation invariance
\begin{align*}
&M \esp\Big[\Big|\mathcal{X}^{0^-}(t)\cap[0,1)\Big|\Big] = \esp\Big[\Big|\mathcal{X}^{0^-}(t)\cap [0,M)\Big|\Big]\\
&\qquad \qquad =\sum_{i\in\Z}\esp\Big[\Big|\mathcal{X}^{0^-,[iM,(i+1)M)}(t)\cap [0,M) \Big|\Big]\\
&\qquad \qquad =\sum_{i\in\Z}\esp\Big[\Big|\mathcal{X}^{0^-,[0,M)}(t)\cap [iM,(i+1)M)\big)\Big|\Big] = \esp\Big[\Big|\mathcal{X}^{0^-,[0,M)}(t)\Big|\Big] \, ,
\end{align*}
where the third equality above also uses the symmetry of GRDF paths. Since $\mathcal{X}^{0^-,[0,M)}(0)$ has at least $M$ points which are 0, 1, 2, ..., M-1, then
\begin{align*}
&M \esp\Big[\Big|\mathcal{X}^{0^-}(t)\cap[0,1)\Big|\Big] \\
&\qquad \qquad  =\sum_{j=M}^{\infty}\esp\Big[\big|\mathcal{X}^{0^-,[0,M)}(t)\big|\Big|\big|\mathcal{X}^{0^-,[0,M)}(0)\big|=j\Big]\pr\Big[\big|\mathcal{X}^{0^-,[0,M)}(0)\big|=j\Big] \, .
\end{align*}
From here the proof is very close to that of Lemma \ref{lemma to the condition B}, given $|\mathcal{X}^{0^-}(0)\cap[0,M)|=j$, $j\ge M$, let $\pi_1,\dots,\pi_j$ be the paths in $\mathcal{X}^{0^-}$ such that $0 \le \pi_1(0) < \pi_2(0) < ... < \pi_j(0) < M$ for $i=1,\dots,j$ and define
\begin{align*}
    \vartheta_{i,i+1}:=\inf\{n\geq 1: \pi_i(t)=\pi_{i+1}(t), \text{ for all } t\geq n\}.
\end{align*}  
Then 
\begin{align*} \esp\Big[\big|\mathcal{X}^{0^-,[0,M)}(t)\big|\Big|\big|\mathcal{X}^{0^-,[0,M)}(0)\big|=j\Big]\leq\esp\Big[1+\sum_{i=1}^{j-1}1_{\{\vartheta_{i,i+1}>t\}}\Big||\mathcal{X}^{0^-,[0,M)}(0)\big|=j\Big].
\end{align*}
Note that $|\pi_i(0)-\pi_{i+1}(0)| \le 1$ because $0,1,...,M-1 \in \mathcal{X}^{0^-,[0,M)}(0)$, then by Lemma \ref{renewals times 3.0} and (\ref{enq:6}) we have that there exists a constant $C >0$ and a integrable random variable $Z$, both not depending on $M$, such that 
\begin{align*}
   P\big[\vartheta_{i,i+1}>t\big||\mathcal{X}^{0^-,[0,M)}(0)\big|=j\big]&\leq\frac{2(1+C)}{\sqrt{t}}\esp\big[2Z+1\big||\mathcal{X}^{0^-,[0,M)}(0)\big|=j\big]\\
    &\leq\frac{2(1+C)}{\sqrt{t}}\esp\big[3Z\big||\mathcal{X}^{0^-,[0,M)}(0)\big|=j\big]\\
    &=\frac{\widetilde{C}}{\sqrt{t}}\esp\big[Z\big||\mathcal{X}^{0^-,[0,M)}(0)|=j\big].
\end{align*}
Hence
$$
M\esp\big[\big|\mathcal{X}^{0^-}(t)\cap[0,1)\big|\big]\leq1+\frac{\widetilde{C}}{\sqrt{t}}\sum_{j=M}^{\infty}j\esp\big[Z\big||\mathcal{X}^{0^-,[0,M)}(0)|=j\big]\pr\big[|\mathcal{X}^{0^-,[0,M)}(0)|=j\big]
$$
which, as in \eqref{enq:9}, can be shown to be bounded above by
\begin{align*}    
    &1+\frac{\widetilde{C}}{\sqrt{t}}\Big(\esp\Big[\big|\mathcal{X}^{0^-,[0,M)}(0)\big|^2\Big]\Big)^{\frac{1}{2}}\Big(\esp[(Z)^2]\Big)^{\frac{1}{2}} \\ 
    &\qquad \qquad  \leq 1+\frac{\widetilde{C}}{\sqrt{t}}\Big(M\esp\Big[\sum_{i=1}^{M}\big|\mathcal{X}^{0^-,[i-1,i]}(0)\big|^2\Big]\Big)^{\frac{1}{2}}\Big(\esp[(Z)^2]\Big)^{\frac{1}{2}}\\
    &\qquad \qquad = 1+\frac{\widetilde{C}}{\sqrt{t}}M\Big(\esp\Big[\big|\mathcal{X}^{0^-,[0,1]}(0)\big|^2\Big]\Big)^{\frac{1}{2}}\Big(\esp[(Z)^2]\Big)^{\frac{1}{2}} \, .
\end{align*}
Thus 
$$
\esp\Big[\Big|\mathcal{X}^{0^-}(t)\cap[0,1)\Big|\Big]\leq \frac{1}{M} + \frac{C_2}{\sqrt{t}} \, ,
$$
where $C_2:=\widetilde{C}\Big(\esp[(Z)^2]\Big)^{\frac{1}{2}}\Big(\esp\Big[\big|\mathcal{X}^{0^-,[0,1]}(0)\big|^2\Big]\Big)^{\frac{1}{2}}$ which is finite by Lemma \ref{paths in [a,b]}. Since $M$ is arbitrary we obtain the bound in the statement. \end{proof}

\medskip

\begin{lemma}\label{corollary to E} There exists a constant $C_3 > 0$ such that
\begin{align*}
\esp\Big[\big|\mathcal{X}_n^{0^-}(t)\cap[0,M)\big|\Big]\leq\frac{MC_3}{\sqrt{t}} \, ,
\end{align*}
for every $n\ge 1$, $M\ge 1$ and $t>0$.
\end{lemma}

\begin{proof}
Using  Lemma \ref{lemma 1} we have that for all $n\geq 1$ 
\begin{align*}
\esp\Big[\big|\mathcal{X}^{0^-}(n^2\gamma t)\cap [0,n\sigma M)\big|\Big]&=n\sigma M\esp\Big[\big|\mathcal{X}^{0^-}(n^2\gamma t)\cap [0, 1)\big|\Big]\\
&\leq\frac{n\sigma M C_2}{\sqrt{n^2\gamma t}} = \frac{\gamma^{-1/2} \sigma M C_2}{\sqrt{t}}.
\end{align*}
\end{proof}

\medskip

\section{Condition T}
\label{sec:T}

In this section we prove condition $T$ in Theorem \ref{convergence theo}, which follows from Proposition \ref{tightness} at the beggining of this section. The idea behind the proof comes from \cite{newman2005convergence}. Only technical details related to the renewal times impose an extra difficult, even though,  we present it here for the sake of completeness. 

\smallskip

By homogeneity of the GRDF all the estimates on $A_{\mathcal{X}_n}(x_0,t_0;\rho,t)$ are uniform on $(x_0,t_0)\in\Z^2$. Here we only consider $(x_0,t_0) = (0,0)$ leaving the verification for other choices of $(x_0,t_0)$ to the reader. The case $n \gamma t_0 \notin \mathbb{Z}$ demands an extra care, but can be dealt analogously as done on previous sections to deal with paths crossing some time level not necessarily on the rescaled space/time lattice. With this in mind, condition T is a consequence of the next result.

\smallskip

\begin{proposition}\label{tightness}Denote by $A_{\mathcal{X}_n}^+(x_0,t_0;\rho,t)$ the event that $\mathcal{X}_n$ contains a path touching  both $R(x_0,t_0;\rho,t)$ and the right boundary of the rectangle $R(x_0,t_0;20\rho,4t)$. Then
\begin{align*}
\lim_{t\rightarrow 0^+}\frac{1}{t}\limsup_{n\rightarrow \infty} \pr\Big[A^+_{\mathcal{X}_n}(0,0;\rho,t)\Big]=0.    
\end{align*}
\end{proposition}

\smallskip

Before we prove Proposition \ref{tightness} we need some lemmas. The first one gives an uniform bound on the overshoot distribution on the renewal times for paths in the GRDF.

\medskip

\begin{lemma}\label{tau_0^+}
For $x\in\Z_{-}$ let $\{T_i^x:i\geq 1\}$ be a sequence of renewal times from Corollary \ref{renewal time corollary one path} for the path that starts in $(x,0)$. Define $\{Y_i^x:i\geq 1\}$ as the first component of the path $\pi^{(x,0)}$ on the renewal time $T_i^x$, i.e.
$$
Y_i^x = X^{(x,0)}_{T_i^x}.
$$
Also define $\nu_+^x=\inf\{n\geq 1: Y_n^x\geq 1\}$. Then we have
$$
  \sup_{x\in\Z_-}\esp\big[(Y^x_{\nu_+^x})^k\big]<\infty,
$$
for all $k\geq 1$.
\end{lemma}

\begin{proof}
The proof follows from Lemma \ref{overshoot UI} just below which is Lemma $2.6$ in \cite{newman2005convergence} where it is proved. 
\end{proof}

\begin{lemma}\label{overshoot UI} (Lemma $2.6$ in \cite{newman2005convergence})
Let $(S^x_n)_{n\geq 0}$ be a random walk with increments distributed as a random variable $Z$ such that it starts from $x\in\Z_-$ at time $0$. If $\esp[|Z|^{k+2}]<\infty$ then $\{(S^x_{\nu^x_+})^k\}_{z\in\Z_-}$, where $\nu_+^x=\inf\{n\geq 1; S_n^x>0\}$, is uniformly integrable.
\end{lemma}

\smallskip

A path in the GRDF is obtained from linear interpolation between open points in the envinroment, we say that these open points defining the path are the ones visited by the path. The next Lemma states that the probability of having paths that cross a box $R(0,0;n\rho\sigma,n^2t\gamma)$ but do not visit any point in $R(0,0;2n\rho\sigma,2n^2t\gamma)$ goes to zero as $n \rightarrow \infty$.

\medskip

\begin{lemma}\label{D(nu)} Let $D(n\rho\sigma,n^2t\gamma)$ be the event that paths in $\mathcal{X}$ cross $R(0,0;n\rho\sigma,n^2t\gamma)$ without visit any point in $R(0,0;2n\rho\sigma,2n^2t\gamma)$, then
\begin{align*}
    \lim_{n \rightarrow \infty} \pr\Big[D(n\rho\sigma,n^2t\gamma)\Big]=0.
\end{align*}
\end{lemma}

\medskip

\begin{proof}

Recall from (\ref{H(u)}) the definition of the random variable 
\begin{align*}
    H(v):=\inf\Big\{n\geq 1: \sum_{j=1}^{n}\mathbbm{1}_{\{(v(1),v(2)+j)\text{ is open}\}}=K\Big\} \, ,
\end{align*}
where $v\in\Z^2$ and $K\in\N$ is such that $\pr[W_v\leq K]=1$. Put $H=H((0,0))$ and note that $H$ has negative binomial distribution of parameters $p$ and $K$, thus it has finite absolute moments of any order. Recall that jumps are made to the upmost open site on a chosen open level. If some path in $D(n\rho\sigma,n^2t\gamma)$ comes from $\Z\times\Z_{-}$ before crossing  $R(0,0;2n\rho\sigma,2n^2t\gamma)$ then $H(v)>n\rho\sigma$ for some $v\in\{-n\rho\sigma,\dots,n\rho\sigma\}\times\{0\}$, and 
\begin{align*}
    &\pr\big[H(v)>n\rho\sigma \text{ for some } v\in\{-n\rho\sigma,\dots,n\rho\sigma\}\times\{0\}\big]\\
    &\leq 2n\rho\sigma\pr[H>n\rho\sigma]\\&
    \leq 2n\rho\frac{\esp[H^2]}{(n\rho\sigma)^2}\rightarrow 0 \text{ as } n \text{ goes to infinity.}
\end{align*}
Any other path in $D(n\rho\sigma,n^2t\gamma)$ come from points with first component bigger than $2n\rho\sigma$ or smaller than $-2n\rho\sigma$ and second component in $\{0,\dots,n^2t\gamma\}$. If from some $v=(v(1),v(2))$ with $v(1)> 2n\rho\sigma$ and $v(2)\in\{0,\dots,n^2t\gamma\}$ the path crosses $R(0,0;2n\rho\sigma,2n^2t\gamma)$ then $H(v)>v(1)$. In case that $v(1)<-2n\rho\sigma$ we have that $H(v)>-v(1)$. Then the probability that one of these paths crosses $R(0,0;2n\rho\sigma,2n^2t\gamma)$ is bounded by
\begin{align*}
    2\sum_{j=1}^{2n^2t\gamma}\sum_{v\in \{(2n\rho\sigma,\infty)\cap\Z\}\times\{j\}}\pr[H(v)>v(1)]&\leq 2(2n^2t\gamma)\sum_{i\geq 1}\pr[H>i+2n\rho\gamma]\\
    &\!\!\!\!\!\!\!\!\!\!\!\!\!\!\!\!\!\!\!\!\!\!\!\!\!\!\!\!\!\! \leq2(2n^2t\gamma)\sum_{i\geq 1}\frac{\esp[H^6]}{(i+2n\rho\sigma)^6}\\
    &\!\!\!\!\!\!\!\!\!\!\!\!\!\!\!\!\!\!\!\!\!\!\!\!\!\!\!\!\!\! \leq2(2n^2t\gamma)\sum_{i\geq 1}\frac{\esp[H^6]}{i^3(2n\rho\sigma)^3} = \frac{C(t,\rho)}{n} \rightarrow 0 \text{ as } n\rightarrow\infty.
\end{align*}
This completes the proof.
\end{proof}

\medskip

We still need one last lemma to prove Proposition \ref{tightness}. The result is important to control how much two paths in the GRDF can become far apart each other after crossing and before coalescence.

\medskip

\begin{lemma}\label{tau(x,y,u^+)} Let $x,y,x_1,\dots,x_m$ be distinct points in $\Z$ with $x<y$.  Define $u=(x,0)$ and $v=(y,0)$. Consider the random times $\{T_n:n\geq 1\},\{\tau_n(u):n\geq 1\}$ and $\{\tau_n(v):n\geq 1\}$ as introduced in Corollary \ref{renewal time corollary} for the points $u$, $v$, $(x_1,0)$,\dots,$(x_m,0)$. 
Put $Y_n:= X_{\tau_n(u_0)}^{u}(1)-X_{\tau_n(u_l)}^{v}(1)$, $n\ge 1$, which are the increments between $X^u$ and $X^v$ on the common renewal times for $u$, $v$, $(x_1,0)$,\dots,$(x_m,0)$. Define the random times
\begin{align*}
\nu_{x,y} = \inf\{ j \ge 0 : Y_j = 0\} \quad \textrm{and} \quad \nu_{x,y,\rho^+}:=\inf\big\{j \geq 0: Y_j \ge n \rho \sigma \big\},
\end{align*}
then for $\rho>0$ there exists a constant $C$ depending only on $t$ and $\rho$ such that for all $n$ large enough we have that
\begin{align*}
\pr\big[\nu_{x,y,\rho^+}<\nu_{x,y}\wedge(n^2t\gamma)\big]<\frac{C(t,\rho)}{n} \, .
\end{align*}
\end{lemma}

\smallskip


\begin{proof} The proof is analogous to the proof of Lemma $3.2$ in \cite{coletti2014convergence}. For $l\in\Z$ consider $u_0=(0,0)$, $u_l=(l,0)$ and the random times $(\tau_n(u_0))_{n\geq 1}, (\tau_n(u_l))_{n\geq 1}$ as introduced in Corollary \ref{renewal time corollary} for the points $u_0,u_l$. Define the following random walk
\begin{align*}
    Y_0^l:=l, Y_n^l:= X_{\tau_n(u_0)}^{u_0}(1)-X_{\tau_n(u_l)}^{u_l}(1) \text{ for } n\geq 1.
\end{align*}
Let $B^l(x,t)$ be the set of trajectories that remain in the interval $[l-x,l+x]$ during the time $[0,t]$. As in Proposition 2.4 in \cite{newman2005convergence}, by the independence of the increments which implies the strong Markov property, we have that 
\begin{align*}
    \pr(\nu_{x,y}>n^2\gamma t)\geq\pr(\nu_{x,y,\rho^+}<n^2\gamma t\wedge \nu_{x,y} )\inf_{l\in\Z}\pr(Y^l\in B^l(n\sigma \rho,n^2\gamma t)).
\end{align*}
Note that 
\begin{align*}
    \inf_{l\in\Z}\pr\big(Y^l\in B^l(n\sigma\rho,n^2\gamma t)\big)=1-\sup_{l\in\Z}\pr\Big(\sup_{i\leq n^2\gamma t}|Y^l_i- l|\geq n\sigma \rho\Big).
\end{align*}
We have that
$$
    \limsup_{n\rightarrow\infty} \sup_{l\in\Z}\pr\Big(\sup_{i\leq n^2\gamma t}|Y^l_i-l|>n\sigma \rho\Big) \, ,
$$
is bounded above by
\begin{align*}    
    & \limsup_{n\rightarrow\infty} \sup_{l\in\Z}\pr\Big(\sup_{i\leq n^2\gamma t}|X_{\tau_i(u_0)}^{u_0}(1)| +|X_{\tau_i(u_l)}^{u_l}(1)-l|> \frac{n\sigma \rho}{2} \Big)\\
    & \qquad \qquad \qquad \leq 2 \limsup_{n\rightarrow\infty} \sup_{l\in\Z}\pr\Big(\sup_{i\leq n^2\gamma t}|X_{\tau_i(u_0)}^{u_0}(1)|> \frac{n\sigma \rho}{4} \Big)\\
    & \qquad \qquad \qquad \leq 4\pr\Big(N>\frac{\rho}{4\sqrt{t}}\Big)=4e^{-\frac{\rho^2}{32 \, t}} \, ,
\end{align*}
where $N$ is a standard normal random variable and the last inequality is a consequence of Donsker's Theorem, see also Lemma 2.3 in \cite{newman2005convergence}. Hence
$$
\inf_{l\in\Z}\pr\big(Y^l\in B^l(n\sigma\rho,n^2\gamma t)\big)
$$
is bounded from below by a constant that depends only on $t$ and $\rho$. So
using Proposition \ref{tau(u,v)} we obtain a constant $\tilde{C}(t,\rho)$ such that 
\begin{align*}
\pr(\nu_{x,y,\rho^+}<n^2\gamma t\wedge \nu_{x,y} )
\leq \frac{\pr(\nu_{x,y}>n^2\gamma t)}{\inf_{l\in\Z}\pr\big(Y^l\in B^l(n\sigma\rho,n^2\gamma t)\big)} \leq \frac{\tilde{C}(t,\rho) \, |y-x|}{n} \, .
\end{align*}
From the previous inequality we should follow the same steps as in the proof of Proposition 2.4 in \cite{newman2005convergence} to get an upper bound that do not depend on $|y-x|$.
\end{proof}

\medskip

\begin{proof}[Proof of Proposition \ref{tightness}] 
Let $\pi_1,\pi_2,\pi_3,\pi_4$ be the paths that start in $5\lfloor n\rho\sigma\rfloor$, $9\lfloor n\rho\sigma \rfloor$, $13\lfloor n\rho\sigma\rfloor$ and $17\lfloor n\rho\sigma\rfloor$ respectively at time zero. Let us denote by $B^{n,t}_i$, $i=1,\dots,4$, the event that $\pi_i$ stays within distance $n\rho\sigma$ of $\pi_i(0)$ until time $2tn^2\gamma$, see Figure 8 below. From the invariance principle we have that 
$$
\lim_n\pr[(B_i^{n,t})^c]=\pr \Big[\sup_{s\in[0,t]}|\mathcal{B}_s|>\rho\Big]\leq4e^{-\frac{\rho^2}{2t}}
$$
for all $i=1,\dots,4$. Then
\begin{align}
\label{Bint}
\frac{1}{t} \lim_{n \rightarrow \infty} \pr\Big[(B_i^{n,t})^c\Big]\rightarrow 0 \text{ as } t\rightarrow 0^+.    
\end{align}
Recall the definition of the set $D(n\rho\sigma,n^2t\gamma)$ from the statement of Lemma \ref{D(nu)} and note that
\begin{eqnarray*}
 \lefteqn{\lim_{t\rightarrow 0^+}\frac{1}{t}\limsup_{n \rightarrow \infty} \pr\Big[A^+_{\mathcal{X}_n}(0,0;\rho,t)\Big]= \lim_{t\rightarrow 0^+}\frac{1}{t}\limsup_{n \rightarrow \infty} \pr\Big[A^+_{\mathcal{X}}(0,0;\rho n\sigma,tn^2\gamma)\Big]}\\ 
&\leq \lim_{t\rightarrow 0^+}\frac{1}{t}\limsup_{n \rightarrow \infty} \pr\Big[D(n\rho\sigma,tn^2\gamma)\Big]+4 \lim_{t\rightarrow 0^+}\frac{1}{t}\lim_{n \rightarrow \infty} \pr\Big[(B_1^{n,t})^c\Big]\\&+ \lim_{t\rightarrow 0^+}\frac{1}{t}\limsup_{n \rightarrow \infty} \pr\Big[A^+_{\mathcal{X}}(0,0;\rho n\sigma,tn^2\gamma),\cap_{i=1}^4B_i^{n,t},\big(D(n\rho\sigma,tn^2\gamma)\big)^c\Big].
\end{eqnarray*}
By \eqref{Bint} and Lemma \ref{D(nu)}, to finish the proof of Proposition \ref{tightness}, we only have to prove (see Figure 8) that for every $t>0$
\begin{align}
\label{TT}
    \limsup_{n \rightarrow \infty} \pr\Big[A^+_{\mathcal{X}}(0,0;\rho n\sigma,tn^2\gamma),\cap_{i=1}^4B_i^{n,t},(D(n\rho\sigma,tn^2\gamma))^c\Big]=0 \, .
\end{align}
Fix some $(x,m)\in R(0,0;2n\rho\sigma,2n^2t\gamma)$ and take $(T_i)_{i\geq 1}$ as the renewal times introduced in Corollary \ref{renewal time corollary} for the points $(5\lfloor n\rho\sigma\rfloor,0)$,$(9\lfloor n\rho\sigma\rfloor,0)$,$(13\lfloor n\rho\sigma\rfloor,0)$, $(17\lfloor n\rho\sigma\rfloor,0)$ and the $(x,m)$. Put $T_0 = m$ and $Y^{(x,m)}_i = X^{(x,m)}_{T_i}$, $i\ge 0$. Define the stopping time $\nu_j^{(x,m)}$ (with respect to $(\mathcal{F}^{(x,m)}_n)_{n\ge 1}$ defined on Lemma \ref{tpr}), for $j=1,\dots,5$, as the first time that the random walk $\big(Y_i^{(x,m)}\big)_{i\geq 0}$ exceeds $(4j-1)\lfloor n\rho\sigma\rfloor$, and the random time $\nu^{(x,m)}$ the first time that $\pi^{(x,m)}$ exceeds $20 n\rho\sigma$. Then
\begin{align}
\label{TT1}
\pr\big[\nu^{(x,m)}<4n^2t\gamma,\cap_{i=1}^4B_i^{n,t}\big] & \leq\pr\big[T_{\nu_5^{(x,m)}}<4n^2t\gamma,\cap_{i=1}^4B_i^{n,t}\big] \nonumber \\
& \quad + \pr\big[\nu^{(x,m)}<4n^2t\gamma,T_{\nu_5^{(x,m)}}\geq 4n^2t\gamma\big].
\end{align}

Note that on the event 
$$
\big\{\nu^{(x,m)}<4n^2\gamma t,T_{\nu_5^{(x,m)}}\geq 4n^2t\gamma \big\}
$$ 
the path $\pi^{(x,m)}$ crosses the interval $\big(19\lfloor n\rho\sigma \rfloor,20 n\rho\sigma \big)$ without renewal before time $n^2t\gamma$. Since the displacement between consecutive renewal times is bounded by some random variable $Z$ with finite moments, and up to time $n^2\gamma t$ the number of renewals is bounded by $n^2\gamma t$, we have that
\begin{align}
\label{TT2}
\pr\big[\nu^{(x,m)}<4n^2\gamma t,T_{\nu_5^{(x,m)}} \geq 4n^2t\gamma\big] & \leq n^2t\gamma\pr\big[Z> n\rho\sigma \big] \leq \frac{n^2t\gamma\esp[Z^6]}{(n\rho\sigma)^6} \le \frac{C_1}{n^4}.
\end{align}

\begin{figure}
\label{fig:tight}
\begin{overpic}[scale = .9]
{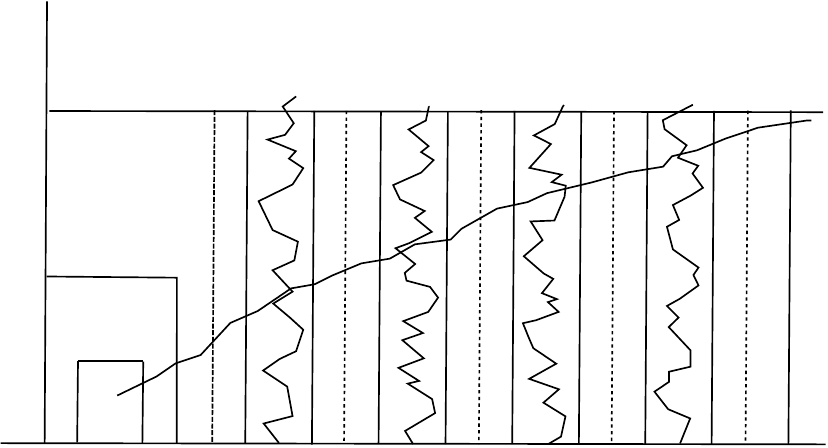}
\put(10,4){\tiny{$(x,m)$}}
\put(-2,40) {\tiny{$4n^2\gamma t$}} \put(-2,20){\tiny{{$2n^2\gamma t$}}} \put(-2,10){\tiny{{$n^2\gamma t$}}}
\put(24,-2) {\tiny{$3\widetilde{\rho}$}} \put(28,-2) {\tiny{$4\widetilde{\rho}$}} \put(32,-2) {\tiny{$5\widetilde{\rho}$}} \put(36,-2) {\tiny{$6\widetilde{\rho}$}}
\put(40,-2) {\tiny{$7\widetilde{\rho}$}} \put(44,-2) {\tiny{$8\widetilde{\rho}$}} \put(48,-2) {\tiny{$9\widetilde{\rho}$}} \put(52,-2) {\tiny{$10\widetilde{\rho}$}}
\put(56,-2) {\tiny{$11\widetilde{\rho}$}} \put(60,-2) {\tiny{$12\widetilde{\rho}$}} \put(64,-2) {\tiny{$13\widetilde{\rho}$}} \put(68,-2) {\tiny{$14\widetilde{\rho}$}}
\put(72,-2) {\tiny{$15\widetilde{\rho}$}} \put(76,-2) {\tiny{$16\widetilde{\rho}$}} \put(80,-2) {\tiny{$17\widetilde{\rho}$}} \put(84,-2) {\tiny{$18\widetilde{\rho}$}}
\put(88,-2) {\tiny{$19\widetilde{\rho}$}} \put(93,-2) {\tiny{$ 20 n\rho\sigma$}}
\end{overpic}
\caption{Realization of $A^+_{\mathcal{X}}(0,0;\rho n\sigma,tn^2\gamma) \, \cap \, \cap_{i=1}^4B_i^{n,t}$ where $A^+_{\mathcal{X}}(0,0;\rho n\sigma,tn^2\gamma)$ occurs because the path $\pi^{(x,m)}$, for some $(x,m) \in R(0,0;\rho n\sigma,tn^2\gamma)$, touchs the right boundary of the rectangle $R(0,0;20\rho n\sigma,4tn^2\gamma)$. Notation: $\tilde{\rho} = \lfloor n \sigma \rho \rfloor$. }
\end{figure}
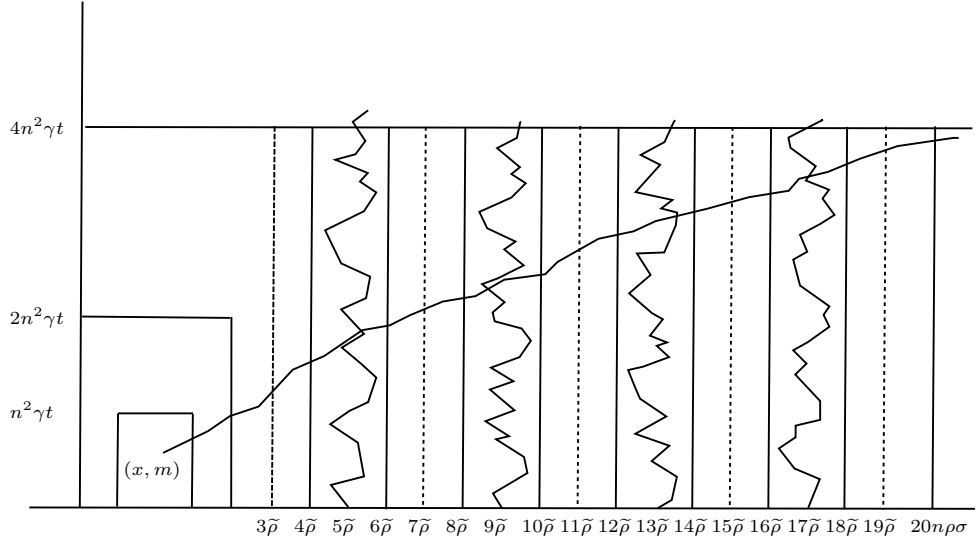

We also have
\begin{align}
\label{TT3}
&\pr\Big[T_{\nu_5^{(x,m)}}<4n^2t\gamma,\cap_{i=1}^4B_i^{n,t}\Big] \nonumber \\
&\leq
\pr\Big[Y_{\nu_j^{(x,m)}}^{(x,m)}\leq \big(4j-\frac{1}{2}\big)\lfloor n\rho\sigma\rfloor, \, j=1,...5, \, T_{\nu_5^{(x,m)}}<4n^2t\gamma,\cap_{i=1}^4B_i^{n,t}\Big] \nonumber \\
& \qquad \qquad \qquad \qquad  + \sum_{j=1}^5 \pr\Big[Y_{\nu_j^{(x,m)}}^{(x,m)}> \big(4j-\frac{1}{2}\big)\lfloor n\rho\sigma\rfloor\Big] \nonumber \\
&\leq \pr\Big[Y_{\nu_j^{(x,m)}}^{(x,m)}\leq \big(4j-\frac{1}{2}\big)\lfloor n\rho\sigma\rfloor, \, j=1,...5, \, T_{\nu_5^{(x,m)}}<4n^2t\gamma,\cap_{i=1}^4B_i^{n,t}\Big] \nonumber \\
&\qquad \qquad \qquad \qquad +5\sup_{x\in\Z_{-}}\pr\Big[Y^{(x,m)}_{\nu_{^+}^x}>\frac{\lfloor n\rho\sigma\rfloor}{2}\Big].
\end{align}
By Lemma \ref{tau_0^+} and Corollary \ref{renewal time corollary} there exists a constant $C_2 >0$ such that 
\begin{align*}
\sup_{x\in\Z_{-}}\pr\Big[Y^{(x,m)}_{\nu_{^+}^x}>\frac{\lfloor n\rho\sigma\rfloor}{2}\Big]\leq \frac{C_2}{n^{4}}.
\end{align*} 
Using the strong Markov property and Lemma \ref{tau(x,y,u^+)} we get a constant $C_3 >0$ such that 
\begin{align*}
&\pr\Big[Y_{\nu_j^{(x,m)}}^{(x,m)}\leq \big(4j-\frac{1}{2}\big)\lfloor n\rho\sigma\rfloor, \, j=1,...5, \, T_{\nu_5^{(x,m)}}<4n^2t\gamma,\cap_{i=1}^4B_i^{n,t}\Big]\\ 
&\le \pr\Big[\vartheta_{x,y,\rho}^+<\vartheta_{x,y}\wedge(n^2t\gamma)\Big]^4 \leq\frac{C_3}{n^4}.
\end{align*}
Hence by \eqref{TT3}
\begin{align}
\label{TT4}
\pr\Big[T_{\nu_5^{(x,m)}}<4n^2t\gamma,\cap_i^4B_i^{n,t}\Big]\leq \frac{C_3}{n^4}+\frac{5C_2}{n^{4}}.   
\end{align}
We can go back to \eqref{TT1}, use \eqref{TT2}, \eqref{TT3} and \eqref{TT4} to conclude that  
$$
\pr\big[\nu^{(x,m)}<4n^2t\gamma,\cap_{i=1}^4B_i^{n,t}\big] \le \frac{(5C_1 + C_2 + C_3)}{n^4} \, .
$$
Therefore we can estimate the probability in \eqref{TT} as
\begin{align*}
    &\pr\Big[A^+_{\mathcal{X}}(0,0;\rho n\sigma,tn^2\gamma),\cap_{i=1}^4B_i^{n,t},\{D(n\rho\sigma,n^2t\gamma)\}^c\Big]\\&\leq \pr\Big[\exists (x,m)\in R(0,0;2n\rho\sigma,2n^2t\gamma); \nu^{(x,m)}<4n^2t\gamma,\cap_{i=1}^4B_i^{n,t}\Big].
\end{align*}
Since $|R(0,0;2n\rho\sigma,2n^2t\gamma)| \le 8 t \rho \sigma \gamma n^3$ we have that  
\begin{align*}
   &\limsup_{n \rightarrow \infty} \pr\Big[A^+_{\mathcal{X}}(0,0;\rho n\sigma,tn^2\gamma),\cap_{i=1}^4B_i^{n,t},\{D(n\rho\sigma,n^2t\gamma)\}^c\Big] \\
   & \le \limsup_{n \rightarrow \infty} ( 8 t \rho \sigma \gamma n^3 ) \pr\big[\nu^{(x,m)}<4n^2t\gamma,\cap_{i=1}^4B_i^{n,t}\big] \\
   &\leq \limsup_{n \rightarrow \infty}( 8 t \rho \sigma \gamma n^3) \frac{(5C_1 + C_2 + C_3)}{n^4} = 0.
\end{align*}
\end{proof}

\bigskip

\appendix

\section{A technical estimate}
\label{sec:well}

\begin{lemma}\label{sum moments}
 Let $N$ be some positive integer random variable and $(\zeta_n)_{n\geq 1}$ a non-negative sequence of identically distributed random variables. If for some $k\geq 1, \delta >0 \text{ and } l>\frac{(k+2)(1+\delta)}{\delta}$ we have $\esp[\zeta_1^{k(1+\delta)}]$ and $\esp[N^l]$ finite, then for $S:=\sum_{n=1}^{N}\zeta_n$ we get that $\esp[S^k]$ is also finite.
 \end{lemma}
 
 \begin{proof}
 We have that $0\leq S\leq N\max_{1\leq j\leq N}\zeta_j$ what implies that 
$$
S^k\leq N^k\max_{1\leq j\leq N}\zeta^k_j\leq N^k\sum_{j=1}^{N}\zeta^k_j.
$$ 
Hence
 \begin{align*}
 \esp\big[S^k\big]&\leq\esp\Big[N^k\sum_{j=1}^{N}\zeta_j^k\Big]=\sum_{n=1}^{\infty}n^k\sum_{j=1}^{n}\esp\big[\mathbbm{1}_{\{N=n\}}\zeta^k_j\big].
 \end{align*}
 Applying H\"older inequality we get
 \begin{align*}
 \esp\big[S^k\big]\leq\sum_{n=1}^{\infty}n^k\sum_{j=1}^{n}\esp\big[\zeta_j^{k(1+\delta)}\big]^{\frac{1}{1+\delta}}\pr[N=n]^\frac{\delta}{1+\delta}=\esp\big[\zeta_1^{k(1+\delta)}\big]^{\frac{1}{1+\delta}}\sum_{n=1}^{\infty}n^{k+1}\pr[N=n]^{\frac{\delta}{1+\delta}}.
 \end{align*}
 Applying Chebyshev inequality we get that $\esp[S^k]$ is bounded above by
 \begin{align*}
 \esp\big[\zeta_1^{k(1+\delta)}\big]^{\frac{1}{1+\delta}}\sum_{n=1}^{\infty}n^{k+1}\frac{\esp[N^l]^\frac{\delta}{1+\delta}}{n^{\frac{l\delta}{1+\delta}}}=\esp\big[\zeta_1^{k(1+\delta)}\big]^{\frac{1}{1+\delta}}\esp[N^l]^{\frac{\delta}{1+\delta}}\sum_{n=1}^{\infty}\frac{1}{n^{\frac{l\delta}{1+\delta}-(k+1)}}<\infty.
 \end{align*}
 \end{proof}
 
\bigskip

\noindent \textit{Acknowledgments:} We would like to thank Maria Eulalia Vares, Leandro Pimentel and Luiz Renato Fontes for useful comments and suggestions. We would also like to thank the anonymous referee for useful feedback and suggestions regarding the paper.

\medskip

\noindent \textit{Data availability statement:} Data sharing not applicable to this article as no datasets were generated or analysed during the current study.

\medskip

\noindent \textit{Conflict of interest statement:} The authors have no competing interests to declare that are relevant to the content of this article.

\medskip

\markboth{Refereces}{References}
\bibliographystyle{plain}
\bibliography{reference1}

\begin{thebibliography}{10}

\bibitem{bmsv}
Samir Belhaouari, Thomas Mountford, Rongfeng Sun, and Glauco Valle.
\newblock Convergence and sharp results for the voter model interfaces.
\newblock {\em Electronic Journal of Probability}, 11:279--296, 2006.

\bibitem{coletti2009scaling}
Cristian Coletti, Luiz~Renato Fontes, and Erasmo~Souza Dias.
\newblock Scaling limit for a drainage network model.
\newblock {\em Journal of Applied Probability}, 46(4):1184--1197, 2009.

\bibitem{coletti2014convergence}
Cristian Coletti and Glauco Valle.
\newblock Convergence to the brownian web for a generalization of the drainage
  network model.
\newblock {\em Annales de l'Institut Henri Poincar{\'e}, Probabilit{\'e}s et
  Statistiques}, 50(3):899--919, 2014.

\bibitem{durrett}
Rick Durrett.
\newblock {\em Probability: theory and examples}.
\newblock Cambridge university press, 2010.

\bibitem{ffw}
Pablo Ferrari, Luiz~Renato Fontes, and Xian-Yuan Wu.
\newblock Two-dimensional poisson trees converge to the brownian web.
\newblock {\em Annales de l'Institut Henri Poincar\'e. B, Probability and
  Statistics}, 41:851--858, 2014.

\bibitem{fontes2004brownian}
Luiz~Renato Fontes, Marco Isopi, Charles~M Newman, and Krishnamurthi
  Ravishankar.
\newblock The brownian web: characterization and convergence.
\newblock {\em The Annals of Probability}, 32(4):2857--2883, 2004.

\bibitem{fontes2015scaling}
Luiz~Renato Fontes, Leon Valencia, and Glauco Valle.
\newblock Scaling limit of the radial poissonian web.
\newblock {\em Electronic Journal of Probability}, 20, 2015.

\bibitem{newman2005convergence}
Charles~M Newman, Krishnamurthi Ravishankar, and Rongfeng Sun.
\newblock Convergence of coalescing nonsimple random walks to the brownian web.
\newblock {\em Electronic Journal of Probability}, 10:21--60, 2005.

\bibitem{roy2013random}
Rahul Roy, Kumarjit Saha, and Anish Sarkar.
\newblock Random directed forest and the brownian web.
\newblock {\em Annales de l'Institut Henri Poincar{\'e}, Probabilit{\'e}s et
  Statistiques}, 52(3):1106--1143, 2016.

\bibitem{schertzer2015brownian}
Emmanuel Schertzer, Rongfeng Sun, and Jan~M Swart.
\newblock The brownian web, the brownian net, and their universality.
\newblock In {\em random processes and some applications}, pages 270--368.
  Cambridge university press, 2017.

\end{thebibliography}
\end{document}